\documentclass[english,11pt]{amsart}
\usepackage[T1]{fontenc}
\usepackage[latin9]{inputenc}
\usepackage{amsthm}
\usepackage{amstext}
\usepackage{graphicx}
\usepackage{amssymb}
\usepackage{esint}
\numberwithin{equation}{section} %% Comment out for sequentially-numbered
\numberwithin{figure}{section} %% Comment out for sequentially-numbered
\theoremstyle{plain}
\newtheorem{theorem}{Theorem}
  \theoremstyle{plain}
  
  \theoremstyle{definition}
  \newtheorem{definition}[theorem]{Definition}
  \theoremstyle{plain}
  \newtheorem{lemma}[theorem]{Lemma}
  \theoremstyle{plain}
  
  \theoremstyle{remark}
  \newtheorem*{rem*}{Remark}

\def\AMS{{\it AMS subject classifications: }}
\usepackage{babel}

\usepackage{amsmath,subfigure}
\usepackage{amsbsy}
\usepackage{epsfig}
\usepackage{chemarr}

\usepackage{eepic}
\usepackage{latexsym}
\usepackage{dsfont}

\providecommand{\assumptionname}{Assumption}
\newtheorem{assumption}{\assumptionname}{\bfseries}{\rmfamily}
\newcommand{\vect}[1]{\boldsymbol{#1}}
\newcommand{\tensor}[1]{\boldsymbol{#1}}
\def \ve{\varepsilon}

\def\longrightharpoonup{\relbar\joinrel\rightharpoonup}
\def\longleftharpoondown{\leftharpoondown\joinrel\relbar}

\def\longrightleftharpoons{
  \mathop{
    \vcenter{
      \hbox{
      \ooalign{
        \raise1pt\hbox{$\longrightharpoonup\joinrel$}\crcr
	  \lower1pt\hbox{$\longleftharpoondown\joinrel$}
	  }
      }
    }
  }
}

\numberwithin{equation}{section}

\title[Water transport in plant tissues]{Homogenization approach to water transport in plant tissues with periodic microstructures}
\author{Andr\'es Chavarr\'{\i}a-Krauser, Mariya Ptashnyk}
 
 \thanks{\noindent Center for Modelling and Simulation in the Biosciences \& Interdisciplinary Center for Scientific Computing, Universit\"at Heidelberg, INF 368, 69120 Heidelberg, Germany;   
 Department  of Mathematics, University of Dundee, Old Hawkhill,  Dundee DD1 4HN
Scotland, UK \\
andres.chavarria@bioquant.uni-heidelberg.de, mptashnyk@maths.dundee.ac.uk}

\begin{document}

\maketitle

\begin{abstract}
Water flow in plant tissues takes place in two different physical domains separated by semipermeable membranes: cell insides and cell walls. The assembly of all cell insides and cell walls are termed \emph{symplast} and \emph{apoplast}, respectively. Water transport is pressure driven in both, where osmosis plays an essential role in membrane crossing. In this paper, a microscopic model of water flow and transport of an osmotically active solute in a plant tissue is considered. The model is posed on the scale of a single cell and the tissue is assumed to be composed of periodically distributed cells. The flow in the symplast can be regarded as a viscous Stokes flow, while Darcy's law applies in the porous apoplast. Transmission conditions at the interface (semipermeable membrane) are obtained by balancing the mass fluxes through the interface and by describing the protein mediated transport as a surface reaction. Applying homogenization techniques, macroscopic equations for water and solute transport in a 
plant tissue are derived. The macroscopic problem is given by a Darcy law with a force term proportional to the difference in concentrations of the osmotically active solute in the symplast and apoplast; i.e. the flow is also driven by the local concentration difference and its direction can be different than the one prescribed by the pressure gradient.
\end{abstract}

%*******************************************************************
%KEYWORDS
%*******************************************************************

% these are examples, put your key words

%*******************************************************************
%AMS SUBJECT CLASSIFICATION
%*******************************************************************
\AMS {35B27,  35K61,  74Qxx,  76M50,  76S05,  76D07 }

% these are examples, put your numbers

%*******************************************************************
%ARTICLE BODY
%*******************************************************************

%%%%%%%%%%%%%%%%%%%%%%%%%%%%%%%%%%%%%%%%%%%%%%%%%%%%%%%%%%
\section{Introduction}

Plant tissues are in general composed of two domains separated by selective membranes: apoplast and symplast. The apoplast is composed of cell walls and intercellular spaces, while the symplast is constituted by all protoplasts which can be connected by plasmodesmata. Therefore, the path of water and solutes is threefold: apoplastic, symplastic and transcellular, \cite{steudle:1998,steudle:2000}. A first quantitative model of water transport in plants was proposed by van den Honert, \cite{honert:1948,tyree:1997}. The idea was to describe water flow in analogy to the flow of electric current through a resistor network. This phenomenological approach is still contemporary, \cite{Nobel:1999}, and is also used in engineering to describe water supply networks. Pressure assumes the role of the electric potential and, hence, pressure gradients produce a flux proportional to the hydraulic conductivity. This pressure driven flux was extended to include osmotically driven fluxes (diffusional fluxes) and the concept of 
water potential was introduced, \cite{Nobel:1999}. The general concept adapted from nonequilibrium thermodynamics is that, differences in water potential produce equilibrating forces which drive the water fluxes, \cite{tyree:1969}. This relation will be presented below in more detail. Plant biologists have used this concept to describe water uptake of single cells (Fig. \ref{fig:waterfluxcell}), e.g. during cell expansion \cite{Nobel:1999}.

\begin{figure}[t]
\centering
\includegraphics[width=13cm]{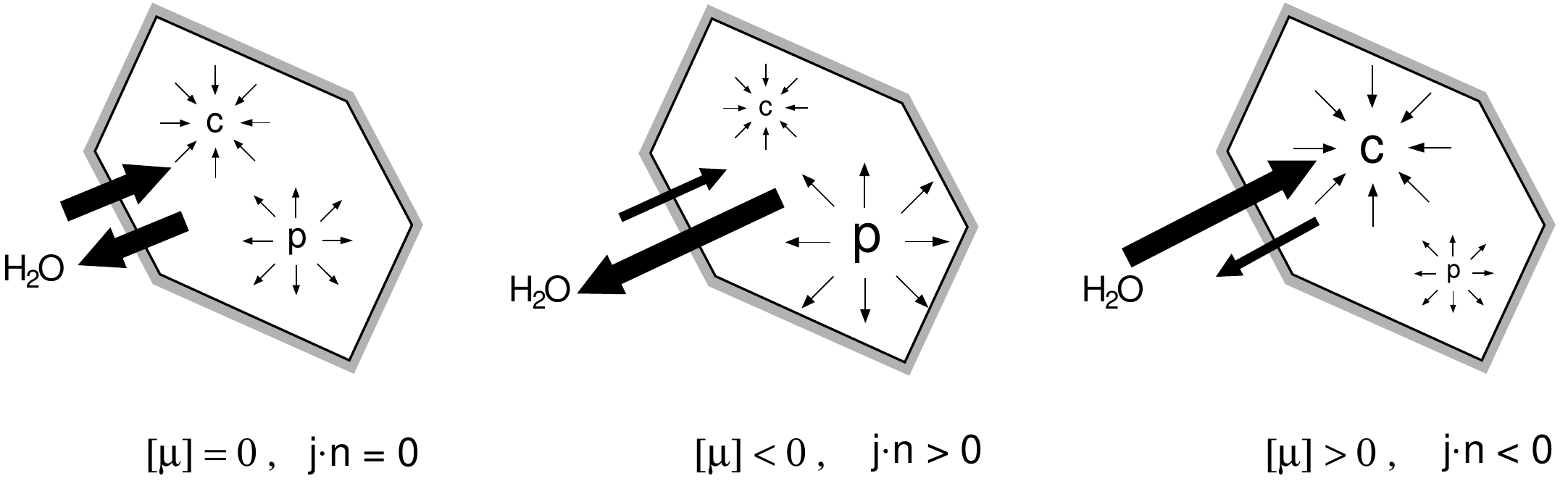}
 \caption{Scheme of water flux into and out of a single cell. The flux is proportional to the jump in chemical potential $[\mu]$, which depends on internal pressure $p$ and solute concentration $c$.}
\label{fig:waterfluxcell}
\end{figure}

Besides van den Honert's approach, less has been undertaken to extend and apply the concept for whole tissues. An interesting question is how the concepts should be used in continuum models of tissues. As will become clear later, the central problem is to find suitable transmission conditions, which describe the fluxes through the plasma membranes, and thus, between the apoplast and symplast. Another interesting task is to obtain simplified models for situations where the cell scale is small compared to the tissue or organ considered. Take for example functional structural plant models, in which often macroscopic sections of organs are combined and simulated, \cite{vos:2010}. Whilst less complexity is the principle of a simplified model, sufficient information on the microstructure should still be part of it. Fortunatelly, plant tissues tend to be sufficiently periodic and periodic homogenization lends to treat the problem, \cite{cioranescu:1999}.\\

\begin{figure}[b]
\centering
\includegraphics[width=8cm]{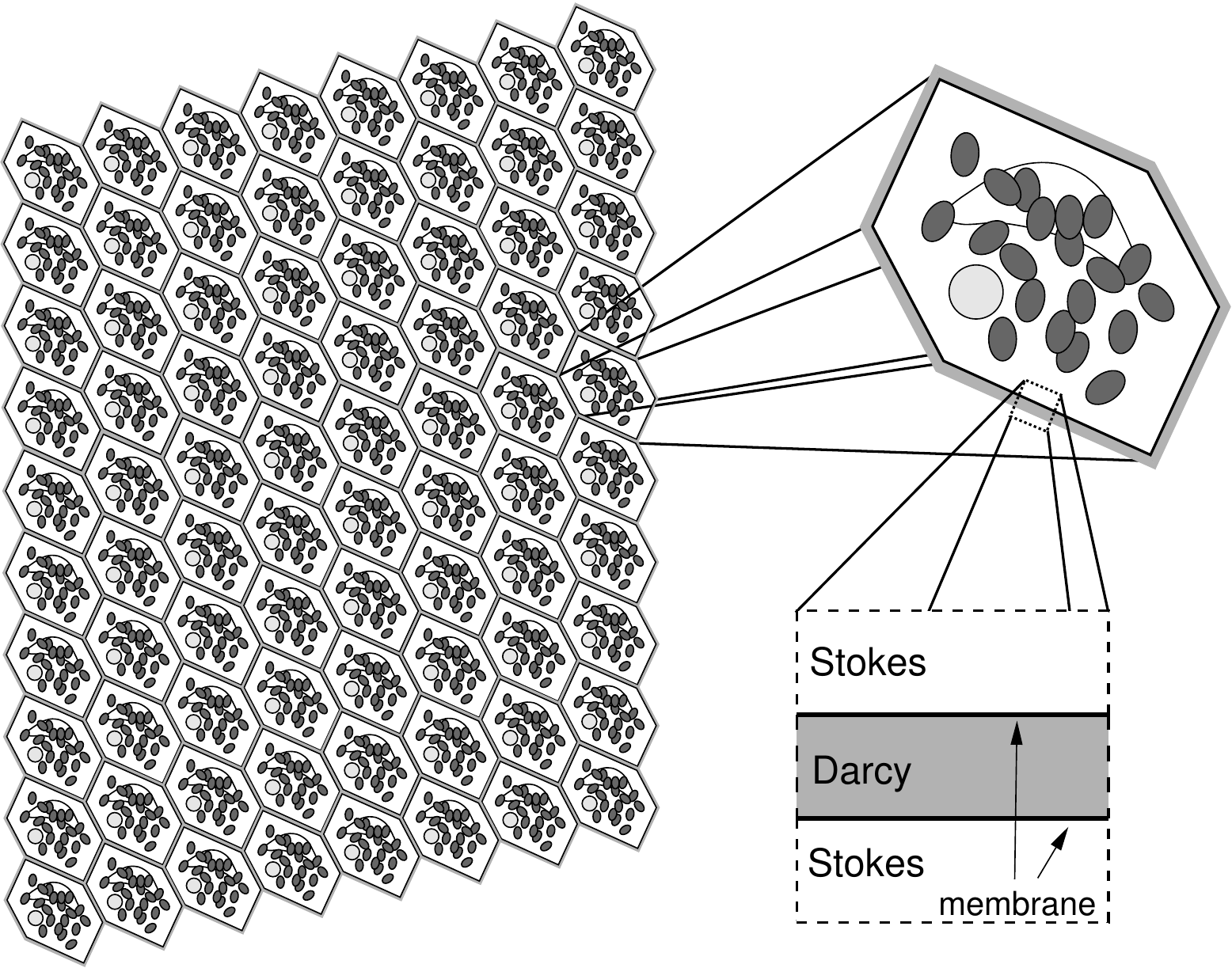}
 \caption{Scheme of a typical periodic plant tissue, for example lamellae of the moss \textit{Plagiomnium affine}.}
\label{fig:plant_tissue}
\end{figure}

%%%%%%%%%%%%%%%%%%%%%%
A model of water and solute fluxes in plant tissues needs a precise description of transport processes between the cell inside (symplast) and the cell wall (apoplast) (Fig. \ref{fig:plant_tissue}). This transition takes place through semipermeable membranes and represents the connection between two different physical domains. These domains are the porous cell wall described best by a Darcy law and the cell cytoplasm in which a viscous flow applies. One of the main mathematical problems is that, in contrast to \cite{Mikelic_2000, Arbogast}, the free fluid and porous media
domains do not  interact directly, as the membrane separates the domains and controls actively and passively the fluxes of water and solutes. Thus the continuity  of normal forces and the  Beavers-Joseph-Saffman transmission condition between free fluid and porous medium do not apply. A similar situation to the one here is found in models of early atherosclerotic lesions \cite{calvez:2010} or low-density lipoproteins transfer through arterial walls, \cite{prosi:2005, sun:2007}. Kedem-Katchalsky transmission conditions are used by those authors to couple the Navier-Stokes equations for blood flow in the arterial lumen with a Darcy law in the arterial wall. However,  the dependence of fluid flow across the membrane on the oncotic pressure difference -- proportional to the jump in lipoproteins concentration -- was neglected by those authors to simplify the analysis of the model equations. Besides, transport of the solute in plant cells is mediated  by proteins and can take place opposite to the gradient in 
chemical potential by usage of energy (e.g. ATPase pumps) and has to be based on a surface reaction mechanism.\\

Coupled  free fluid and porous media problems have received an increasing attention during the last years from the mathematical and the numerical point of view. Well-posedness analysis and numerical algorithms for coupled Stokes-Darcy and Navier-Stokes-Darcy problems with  Beavers-Joseph-Saffman transmission condition between free fluid and porous medium   were addressed in \cite{Quarteroni_2009, Galvis, Layton2003} and references therein. 
Along many results on the homogenization of Stokes and Navier-Stokes problems and derivation of Darcy law, \cite{Allaire1989, hornung, Lipton, Mikelic_2000, Tartar1980}, 
the multiscale analysis for a Stokes-Dracy system modeling water flow in a vuggy porous media with Beavers-Joseph-Saffman transmission condition  was considered in  \cite{Arbogast}. The macroscopic equations were derived using formal asymptotic expansion and two-scale convergence method. A formal asymptotic expansion was also applied successfully to define  the transport velocity of auxin in a plant tissue, \cite{AndresMariya}.

New transmission conditions at the cell-membrane-cell wall interface and the coupling between the flow velocity and solute concentrations via transmission conditions reflecting the osmotic nature of the water flow through a semipermeable membrane distinguish our model from the problem studied in \cite{Arbogast}. Additional technical difficulties are introduced due to the distinction between  symplastic and apoplastic velocities and the presence of plasmodesmata populated cell wall pieces. To show the existence of a unique solution of the microscopic model we apply the abstract theory of mixed problems (saddle-point problems), \cite{Girault,Layton2003}, where the coercivity in the divergence-free space and the inf-sup condition ensure the existence of a unique solution of the coupled Stokes-Darcy model. 
The methods of the two-scale convergence and unfolding operator are used  to derive macroscopic equations for the fluid flow and for the transport of osmotic active solutes.
A generalized Darcy law with a force term given by  the difference of solute concentrations in apoplast and symplast defines the macroscopic water velocity.
Two initial-boundary-value problems for the concentration of solutes in  symplast and in apoplast, respectively,  coupled via ordinary differential equations for the transporter concentrations, describe the dynamics of solute concentrations in a plant tissue.\\

%%%%%%%%%%%%%%%%%%%%%%%%%%%%%%%%%%%%%%%%%

The paper is organized as follows. A thorough introduction to non-advective water fluxes is given. These concepts are then used to derive a biophysical model for transport of  water and osmotically active  solutes through a cell membrane, and to obtain transmission conditions between the symplast and apoplast (Section~\ref{PhysModel}).  Based on this biophysical model, a microscopic model for transport in a plant tissue is formulated (Section~\ref{MicroModel}). Well-posedness and a priori estimates for solutions of the microscopic model are shown (Section~\ref{A_priori_Estim}), followed by derivation of averaged macroscopic equations for water and solute transport defined on the scale of a plant tissue (Section~\ref{MacroModel}). At last, some  results on two-scale convergence and periodic unfolding method are formulated (Appendix).

%%%%%%%%%%%%%%%%%%%%%%%%%%%%%%%%%%%%%%%%%%%%%%%%
\subsection{Water fluxes and chemical potential}

As noted above, water fluxes in cells are known to be driven by gradients in chemical potential, \cite{Nobel:1999}. Biologist use the concept of water potential $\Psi$, which is the chemical potential of water per unit specific volume -- i.e. expressed in pressure units instead of energy per particle. The chemical potential $\mu_i$ of a species $i$ is a measure of how much the internal energy of a thermodynamical system changes when the number of particles of species $i$ is varied, \cite{Landau:5}. In equilibrium, the entropy of the system is maximal and consequently the temperature $T$, pressure $p$ and the chemical potentials $\mu_i$ are constant in abscence of external fields. Consider now a situation in which the system is not in equilibrium. In this case $\mu_i$ will not be constant and diffusion fluxes arise, which drive the system into equilibrium. If the system is near to equilibrium, the diffusion fluxes can be assumed to be proportional to the gradients in chemical potentials $\nabla \mu_i$, \cite{
Landau:6}. Irreversible thermodynamics proposes that for small gradients, the diffusion mass flux density $\vect{j}_i$ of a species $i$ is given by a weighted sum of the gradients of all chemical potentials. 

Diffusion is a spontaneous molecular equilibration process which does not produce bulk flows in closed systems. The so called mass constraint applies when the system is not divided by a membrane, \cite{Giovangigli:1999}
\begin{equation}\label{eq:massconstraint}
\vect{j} = \sum_i \vect{j}_i = 0 \; ,
\end{equation}
which states that diffusion cannot produce a macroscopic movement of the mixture. Condition \eqref{eq:massconstraint} is not valid when the domain is separated by a semipermeable membrane, as such membranes allow the movement of the solvent but not of all solutes. A consequence is that in such a case the system tends to a local equilibrium: entropy reaches a local maximum and removing the membrane would allow a further increase. To account for semipermeable membranes, the concepts were extended to include a so called \emph{reflection coefficient}, \cite{Nobel:1999}. This coefficient is a measure of how much solutes are reflected by the membrane. A value of one means that all solutes are retained, while all solutes cross freely for a value of zero. The reflection coefficient lies in general between $0$ and $1$ for a real membrane, \cite{Nobel:1999}.

For the sake of simplicity, we will consider here and in the sequel the case of a binary mixture of a solvent ($i=1$) and a solute ($i=2$). Following the notation of \cite{Landau:6}, we introduce the mass fraction of the solute as a concentration
\begin{equation}
 c := \rho_2/ \rho \; ,
\end{equation}
where $\rho_i$, for $i=1,2$ is the mass density of species $i$ and $\rho=\rho_1+\rho_2$ the mass density of the mixture. Consequently the ``concentration'' of the solvent is $1-c$. This duality of the concentrations allows to introduce one chemical potential 
\begin{equation}
 \mu= \frac{\mu_2}{m_2} - \frac{\mu_1}{m_1}\; ,
\end{equation}
instead of two, \cite{Landau:6}, where $m_i$ is the mass of one particle of species $i$, for $i=1,2$. Note that $m_1$ and $m_2$ are needed to obtain a chemical potential density in units energy per mixture mass. The flux density of the solute is then $\vect{j}_2=-\alpha \nabla \mu$ for $\alpha > 0$. Using the mass constraint \eqref{eq:massconstraint}, delivers for that case the flux density $\vect{j}_1 = \alpha \nabla \mu$ of the solvent. The combined chemical potential cannot be used in the case of a semipermeable membrane and the approach needs to be extended by a reflection coefficient. For this purpose, we introduce two diffusion driving potentials
\begin{align}\label{eq:chempotentials}
  \begin{aligned}
    \tilde\mu_1 &:=\frac{\mu_1}{m_1} - (1-\varsigma)\frac{\mu_2}{m_2}\; & \text{and} \quad \tilde\mu_2 &:=(1-\varsigma)\left(\frac{\mu_2}{m_2} - \frac{\mu_1}{m_1}\right)\; , \\
    \tilde\mu_1 &=-(1-\varsigma)\mu + \varsigma\frac{\mu_1}{m_1}& \text{and} \quad \tilde\mu_2 &=(1-\varsigma)\mu\; ,
  \end{aligned}
\end{align}
where $0 \leq \varsigma \leq 1$ is a reflection coefficient. Setting $\varsigma=0$ renders $\tilde{\mu}_2 = -\tilde{\mu}_1 = \mu$, which is consistent with \cite{Landau:6}, while setting $\varsigma=1$ gives $\tilde\mu_1=\mu_1/m_1$ and $\tilde\mu_2 = 0$. The corresponding mass flux densities are simply
\begin{align}\label{eq:totalmassfluxgeneral}
  \begin{aligned}
    \vect{j}_i & =- \alpha\, \nabla \tilde\mu_i\; , &&\text{for } i=1,2\; ,\\
    \vect{j} & =- \alpha\, \nabla(\tilde{\mu}_1+ \tilde{\mu}_2) =- \alpha\, \varsigma\, \nabla \mu_1\; , \\
  \end{aligned}
\end{align}
where $\alpha > 0$ is a coefficient related to the permeability towards the solvent. The expression for $\vect{j}$ shows that reflection ($\varsigma \neq 0$) is a must to produce bulk fluxes, and is actually the mechanism exploited by plant cells, \cite{Nobel:1999}. Fig. \ref{fig:waterfluxcell} presents a scheme of water fluxes through a membrane of a single cell with a semipermeable membrane ($\varsigma=1$, $\tilde\mu_2 = 0$ and $\vect{j}_2=0$).\\

Often in mathematical models, diffusion is assumed to be driven only by concentration gradients, which is equivalent to setting $\tilde{\mu}_2 = \tilde{\mu}_2(c)$. For this case, Eqs. \eqref{eq:totalmassfluxgeneral} deliver Fick's law with reflection, \cite{Landau:6}
\begin{equation}\label{eq:fick}
\vect{j}_2 \approx -\rho (1-\varsigma)D\, \nabla c \; ,
\end{equation}
where $D > 0$ is the usual diffusion coefficient [see Eq. \eqref{eq:coefficients}]. This approach ignores that the chemical potential depends also on pressure $p$ and temperature $T$, which implies that gradients in $p$ and $T$ produce also fluxes. This dependency is normally assumed to be small, although the pressure term is known to be important in sedimentation processes, \cite{Vailati:1998}, and in cases where the pressure is nonharmonic and concentration gradients are small, \cite{Chavarria:2010}. Fick's law \eqref{eq:fick} is a fairly good approximation in a homogeneous domain, but is inapplicable for obtaining transmission conditions at semipermeable interfaces. Large pressure differences can arise across a membrane and the contributions of pressure and concentration driven diffusion are of comparable magnitude.\\

%%%%%%%%%%%%%%%%%%%%%%%%%%%%%%%%%%%%%%%%%%%%

To account for pressure driven diffusion, the potentials $\tilde\mu_i$ are set to depend on concentration and pressure. For simplicity, the dependence on temperature is neglected, as large gradients are not usual in plant tissues. Eqs. \eqref{eq:totalmassfluxgeneral} render
\begin{align}\label{eq:generalfluxgradthree}
  \begin{aligned}
    \vect{j}_2 & = -\rho\,(1-\varsigma) \big(D\,\nabla c + G\, \nabla p\big)\; ,\\
    \vect{j} & = \rho\, \varsigma \big(\mathcal{D}\,\nabla c - \mathcal{G}\, \nabla p \big) \qquad \text{and}\\
    \vect{j}_1 & = -\vect{j}_2 + \vect{j}\; ,
    \end{aligned}
\end{align}
where
\begin{align}\label{eq:coefficients}
D & = \frac{\alpha}{\rho}\frac{\partial \mu}{\partial c} > 0\; , & G & = \frac{\alpha}{\rho}\frac{\partial \mu}{\partial p}\; , & \mathcal{D} & = -\frac{\alpha}{\rho}\frac{\partial \mu_1}{\partial c} > 0\; , & \mathcal{G} & = \frac{\alpha}{\rho}\frac{\partial \mu_1}{\partial p}>0 \; .
\end{align}
The diffusion coefficients $D$ and $\mathcal{D}$ are positive, while the barodiffusion coefficient $G$ has no definite sign, \cite{Landau:6}. Assuming incompressibility of the solvent has as a consequence that $\mathcal{G}$ is positive, \cite{Nobel:1999}. These signs concord with what is known from biology: diffusion fluxes across the mebrane follow the concentration gradient and are oriented against the pressure gradient (see Fig. \ref{fig:waterfluxcell}).

%%%%%%%%%%%%%%%%%%%%%%%%%%%%%%%
%%%%%%%%%%%%%%%%%%%%%%%%%%%%%%%
\section{Biophysical model}\label{PhysModel}
%%%%%%%%%%%%%%%%%%%%%%%%%%%%%%
\subsection{Mass conservation}
The total flux density of a species and the mixture is given by the combination of the contributions of advection $\rho_i\, \vect{v}$, diffusion $\vect{j}_i$ and membrane transport via transporting proteins $\vect{a}_i$
\begin{align}
  \begin{aligned}
    \vect{J}_i & = \rho_i\, \vect{v} + \vect{j}_i + \vect{a}_i\; , \qquad i=1,2 \; ,\\
    \vect{J} & = \rho\, \vect{v} + \vect{j} + \vect{a}\; , 
  \end{aligned}
\end{align}
where $\vect{j} := \sum_i \vect{j}_i$ and $\vect{a}:=\sum_i \vect{a}_i$. Consequently, conservation of the species and total mass are given by
\begin{subequations}
\begin{align}
 \partial_t \rho_i + \operatorname{div}(\rho_i\, \vect{v} + \vect{j}_i + \vect{a}_i ) & = 0 \; , \qquad i=1,2 \; ,\label{eq:speciesconservation}\\
 \partial_t \rho + \operatorname{div}(\rho\, \vect{v} + \vect{j} + \vect{a} ) & = 0 \; . \label{eq:totalconservation}
\end{align}
\end{subequations}

First, consider the laws in a compartment without an interface. Transport via proteins takes place only on membranes. Therefore, $\vect{a}_i$ are zero inside a domain $\Omega$ that does not contain a membrane, and the mass constraint \eqref{eq:massconstraint} applies. Conservation of solute and mass in $\Omega$ follows thus
\begin{subequations}
\begin{align}
 \partial_t \rho_i + \operatorname{div}\big(\rho_i\, \vect{v} + \vect{j}_i\big) & = 0 \qquad \text{in}\ \Omega \times (0,\infty)\; , \qquad i=1,2 \; ,\label{eq:speciesconservationgeneral}\\
 \partial_t \rho + \operatorname{div}(\rho\, \vect{v}) & = 0 \qquad \text{in}\ \Omega \times (0,\infty)\; . \label{eq:totalconservationgeneral}
\end{align}
\end{subequations}

\begin{figure}[t]
\centering
\includegraphics[width=6cm]{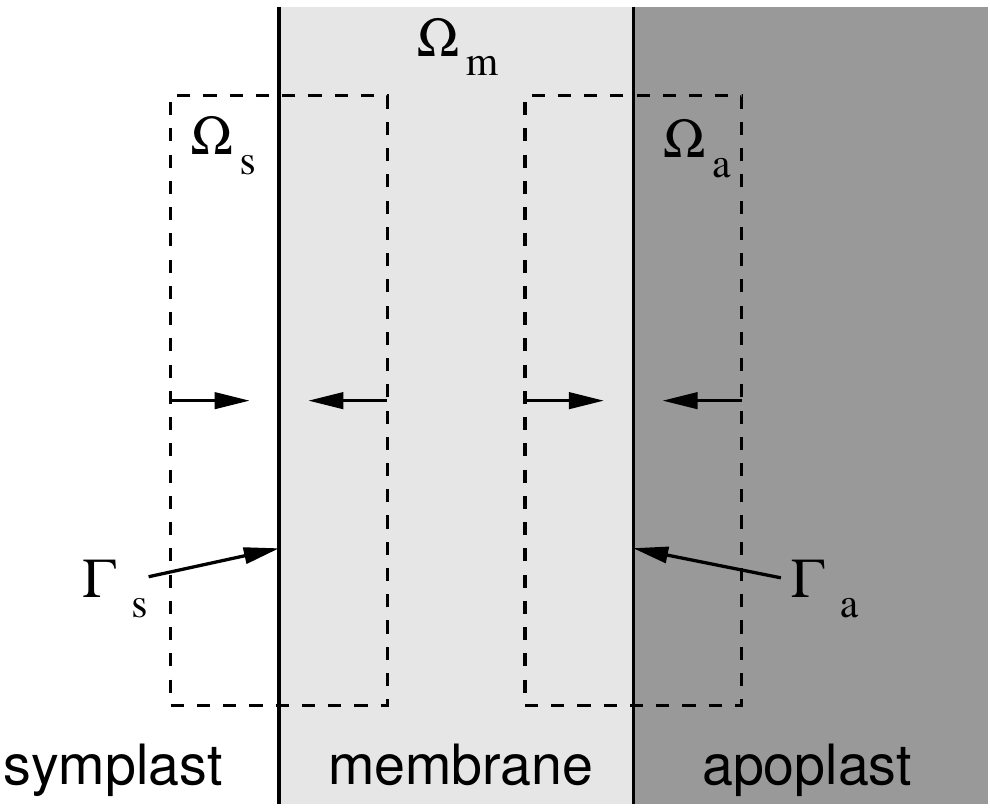}
 \caption{Scheme of the method used to obtain the transmission conditions. Small domains $\Omega_s$ and $\Omega_a$ enclosing the interface symplast-membrane and membrane-apoplast, respectively, are reduced to the surfaces $\Gamma_s$ and $\Gamma_a$. The membrane is represented by the domain $\Omega_m$.}
\label{fig:transmission}
\end{figure}

Consider now the case where a membrane is present. Shall $\Omega_s$ be a domain enclosing an arbitrary piece $\Gamma_s$ of the interface between the symplast and the membrane (Fig. \ref{fig:transmission}). Shall the membrane be represented by the domain $\Omega_m$. Note that the derivation of the conditions for the interface between membrane and apoplast is conducted in the same manner and will not be given here explicitly. Because of the membrane, $\varsigma \neq 0$ and $\vect{a}_i \neq 0$ for $i=1,2$. Integration of \eqref{eq:speciesconservation} over $\Omega_s$ and application of Gau\ss' law gives
\[
 \partial_t \int_{\Omega_s} \rho_i\, dx + \int_{\partial \Omega_s} ( \rho_i\, \vect{v} + \vect{j}_i + \vect{a}_i ) \cdot \vect{n}\, d\gamma = 0 \; .
\]
The membrane can be assumed to not allow advective fluxes, so that on $\partial \Omega_s \cap \Omega_m$ we have $\vect{v} = 0$. On $\partial \Omega_s \backslash \overline{\Omega}_m$ there is no protein mediated transport so that $\vect{a} = 0$ there. The thickness of $\Omega_s$ is reduced to zero such that the interface is kept inside ($\Omega_s \to \Gamma_s$). The first integral tends to zero, while the second tends to an integral over $\Gamma_s$
\[
 \int_{\Gamma_s} \big( - \rho_i\, \vect{v} + \overline{\vect{j}}_i - \vect{j}_i + \vect{a}_i \big) \cdot \vect{n}\, d\gamma = 0 \; ,
\]
where $\overline{\vect{j}}_i$ is the diffusion flux in $\Omega_m$. Note that $\vect{n}$ points from the symplast to the apoplast. Because $\Omega_s$ and $\Gamma_s$ were chosen arbitrarily, the integrand has to be zero. A similar approach can be applied to the mass flux of the mixture, where $\vect{j}=0$ in $\Omega$ is used. We obtain the conditions
\begin{align}
\begin{aligned}
 (\rho_i\, \vect{v} + \vect{j}_i ) \cdot \vect{n} & =  (\vect{a}_i + \overline{\vect{j}}_i)\cdot \vect{n} \qquad \text{on}\ (\Gamma_s \cup \Gamma_a) \times (0,\infty)\; ,\\
 \rho\, \vect{v} \cdot \vect{n} & =  (\vect{a} + \overline{\vect{j}})\cdot \vect{n}\ \ \qquad \text{on}\ (\Gamma_s \cup \Gamma_a) \times (0,\infty)\; ,
\end{aligned}
\end{align}
where $\overline{\vect{j}} = \overline{\vect{j}}_1 + \overline{\vect{j}}_2$. The fluxes in $\Omega$ were written on the left hand side, while the fluxes in $\Omega_m$ are on right hand side. The diffusion fluxes $\overline{\vect{j}}_i$ are normally assumed to be constant in the membrane (constant gradients), \cite{Nobel:1999}. Protein mediated transport can be considered to be a chemical reaction allowing buffering of species, and hence, the fluxes $\vect{a}_{i,\sf I}$ and $\vect{a}_{i,\sf II}$ through $\Gamma_s$ and $\Gamma_a$, respectively, have not necessarily to be equal. Expressions for $\vect{a}_{i,\sf I}$ and $\vect{a}_{i,\sf II}$ will be developed in Sec. \ref{sec:protein_transport}.

The membrane $\Omega_m$ is very thin compared to the rest and is usually reduced to an interface denoted here as $\Gamma_m$. An expression for the membrane's buffering effect is obtained by integration of \eqref{eq:speciesconservation} over $\Omega_m$, application of Gau\ss' law and reduction of $\Omega_m$ to $\Gamma_m$. This process is equivalent to subtraction of the transmission conditions on $\Gamma_s$ from those on $\Gamma_a$. Finally, the transmission conditions on the reduced membrane interface $\Gamma_m$ are in total
\begin{subequations}
\begin{align}
 (\rho_i\, \vect{v} + \vect{j}_i ) \cdot \vect{n} & =  (\vect{a}_{i,\sf I} + \overline{\vect{j}}_i)\cdot \vect{n} \qquad \text{on}\ \Gamma_m \times (0,\infty)\; ,\qquad i=1,2\; , \label{eq:speciestransmission}\\
 \rho\, \vect{v} \cdot \vect{n} & =  (\vect{a}_{\sf I} + \overline{\vect{j}})\cdot \vect{n}\quad\!\qquad \text{on}\ \Gamma_m \times (0,\infty)\; , \label{eq:totalmasstransmission}\\
 [\rho_i\, \vect{v} + \vect{j}_i] \cdot \vect{n} & =  [\vect{a}_i] \cdot \vect{n}\ \ \qquad\qquad \text{on}\ \Gamma_m \times (0,\infty)\; ,\qquad i=1,2\; ,\label{eq:speciesmembraneconservation}\\
 [\rho\, \vect{v}] \cdot \vect{n} & =  [\vect{a}]\cdot \vect{n} \quad\qquad\qquad \text{on}\ \Gamma_m \times (0,\infty)\; , \label{eq:totalmassmembraneconservation}
\end{align}
\end{subequations}
where the assumptions $\overline{\vect{j}}_i=\text{const}$, and $\overline{\vect{j}}=\text{const}$ in the membrane were used, and $[\cdot]$ is defined as the jump across the membrane (i.e. $[\vect{a}_i] = \vect{a}_{i,\sf II} - \vect{a}_{i,\sf I}$).

%%%%%%%%%%%%%%%%%%%%%%
\subsection{Protein mediated transport}\label{sec:protein_transport}
Most solute fluxes through membranes are mediated by transporting proteins, \cite{Nobel:1999}. Transport can be either passive or active with usage of energy (e.g. ATPase pumps). We model this processes as surface reactions, in which the solute on one side ($S_{\sf I}$) reacts with the transporter, builds a complex ($S\!-\!T$), which decays by transporting the solute to the other side ($S_{\sf II}$)
\[
\begin{split}
S_{\sf I} + T &\longrightleftharpoons^{k_1}_{k_2} S\!-\!T \longrightleftharpoons^{k_3}_{k_4} S_{\sf II} + T \; .
\end{split}
\]
\begin{figure}[t]
\centering
\includegraphics[width=5cm]{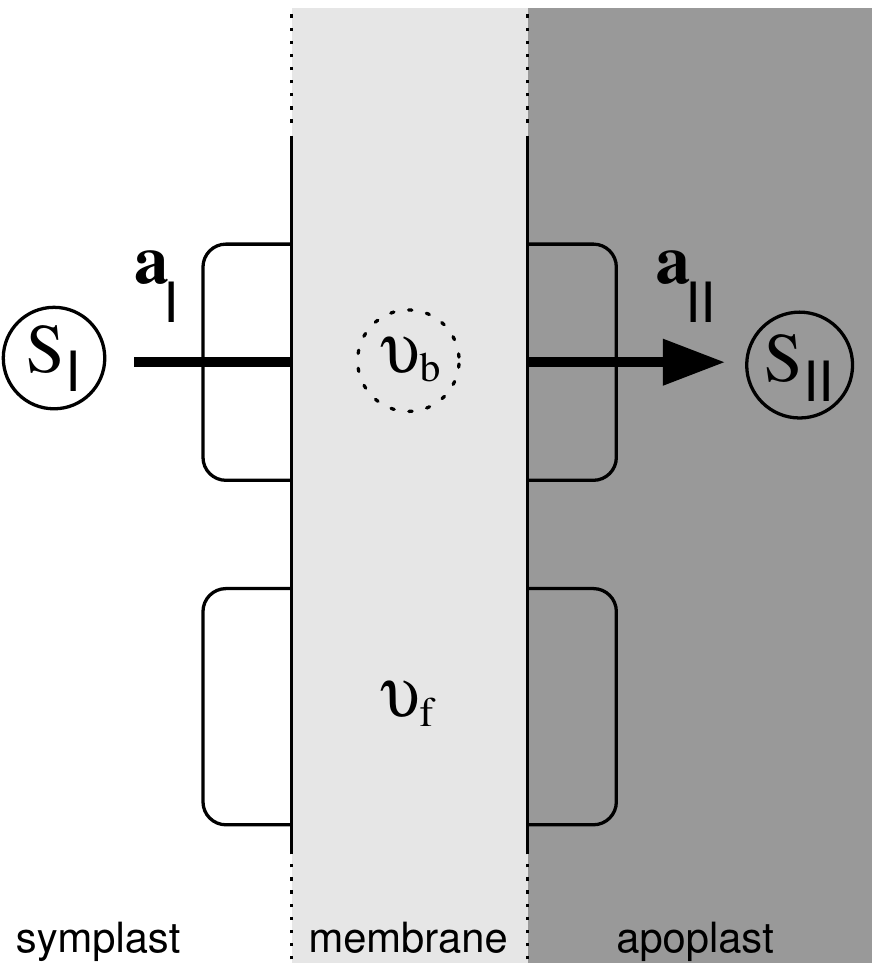}
 \caption{Depiction of the mechanism assumed for protein mediated transport, where a transport direction from symplast to apoplast is shown exemplarily.}
\label{fig:transporters}
\end{figure}
The concentrations of the solute on either sides ($S_{\sf I}$ and $S_{\sf II}$) are denoted by $c_{\sf I}$ and $c_{\sf II}$, respectively. The concentrations of free and bound transporters are denoted with $\vartheta_f$ and $\vartheta_b$, respectively. In general, the density of the mixture can differ between sides and we denote these here with $\rho_{\sf I}$ and $\rho_{\sf II}$. The above simple reaction mechanism produces fluxes on each side
\begin{align}
  \begin{aligned}
    \vect{a}_{\sf I} & = (k_1\, \rho_{\sf I}\, c_{\sf I}\, \vartheta_f  - k_2\, \vartheta_b)\, \vect{n} \; , \\
    \vect{a}_{\sf II} & =(k_3\, \vartheta_b - k_4\, \rho_{\sf II}\,c_{\sf II}\,\vartheta_f)\, \vect{n} \; ,
  \end{aligned}
\end{align}
where $\vect{n}$ is the normal of the membrane pointing out of the cell. These expressions correspond to the rate law of the above reaction mechanism. See Fig. \ref{fig:transporters} for a depiction of this mechanism.

The concentration of free and bound transporters are obtained by pointwise reactions. The rate of change of free transporters due to binding of solute and decay of the complex is given by
\begin{equation}\label{eq:differential_flux}
 [\vect{a}] \cdot \vect{n} = (\vect{a}_{\sf II} - \vect{a}_{\sf I})\cdot \vect{n} = -(k_1\,\rho_{\sf I}\, c_{\sf I} + k_4\,\rho_{\sf II}\,c_{\sf II})\,\vartheta_f + (k_2+k_3)\,\vartheta_b  \; .
\end{equation}
The net transport direction varies in time and depends on the solute concentrations and coefficients $k_1$ to $k_4$. By focusing on a quasi-stationary situation in which $\vect{a}_{\sf II} = \vect{a}_{\sf I}:=\vect{a}$, the following net flux and concentration of free transporters are obtained
\begin{align}
  \begin{aligned}
    \vect{a} & = {(k_1\,k_3\, \rho_{\sf I}\, c_{\sf I} - k_2\,k_4\, \rho_{\sf II}\, c_{\sf II})}/{(k_2+k_3)}\ \vartheta_f\, \vect{n}\; , \\
    \vartheta_f & = {(k_2+k_3)}/{(k_2+k_3 + k_1\, \rho_{\sf I}\, c_{\sf I} + k_4\, \rho_{\sf II}\, c_{\sf II})}\  \vartheta_0\; ,
  \end{aligned}
\end{align}
where $\vartheta_0(t,x) = \vartheta_f(t,x) + \vartheta_b(t,x)$ is the total amount of transporters. We will focus here on the special case where either $k_1=0$ or $k_4 =0$, i.e. on the case of a perfect uni-directional transporter. Using the assumption of quasi stationarity, the mechanism is found to follow Michaelis-Menten kinetics. Exemplary for an efflux transporter ($k_4=0$)
\begin{align}
 \vect{a} & = \rho_{\sf I}\, c_{\sf I}\, k_1\,k_3/(k_2+k_3 + k_1\, \rho_{\sf I}\, c_{\sf I})\ \vect{n}\; .
\end{align}

The applicability of the assumption of quasi stationarity is limited. Transporters are in general subjected to regulation, which reduces or increases the amount of transporters depending on the status of the system. The total amount of transporters varies as follows
\begin{equation}
 \partial_t \vartheta_0 = R(t,x,\vartheta_0) \; ,
\end{equation}
where $R(t,x,\vartheta_0)$ is a function describing the rate of regulation. $R$ is in general composed of a positive production term and a negative degradation term. Negative degradation is often assumed to be proportional to $\vartheta_0$ to prevent negative concentrations. The production term will depend on the local or global solute concentration and regulation is used to control that concentration. A general analysis of robust homeostatic control of a species based on concentrations is presented in, \cite{Ni:2009}, and an example for zinc homeostasis in yeast and plant cells is found in \cite{Claus:2012,Claus:2012b, Claus:2013}. We will not consider details of regulation here.

Assume that only free transporters are regulated and by using Eq. \eqref{eq:differential_flux}, the following system describing the pointwise dynamics of the free and bound transporters is obtained
\begin{align}\label{eq:regulation}
  \begin{aligned}
    \partial_t \vartheta_f & = R(t,x,\vartheta_f) - (k_1\,\rho_{\sf I}\, c_{\sf I} + k_4\,\rho_{\sf II}\,c_{\sf II})\,\vartheta_f + (k_2+k_3)\,\vartheta_b  \; ,\\
    \partial_t \vartheta_b & = \qquad\qquad\qquad\!\! (k_1\,\rho_{\sf I}\, c_{\sf I} + k_4\,\rho_{\sf II}\,c_{\sf II})\,\vartheta_f - (k_2+k_3)\,\vartheta_b\; , \\
    & \qquad\qquad \vartheta_0(t,x)  = \vartheta_f(t,x) +  \vartheta_b(t,x) \; . 
  \end{aligned}
\end{align}

In a general situation several influx and efflux transporters might exist. The fluxes generalize then into
\begin{equation}\label{eq:activeflux}
 \vect{a}_{i,{\sf I}} = \sum_\alpha(k_{1,\alpha}^i\, \rho_{i,{\sf I}}\, \vartheta_{f,\alpha}^i  - k_{2,\alpha}^i\, \vartheta_{b,\alpha}^i)\, \vect{n}\; , \quad \text{and} \quad \vect{a}_{i,{\sf II}}  =\sum_\alpha(k_{3,\alpha}^i\, \vartheta_{b,\alpha}^i - k_{4,\alpha}^i\, \rho_{i,{\sf II}}\,\vartheta_{f,\alpha}^i)\, \vect{n} \; ,
\end{equation}
where $i=1,2$ and all  $\vartheta_{f,\alpha}^i$ and $\vartheta_{b,\alpha}^i$ fulfill an equation system equivalent to Eq. \eqref{eq:regulation}.

%%%%%%%%%%%%%%%%%%%%%%
\subsection{Model}
Inserting the diffusion fluxes of Eq. \eqref{eq:generalfluxgradthree} into Eqs. \eqref{eq:speciesconservationgeneral} and  \eqref{eq:totalconservationgeneral}, and using $\rho_2 = \rho c$ delivers a system describing species and mass conservation in a compartment $\Omega$ with $\varsigma=0$
\begin{align}
  \begin{aligned}
    \partial_t (\rho\, c) + \operatorname{div}\big(\rho\, c\, \vect{v} - \rho\,D\,\nabla c - \rho\,G\, \nabla p\big) & = 0 \qquad \text{in}\ \Omega \times (0,\infty) \; ,\\
    \partial_t \rho + \operatorname{div}(\rho\, \vect{v}) & = 0 \qquad \text{in}\ \Omega \times (0,\infty)\; .
  \end{aligned}
\end{align}
The gradients in the membrane can be assumed to be constant and can be expressed by a jump across the membrane. Therefore, the diffusion fluxes in the membrane $\overline{\vect{j}}_2$ and $\overline{\vect{j}}$ are functions of the jumps in concentration and pressure 
\begin{equation}\label{eq:kedem}
 \overline{\vect{j}}_2([c],[p]) = -\frac{\rho}{h}\big( D\, [c] + G\, [p]\big)\, \vect{n}\; , \quad \text{and} \quad \overline{\vect{j}}([c],[p]) = \frac{\rho}{h}\varsigma\big( \mathcal{D}\, [c] - \mathcal{G}\, [p]\big)\, \vect{n} \; ,
\end{equation}
where the gradients in Eqs. \eqref{eq:generalfluxgradthree} were approximated by a jump across the membrane of thickness $h$. The expressions in Eq. \eqref{eq:kedem} are related to the Kedem-Katchalsky equations, \cite{calvez:2010}.
  
Using Eq. \eqref{eq:activeflux} in \eqref{eq:speciestransmission} and \eqref{eq:speciesmembraneconservation}, the transmission conditions on the interface $\Gamma_m$ for the solute are obtained
\begin{subequations}
\begin{align}
  \big(\rho\, c\, \vect{v} - \rho\, D\,\nabla c - \rho\,G\, \nabla p \big) \cdot \vect{n} & = \big(\vect{a}_{2,\sf I} +\overline{\vect{j}}_2([c],[p])\big) \cdot \vect{n} \qquad \text{on } \Gamma_m \times (0,\infty)\; ,\label{eq:trans_symplast}\\
  \big[\rho\, c\, \vect{v} - \rho\, D\,\nabla c + \rho\,G\, \nabla p\big] \cdot \vect{n} & = [\vect{a}_{2}]\cdot \vect{n}\ \qquad\qquad\qquad\qquad \text{on } \Gamma_m \times (0,\infty) \; . \label{eq:trans_apoplast}
\end{align}
\end{subequations}

Water is almost incompressible ($\rho_1 \approx \textrm{const}$) and compressibility arises here due to the solute influencing the mixture's density $\rho = \rho_1 + \rho_2 = (1-c)\, \rho + c\, \rho$. Using the assumption of a small concentration $c$, in a first approximation $\rho \approx \rho_1 \approx \textrm{const}$. Also, barodiffusion in $\Omega$ can be assumed to be small compared to the concentration driven flux, i.e. it is assumed to play only a role in transport through the membrane. Viscosity can be assumed to dominate the flow due to the small characteristic scale of a few microns, and the flow in the symplast (cell inside) is assumed to be a Stokes flow. The apoplast (cell wall) is porous suggesting a Darcy flow. The equations describing flow, mass and species conservation are hence
\begin{subequations}
\begin{align}
\partial_t \vect{v} - \frac{\eta}{\rho_1} \Delta\vect{v} + \frac{1}{\rho_1}\nabla p & =0 \qquad\text{in symplast} \; ,\label{eq:stokes}\\
\vect{v} + K\, \nabla p & =0  \qquad \text{in apoplast} \; , \label{eq:darcy}\\
\operatorname{div} \vect{v} & = 0 \qquad \text{in symplast and apoplast} \; , \label{eq:incompressibility}\\
\partial_t c  + \operatorname{div}\left( \vect{v}\, c - D\,\nabla c\right) & = 0 \qquad \text{in symplast and apoplast} \; , \label{eq:conservation}
\end{align}
\end{subequations}
where $\eta>0$ is the dynamic viscosity and $K>0$ is the permeability of the apoplast.

Diffusive permeation of the solute can be assumed to be small compared against protein mediated transport. We will also assume that only the solute is subjected to transport via proteins. Assuming that the structure in the membrane is oriented normally, the velocity at the interface can be assumed to be perpendicular to it ($\vect{v} \times \vect{n}=0$). Together with incompressibility, this condition has as a consequence that the normal viscous stress $(\tensor{\sigma}' \vect{n}) \cdot \vect{n}$ is zero on a flat boundary, i.e. the boundary experiences only shear stress. The viscous stress tensor of an incompressible viscous fluid is defined as
\begin{equation}
  \tensor{\sigma}' = 2\, \eta \, \operatorname{S}\!\vect{v} \; ,
\end{equation}
with $\operatorname{S}\!\vect{v}$ the symmetric velocity gradient
\begin{equation}
 \operatorname{S}\!\vect{v}  = \frac{1}{2}\left(\nabla \vect{v} + \nabla \vect{v}^\textsf{T} \right) \; ,
\end{equation}
and is related to the full stress tensor $\tensor{\sigma} = -p\, \tensor{I} + \tensor{\sigma}' = -p\, \tensor{I} + 2\eta\operatorname{S}\!\vect{v}$. Assuming that the effect of corners is small, it is possible to set $(\tensor{\sigma}' \vect{n}) \cdot \vect{n} = 0$ on $\Gamma_m \times (0,\infty)$. Water fluxes mediated by active water transporters are probably small compared to the fluxes via passive transporters (aquaporins): $\|\vect{a}_{\sf I}\| \ll \|\overline{\vect{j}}\|$. Moreover, by adding $(\tensor{\sigma}' \vect{n}) \cdot \vect{n} = 0$ only on the Stokes flow side, Eq. \eqref{eq:totalmasstransmission} becomes
\begin{equation}\label{eg:stress_condition}
 \kappa \big((\tensor{\sigma}'\vect{n}) \cdot \vect{n} + [p] \big) = \delta [c]  - \vect{v} \cdot \vect{n}  \qquad \textrm{on } \Gamma_m \times (0,\infty) \; ,
\end{equation}
where $\kappa := \varsigma\mathcal{G}/h$, $\delta:=\varsigma\mathcal{D} / h$. Note that the normal and the jump $[\cdot]$ are oriented from the Stokes to the Darcy side. The physical meaning of this condition becomes clear by setting $\kappa=\delta=0$
\[
 \vect{v} \cdot \vect{n} = 0 \quad \textrm{for } \kappa=\delta=0 \qquad \textrm{on } \Gamma_m \times (0,\infty) \; ,
\]
i.e. the net mass flux is zero for zero permeability ($\alpha=0$) or for a nonreflecting membrane ($\varsigma =0$). Also, a higher concentration in the cell, $[c]<0$, means that water is driven into the cell, $\vect{v}\cdot \vect{n}< 0$ for $[p]=0$, while a higher pressure in the cell, $[p]<0$, results in water flowing out of the cell, $\vect{v}\cdot \vect{n}> 0$ for $[c]=0$. Compare Fig. \ref{fig:waterfluxcell}.\\

\noindent
Finally, the following conditions are good approximations on the interface $\Gamma_m$
\begin{subequations}
\begin{align}
  \rho_1\,(c_{\sf I}\, \vect{v}_{\sf I} - D_2\,\nabla c_{\sf I} ) \cdot \vect{n} & \approx \vect{a}_{2,\sf I} \cdot \vect{n}&
  &\text{on } \Gamma_m \times (0,\infty)\; ,\label{eq:tc_one}\\
  \rho_1\,[c\, \vect{v} - D_2\,\nabla c ] \cdot \vect{n} & \approx [\vect{a}_{2}] \cdot \vect{n}&
  &\text{on } \Gamma_m \times (0,\infty)\; ,\label{eq:tc_two}\\
  \kappa \big(2\, \eta\,(\operatorname{S}\!\vect{v}\, \vect{n})\cdot \vect{n} +[p] \big) & \approx  \delta [c] - \vect{v}_{\sf I} \cdot \vect{n}&
  &\text{on } \Gamma_m \times (0,\infty) \; , \label{eq:tc_three}\\
  [\vect{v}] \cdot \vect{n} &\approx 0&
  &\text{on } \Gamma_m \times (0,\infty)\; , \label{eq:tc_four}\\
  \vect{v} \times \vect{n} & = 0&
  &\text{on } \Gamma_m \times (0,\infty) \; . \label{eq:tc_five}
\end{align}
\end{subequations}

%%%%%%%%%%%%%%%%%%%%%%%%%%%%%%%%%%%%%%%%%%%%%%%%%%%%%%%
%%%%%%%%%%%%%%%%%%%%%%%%%%%%%%%%%%%%%%%%%%%%%%%%%%%%%%%

\section{Mathematical formulation of  microscopic model}\label{MicroModel}
Let  $\Omega$ be a cube in $\mathbb R^3$ representing a plant tissue, and $\ve>0$ be a  parameter denoting the ratio between the size of a single cell and the size of the considered plant tissue $\Omega$. The microscopic structure of a plant tissue  is reflected in the difference between  the cell wall $\Omega^\ve_{a}\subset \Omega$ and the symplast inside the cells
$\Omega_z^\ve\subset \Omega$.
In the cell wall domain we shall distinguish between the cell wall apoplast  $\Omega_{aw}^\ve$ and parts of the cell wall $\Omega_{as}^\ve$ occupied by both plasmodesmata, that belong to cell symplast, and cell wall apoplast. This partition is a strong simplification of the true geometry of a plant tissue, however, it accounts for the basic structures. See Ref. \cite{Esau} for more anatomical details of plant cells and tissues.

We define  a unit cell $Y=\overline Y_{z}\cup \overline Y_a $,  where 
$Y_z$  is an open domain with a smooth boundary,    representing the part occupied by  symplast inside a cell, 
 and $Y_a$ is the cell wall,   with $\overline Y_a= \overline Y_{aw}\cup \overline Y_{as}$ and $Y_{as}$ is a domain comprising    both plasmodesmata  and cell wall apoplast.   We define also $\overline Y_s= \overline Y_z\cup \overline Y_{as}$. 
The corresponding boundaries are denoted by  $\Gamma_{z}=\partial Y_{aw}\cap \partial Y_z$, $\Gamma_{aw}=\partial Y_{aw}\cap \partial Y_{as}$, $\Gamma_a=\partial Y_a\setminus \partial Y$, $\Gamma_s=\partial Y_s\setminus \partial Y$,
$\Gamma_{as}= \partial Y_{as}\cap \partial Y_{z}$, 
$\Gamma_{zs}= \Gamma_{z} \cup \Gamma_{as}=\partial Y_z$.
Then considering translations of the unit cell given by 
$Y^k_j= Y_j + k$ for  $j=a, z, s, aw$ or $as$, and  $\Gamma^k_i= \Gamma_i + k$
for $i=z, aw, as$ or $zs$, with $k\in \mathbb{Z}^3$,
we can define the domains comprising  the microstructure of a plant tissue  $\Omega^\ve_j= \cup \{ \ve Y_j^k :\, \ve Y^k\subset \Omega, k\in \mathbb{Z}^3 \}$, \, $j=a, z, s, aw$ or $as$.
The microscopic  boundaries are  given by $\Gamma^\ve_i = \cup \{ \ve \Gamma_{i}^k :\, \ve Y^k\subset \Omega, k\in \mathbb{Z}^3 \}$, $i=z, aw, as$, or $ zs$. 
We notice that $\Omega^\ve_s$ and $\Omega_a^\ve$ are connected Lipschitz  domains. 
 The domain  $\Omega^\ve_a$  represent  porous medium of  cell walls, whereas 
 $\Omega^\ve_{sp}$ is introduced to depict   parts of  cell walls, modelled also as a  porous medium,  comprising both cell wall apoplast and many thin channels   of plasmodesmata. 
\begin{figure}[ht!]
\centering
\includegraphics[width=5.4cm]{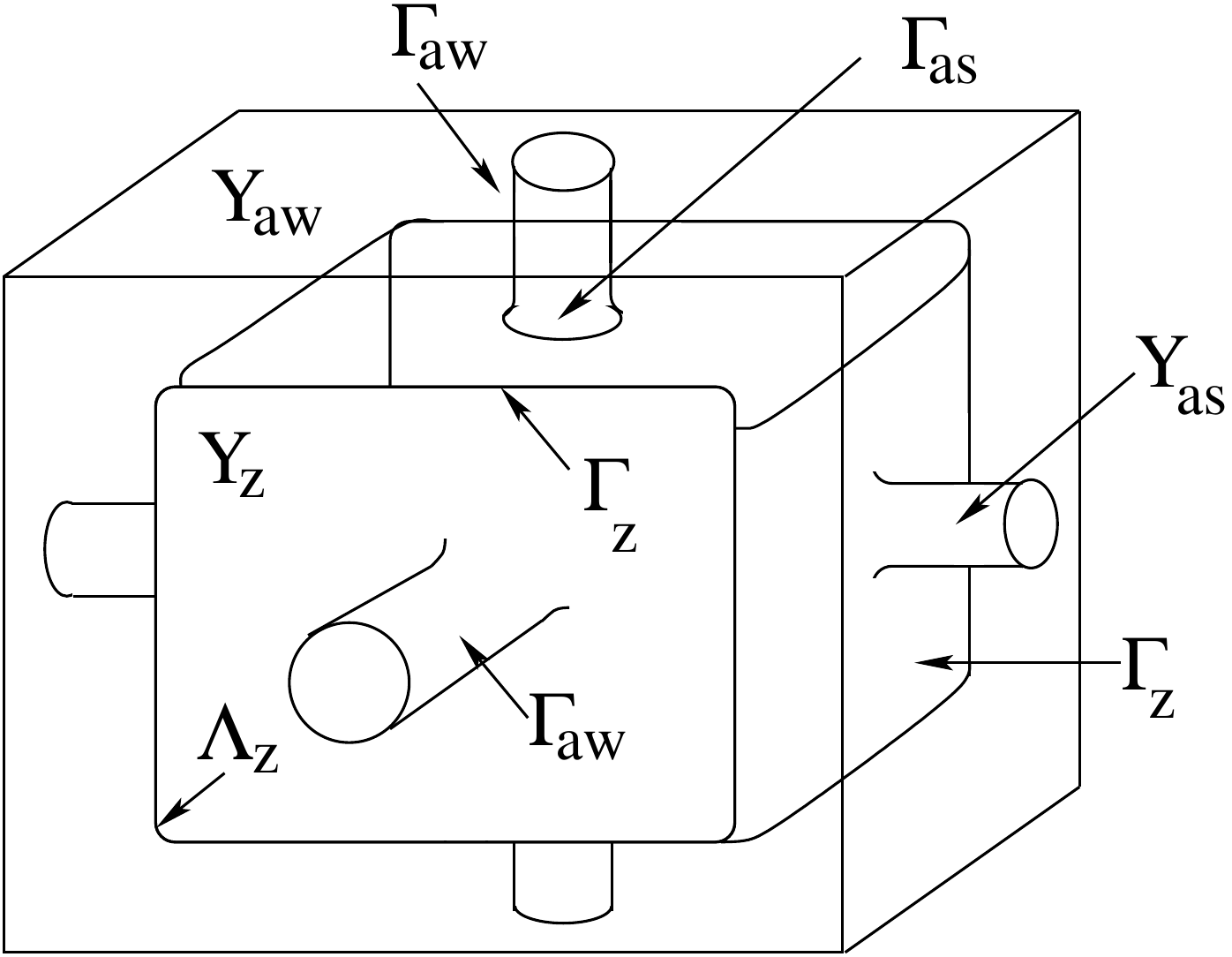}
 \includegraphics[width=4.4cm]{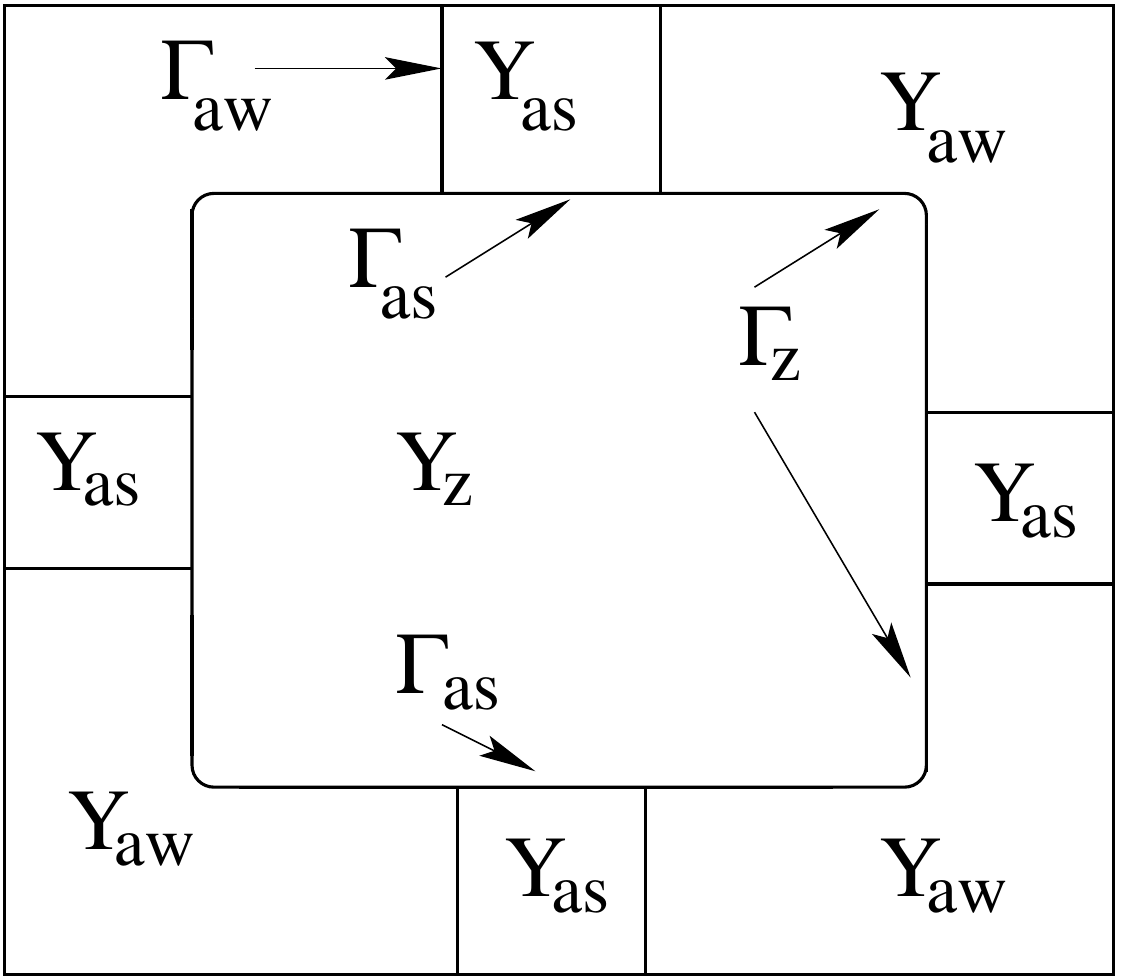}
  \includegraphics[width=5.5cm]{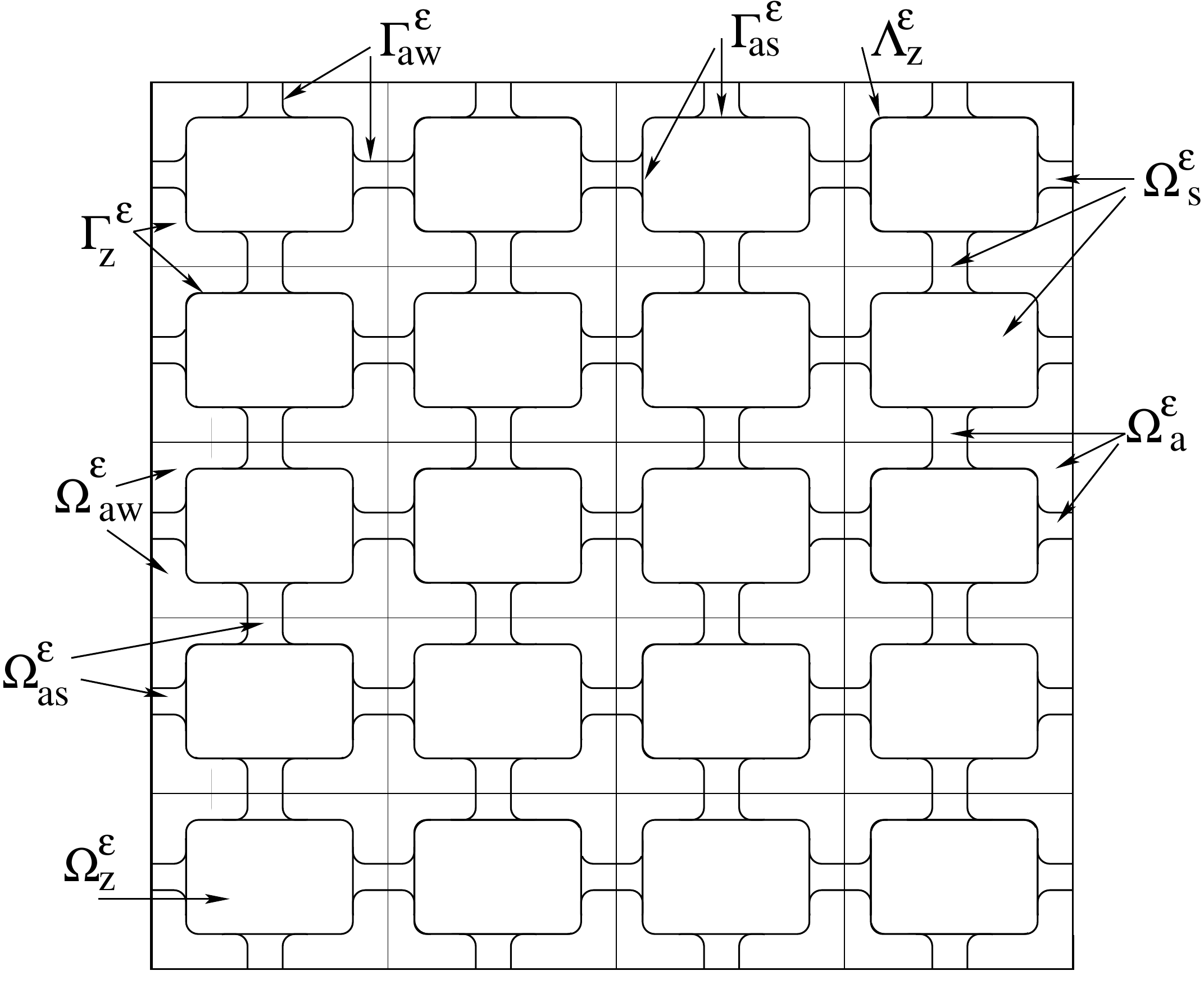}
 \caption{Left: unit cell $Y$. Center: 2D cross-section of a unit cell $Y$. Right: 2D cross-section of a plant tissue $\Omega$.}
\label{fig2}
\end{figure}

In order to simplify the  formulation of the microscopic model  we shall consider in the following the doubled notation for $\Omega_{as}^\ve$, i.e.
$\Omega_{ap}^\ve:= \Omega_{as}^\ve$ and  $\Omega_{sp}^\ve:=\Omega_{as}^\ve$. 

For the velocity  of water we shall consider the stationary version of Stokes equation \eqref{eq:stokes} inside the cells and Darcy equation \eqref{eq:darcy} in the cell walls apoplast  and parts of cell walls comprising plasmodesmata, and incompressibility \eqref{eq:incompressibility} of all flows 
\begin{align}\label{velosity}
  \begin{aligned}
    -\varepsilon^2 \eta \Delta   {\vect v}^\ve_z + \nabla p^\ve_z \hspace{0.1 cm} &=0  &&\text{ in } \Omega^\ve_z\times(0,T)\; , \\
    \text{div } {\vect v}^\ve_z & = 0 && \text{ in } \Omega^\ve_z\times(0,T)\; , \\ 
    {\vect v}^\ve_{i} +  K_{i, \ve} \nabla p^\ve_{i} &= 0  &&\text{ in }  \Omega^\ve_{i} \times(0,T), \quad i=sp, aw, ap\; , \\
    \text{div }  {\vect v}^\ve_{i}  & = 0 && \text{ in } \Omega^\ve_{i}\times(0,T), \quad i=sp, aw, ap \; ,
  \end{aligned}
\end{align}
where $\eta>0$ is the  viscosity constant, and the permeability tensors  are given by $Y$-periodic functions $K_i$, i.e.
$K_{i, \ve}(x)=K_{i}(x/\ve)$ for $x\in \Omega_i^\ve$, with $i=sp, aw, ap$. 

The concentration of an osmotically active solute in different  partitions of a plant cell is assumed to follow the conservation Eq. \eqref{eq:conservation} with a production/degradation term
\begin{align}\label{Eq_Conc}
 \partial_t c^\ve_i + \operatorname{div}(H_M({\vect v}^\ve_i) c^\ve_i - D_i^\ve(t,x)\nabla c^\ve_i) & = F^\ve_i(t,x,c^\ve_i) & &\text{ in }  \Omega_i^\ve\times(0,T), \quad i=z, sp, aw, ap \; ,
\end{align}
where 
\[
H_M(\vect{\xi})_j= \left\{\begin{array}{ll} {\xi}_j \quad&  \text{ for }\,  |\xi_j| \leq M, \\ M \text{sgn}(\xi_j) & \text{ for }  \,  |\xi_j| >  M, \end{array} \right. \; ,\, j=1,2,3\; ,
\]
for some $M>0$,  and $\vect{\xi} \in \mathbb R^3$. The velocities $\vect{v}^\ve_i$ need to be cut-off for technical reasons for the rigorous analysis of the macroscopic model. Assuming bounded velocities in tissues is biologically and physically sensible. The diffusion coefficients $D_i^\ve$ and the production/degradation terms $F^\ve_i$,  with $i=z, sp, aw, ap$,  are determined by $Y$-periodic functions $D_i(t,y)$ and $F_i(t,y,\xi)$, i.e.  
$D_i^\ve(t,x)=D_i(t,\frac x\ve)$  and  $F^\ve_i(t,x, \xi)=  F_i(t,\frac x\ve, \xi)$ for $\xi \in \mathbb R$ and $(t,x)\in (0,T)\times \Omega_i^\ve$.
The symplast and apoplast are coupled by diffusion fluxes and via protein mediated transport of solute through the cell membrane described by the following boundary conditions based on Eqs. \eqref{eq:tc_one} and \eqref{eq:tc_two}
\begin{alignat}{2}
\begin{aligned}
\label{Cell_membrBC}
(D_z^\ve\nabla c^\ve_z - H_M({\vect v}^\ve_z) c^\ve_z)\cdot\vect{n} &= (D_{sp}^\ve\nabla c^\ve_{sp} - H_M({\vect v}^\ve_{sp}) c^\ve_{sp})\cdot\vect{n} \\
&\quad + \ve \beta_{a}^\ve(t,x) \vartheta^\ve_{b,a}- \ve \alpha_s^\ve(t,x) c^\ve_z \vartheta_{f,s}^\ve &&\text{on } \Gamma^{\ve}_{as}\times(0,T)\; ,\\
(D_z^\ve\nabla c^\ve_z - H_M({\vect v}^\ve_z) c^\ve_z)\cdot\vect{n}&=  \ve \beta_{a}^\ve(t,x) \vartheta^\ve_{b,a} \phantom{c^\ve_{aw}}- \ve \alpha_s^\ve(t,x) \, c^\ve_z\, \vartheta_{f,s}^\ve   && \text{on   } \Gamma^\ve_{z}\times(0,T)\; ,\\
(D_{aw}^\ve\nabla c^\ve_{aw} -H_M({\vect v}^\ve_{aw}) c^\ve_{aw})\cdot\vect{n} &=  \ve \alpha_{a}^\ve(t,x) c^\ve_{aw} \vartheta_{f,a}^\ve-\ve \beta_s^\ve(t,x)  \vartheta_{b,s}^\ve  && \text{on   }  \Gamma^\ve_{z}\times(0,T)\; ,\\
(D_{ap}^\ve\nabla c^\ve_{ap} -H_M({\vect v}^\ve_{ap}) c^\ve_{ap})\cdot\vect{n} &=  \ve \alpha_{a}^\ve(t,x) c^\ve_{ap} \vartheta_{f,a}^\ve  \, \, - \ve \beta_s^\ve(t,x)  \vartheta_{b,s}^\ve   && \text{on   }  \Gamma^\ve_{as} \times(0,T)\;, 
\end{aligned}
\end{alignat}
where $\vect{n}$ denotes the outer normal vector to  $\partial\Omega^\ve_z$ and the protein mediated flux $\vect{a}_{2,{\sf I}}$ was expressed by relation \eqref{eq:activeflux} with $\alpha^\ve_{l}$ and $\beta^\ve_l$, for $l=a,s$, are related to the reaction rate coefficients $k_i$, for $i=1,2,3$.  We define the apoplastic and symplastic concentrations $c_a^\ve$  and $c_s^\ve$, respectively, as follows
\begin{align}\label{eq:apoplastic_c}
 c_a^\ve(t,x) = \left\{
\begin{array}{rl}
c_{aw}^\ve(t,x) &\quad \text{ in  } \Omega_{aw}^\ve\times(0,T)\; ,\\
c_{ap}^\ve(t,x) &\quad \text{ in } \Omega_{as}^\ve\times(0,T)\; ,
\end{array}\right. &\quad \text{and} 
&&  c_s^\ve(t,x) = \left\{
\begin{array}{rl}
c_{z}^\ve(t,x) &\quad \text{ in  } \Omega_{z}^\ve\times(0,T)\; ,\\
c_{sp}^\ve(t,x) &\quad \text{ in } \Omega_{as}^\ve\times(0,T)\; .
\end{array}\right.
\end{align}
The dynamics of the transporter concentrations are modelled by following ordinary differential equations
\begin{align}\label{transporter_eq}
  \begin{aligned}
    \partial_t \vartheta_{f,l}^\ve&= R_{l}^\ve(t,x,\vartheta_{f,l}^\ve)- \alpha_{l}^\ve(t,x) c^\ve_{l} \vartheta_{f,l}^\ve + \beta_{l}^\ve(t,x) \vartheta_{b,l}^\ve-\gamma_{f,l}^\ve(t,x) \vartheta_{f,l}^\ve && 
    \text{ on } \Gamma^\ve_{zs}\times(0,T)\; , \\
    \partial_t \vartheta_{b,l}^\ve & = \phantom{ R_{l}^\ve(t,x,\vartheta_{f,l}^\ve)- } \, \, \alpha_{l}^\ve(t,x) c^\ve_{l} \vartheta_{f,l}^\ve - \beta_{l}^\ve(t,x) \vartheta_{b,l}^\ve -\gamma_{b,l}^\ve(t,x) \vartheta_{b,l}^\ve && \text{ on } 
    \Gamma^\ve_{zs}\times(0, T) \; .
  \end{aligned}
\end{align}
with $l=a,s$, and $\gamma^\ve_{j,l}$, with $j=f,b$  are rates of decay of corresponding transporter concentration and 
$R_l^\ve$ are production/degradation terms representing the genetic regulation of the transporters. The main difference to Eq. \eqref{eq:regulation} is the inclusion of the decay terms, which model a possible instability of the transporting proteins. The coefficients and  $R_l^\ve$ are given by $Y$-periodic functions $R_l$, $\alpha_l$, $\beta_l$ and $\gamma_{j,l}$,  defined on $\Gamma_{zs}\times(0, T)$, i.e. $R_l^\ve(t,x, \xi)= R_l(t,\frac x\ve, \xi)$ 
and 
$\alpha^\ve_{l}(t,x)=\alpha_{l}(t,\frac {x}\ve)$,\,
 $\beta^\ve_{l}(t,x)=\beta_{l}(t,\frac {x}\ve)$,   \, 
$\gamma^\ve_{j,l}(t,x) =\gamma_{j,l}(t,\frac{x}\ve)$,  for $\xi \in \mathbb R$  and $(x,t) \in \Gamma_{zs}^\ve\times (0,T)$, with $l=a, s$ and $j=f,b$.\\

In contiguous apoplast and symplast in plasmodesmata populated regions, we pose continuity and zero-flux boundary conditions  -- based on Eqs. \eqref{eq:tc_one} and \eqref{eq:tc_two} taking into consideration that no transporters are present ($\vect{a}_{2,{\sf I}} = 0$ and $[\vect{a}_2] =0$) -- and set the concentration to be continuous
\begin{align}\label{Cell_BC_other}
  \begin{aligned}
    (D_{aw}^\ve\nabla c^\ve_{aw} - H_M({\vect v}^\ve_{aw}) c^\ve_{aw})\cdot\vect{n}_{aw} &= (D_{ap}^\ve\nabla c^\ve_{ap} -H_M({\vect v}^\ve_{ap}) c^\ve_{ap})\cdot\vect{n}_{aw} && \text{ on } \Gamma^{\ve}_{aw}\times(0,T)\; ,\\
    (D_{sp}^\ve\nabla c^\ve_{sp} - H_M({\vect v}^\ve_{sp}) c^\ve_{sp})\cdot\vect{n}_{aw}&= 0 && \text{ on } \Gamma^{\ve}_{aw}\times(0, T)\; , \\
    c^\ve_{aw}& =c_{ap}^\ve   && \text{ on } \Gamma^{\ve}_{aw}\times(0, T)\; ,\\
    c^\ve_z & =c_{sp}^\ve   && \text{ on } \Gamma^{\ve}_{as}\times(0, T)\; ,
  \end{aligned}
\end{align}
where $\vect{n}_{aw}$ is the outer normal vector to $\partial \Omega_{aw}^\ve$.  Initial conditions for solute and transporter concentrations are posed as follows
\begin{align}\label{init_cond}
\begin{aligned}
c^\ve_l(0,x) &= c_{l}^0(x) && \text{ in }  \Omega_l^\ve\; , \\
\vartheta_{j,l}^\ve(0,x) &= \vartheta_{j,l}^{\ve,0}(x)  && \text{ on } \Gamma^\ve_{zs}\; ,
\end{aligned}
\end{align}
where $j=f,b$ and $l=a, s$. The initial conditions for transporter concentrations are defined as  $\vartheta_{j,l}^{\ve,0}(x) = \vartheta_{j,l}^0(x, x/\ve)$ for $x\in \Gamma^\ve_{zs}$, where
 $\vartheta_{j,l}^0(x,y) = \vartheta_{1j,l}^0(x)\vartheta_{2j,l}^0(y)$ for $(x,y)\in\Omega\times \Gamma_{zs}$,    and  $\vartheta_{2j,l}^0$ are $Y$-periodic, with $j=f,b$ and $l=a,s$.\\

Following conditions, based on \eqref{eq:tc_three}, \eqref{eq:tc_four} and \eqref{eq:tc_five}, are considered for the velocity of water on internal  boundaries 
\begin{align}
\label{BC1velosity}
\begin{aligned}
  \left.
    \begin{aligned}
      \kappa_1\big(-2\ve^2 \eta\, (\operatorname{S}\!{\vect v}_z^\ve \vect{n} )\cdot \vect{n}+p^\ve_z -  p_{aw}^\ve \big)
      &= \ve \delta_1(c^\ve_{z}-c^\ve_{aw})  + \ve  {\vect v}_{aw}^\ve\cdot \vect{n}\\
      {\vect v}_z^\ve \cdot \vect{n} ={\vect v}_{aw}^\ve \cdot \vect{n}\; , \qquad  {\vect v}_{z}^\ve \times \vect{n}&=0
    \end{aligned}
  \quad\right\}
  &&&\text{ on }  \Gamma^\ve_{z}\times (0,T)\; , \\[.5em]
  \left.
    \begin{aligned}
      -2\ve^2 \eta\,( \operatorname{S}\!{\vect v}_z^\ve \vect{n}) \cdot \vect{n} +p^\ve_z  &=  p_{sp}^\ve\\
      \kappa_2\big(-2\ve^2 \eta\,(\operatorname{S}\!{\vect v}_z^\ve \vect{n}) \cdot \vect{n}+p^\ve_z - p_{ap}^\ve \big) &= \ve  \delta_2 (c^\ve_{z}-c^\ve_{ap}) +  \ve {\vect v}_{ap}^\ve\cdot \vect{n}\\
      {\vect v}_z^\ve \cdot \vect{n} =   {\vect v}_{sp}^\ve \cdot \vect{n}+ {\vect v}_{ap}^\ve \cdot \vect{n}\; ,& \qquad
      {\vect v}_{z}^\ve \times \vect{n}=0
    \end{aligned}
  \quad\right\}
  &&&\text{ on }  \Gamma^\ve_{as}\times (0,T)\; , \\[.5em]
  p^\ve_{aw}=  p_{ap}^\ve \; , \qquad {\vect v}_{aw}^\ve \cdot \vect{n}_{aw} ={\vect v}_{ap}^\ve \cdot \vect{n}_{aw} \; ,  \qquad   {\vect v}_{sp}^\ve \cdot \vect{n}_{aw} =0 \qquad
  &&&\text{ on }  \Gamma^\ve_{aw}\times (0,T)\; .
\end{aligned}
\end{align}\\

\noindent
On the external boundary $\partial\Omega$ we assume zero-flux boundary conditions for concentrations 
\begin{align}
\begin{aligned}
(H_M( {\vect v}^\ve_{aw}) c^\ve_{aw}- D^\ve_{aw}\nabla c^\ve_{aw})\cdot\vect{n}_{ex} &= 0 
  && \;\qquad\text{on  } (\partial\Omega\setminus  \partial \Omega_{as}^\ve)\times (0,T)\; ,\ \\
   (H_M({\vect v}^\ve_{i}) c^\ve_{i}- D^\ve_{i}\nabla c^\ve_{i} )\cdot\vect{n}_{ex} &= 0
  && \;\qquad\text{on  } (\partial\Omega\cap \partial \Omega_{as}^\ve)\times(0,T)\; ,\ \label{ExternalBC_N}
\end{aligned}
%\end{align}
\intertext{where $i=ap, sp$  and $\vect{n}_{ex}$ is the outer normal vector to $\partial \Omega$. We further prescribe symplastic and apoplastic normal velocities }
%\begin{align}
  \begin{aligned}
    {\vect v}_{aw}^\ve \cdot \vect{n}_{ex} &= v_{D} && \quad\text{ on }\,  (\partial  \Omega_{aw}^\ve\cap \partial\Omega)\times(0,T)\; ,\\
    {\vect v}_{ap}^\ve \cdot \vect{n}_{ex} & = v_{D} && \quad\text{ on }\, (\partial \Omega_{as}^\ve\cap \partial  \Omega)\times(0,T)\; ,\\
    {\vect v}_{sp}^\ve \cdot \vect{n}_{ex} & =0 && \quad\text{ on } \,  (\partial \Omega_{as}^\ve\cap \partial\Omega)\times(0,T)\; .\label{BC3velosity}
  \end{aligned}
\end{align}
To simplify notations, we define similar to \eqref{eq:apoplastic_c} apoplastic and symplastic flow variables
\begin{align}\nonumber
v^\ve_a(t,x) &=  \left\{  
\begin{array}{ll}
v^\ve_{aw}(t,x)  &  \text{ in } \Omega_{aw}^\ve\times (0,T)\; ,\\
v^\ve_{ap}(t,x)  &  \text{ in } \Omega_{as}^\ve\times (0,T)\; ,
\end{array}  \right. & \text{and} 
&& v^\ve_s(t,x) &=  \left\{ 
\begin{array}{ll}
v^\ve_{z}(t,x)  &  \text{ in } \Omega_{z}^\ve\times (0,T)\; ,\\
v^\ve_{sp}(t,x)  &  \text{ in } \Omega_{as}^\ve\times (0,T)\; ,
\end{array}  \right. \\
p^\ve_a(t,x) &=  \left\{  
\begin{array}{ll}
p^\ve_{aw}(t,x)  & \text{ in } \Omega_{aw}^\ve\times (0,T)\; ,\\
p^\ve_{ap}(t,x)  &  \text{ in } \Omega_{as}^\ve\times (0,T)\; ,
\end{array}  \right. & \text{and} 
&& p^\ve_s(t,x) &=  \left\{ 
\begin{array}{ll}
p^\ve_{z}(t,x)  & \text{ in } \Omega_{z}^\ve\times (0,T)\; ,\\
p^\ve_{sp}(t,x)  &  \text{ in } \Omega_{as}^\ve\times (0,T)\; .
\end{array}  \right. \nonumber
\end{align}
The corresponding diffusion coefficients and permeability tensors are given by
\begin{align}\nonumber
D_a(t,y)& = \left\{
\begin{array}{ll}
D_{aw}(t,y) & \text{in }  Y_{aw} \times (0,T)\; ,\\
D_{ap}(t,y) &  \text{in } Y_{as} \times (0,T)\; ,
\end{array}   \right.   & \text{and} && 
D_s(t,y)=   \left\{
\begin{array}{ll}
 D_z(t,y) & \text{in } Y_z \times (0,T)\; ,\\
D_{sp}(t,y) &  \text{in } Y_{as} \times (0,T)\; ,
\end{array}    \right.  \\
K_a(y) & =   \left\{
\begin{array}{ll}
 K_{aw}(y)& \text{in } Y_{aw} \times (0,T)\; ,\\
K_{ap}(y) &  \text{in } Y_{as} \times (0,T)\; . 
\end{array}    \right. 
 \nonumber
\end{align}
The production/degradation terms are defined by  
\begin{align}\nonumber
F_a^\ve(t,x, c_a^\ve) & =   \left\{
\begin{array}{ll}
 F_{aw}^\ve(t,x, c^\ve_{aw})& \text{in } \Omega^\ve_{aw} \times (0,T)\; ,\\
F_{ap}^\ve(t,x, c^\ve_{ap}) &  \text{in } \Omega^\ve_{as} \times (0,T)\; , 
\end{array}    \right. 
& \text{and} && 
F_s^\ve(t,x, c_s^\ve)=   \left\{
\begin{array}{ll}
 F_z^\ve(t,x, c_z^\ve) & \text{in } \Omega^\ve_z \times (0,T)\; ,\\
F_{sp}^\ve(t, x, c_{sp}^\ve) &  \text{in } \Omega^\ve_{as} \times (0,T)\; .
\end{array}    \right. 
\end{align}\\

We assume that $\delta_i\neq 0$ and $\kappa_i\neq 0$, for $i=1,2$, and  define $\kappa^\ve(x) = \kappa(x/\ve)$, $\delta^\ve(x) = \delta(x/\ve)$ for $x\in \Gamma^\ve_{zs}$, where $\kappa$, $\delta$ are $Y$-periodic functions  given by   $\kappa(y)=\kappa_1^{-1}$ on $\Gamma_z$, $\kappa(y)=\kappa_2^{-1}$ on $\Gamma_{as}$ and 
$\delta(y)=\delta_1\kappa_1^{-1}$ on $\Gamma_z$, $\delta(y)=\delta_2 \kappa_2^{-1}$ on $\Gamma_{as}$.

\begin{assumption}\label{assump}
\begin{itemize}
 \item Tensors  $K_{a} \in L^\infty(Y_a)^{3\times 3}$,  $K_{sp} \in L^\infty(Y_{as})^{3\times 3}$  are symmetric and  uniformly elliptic, i.e. $(K_l(y)\xi, \xi)\geq k_l |\xi|^2$ for $k_l>0$, $\xi \in \mathbb R^3$, a.a. $y\in Y_l$, where  $l=a, sp$ and $Y_{sp}:=Y_{as}$.
  \item Diffusion coefficients
 $D_l, \partial_t D_l \in L^{\infty}((0,T)\times Y_l)^{3\times 3}$ are symmetric and uniformly elliptic, i.e. $(D_l(t,y)\xi, \xi)\geq d_l |\xi|^2$ for $d_l>0$, $\xi \in \mathbb R^3$,  a.a. $(t,y)\in (0,T)\times Y_l$, where $l=a,s$. 
\item Production/degradation  $F_l: (0,T)\times Y_l \times \mathbb R \to \mathbb R$ is differentiable in $t$, measurable in $y$, $F_l$ and $\partial_t F_l$ are sublinear in $\xi$, $F_l$ is Lipschitz continuous in $\xi$ uniformly in $(t,y)$, and
$F_l(t,y,\xi_{-})\xi_{-}\leq C|\xi_{-}|^2$ for  $(t,y)\in (0,T)\times Y_l$, where $\xi_{-}=\min\{\xi, 0\}$ and $l=a,s$.
\item Functions $R_l(\cdot,\cdot, \xi) \in C ([0,T]; L^\infty(\Gamma_{zs}))$ 
for all $\xi \in \mathbb R$ are uniformly Lipschitz continuous in $\xi$,   and nonnegative  for nonnegative $\xi$ and $(x,t) \in \Gamma_{zs}\times[0,T]$, with $l=a,s$.
 \item Coefficients  $\alpha_l$,  $\beta_l$,  $\gamma_{j,l}$ $\in C([0,T];L^\infty(\Gamma_{zs}))$ are nonnegative and $\partial_t \alpha_l,  \partial_t \beta_l\in L^{\infty}((0,T)\times \Gamma_{zs})$,  where $j=f,b$ and $l=a,s$. 
\item Initial conditions  $c^0_l\in H^2(\Omega)$ and $\vartheta^0_{1j,l} \in L^\infty(\Omega)$,   $\vartheta^0_{2j,l} \in L^\infty(\Gamma_{zs})$  are nonnegative,  $l=a,s$, $j=f,b$.
\item Boundary condition $v_D\in H^{1/2}(\partial\Omega)$ is given by    $v_D= V_D\cdot \vect{n}_{ex} $ on $\partial\Omega$, where  $V_D \in H^{1}(\Omega)$ 
and $\operatorname{div} V_D=0$ in $\Omega$.
\end{itemize}
\end{assumption}

By  Lemma~4 in  \cite{Tartar1980},  there exists a restriction operator $R_{Y_a} \in \mathcal L(H^1(Y), H^1_{\Gamma_{zs}}(Y_a))$  with the properties 
\begin{equation}\label{restriction_a}
\begin{cases}
 R_{Y_a} \psi = \psi  &  \text{ in a neighborhood of } \partial Y \; , \\
 \psi=0 \text{ on } \Gamma_{zs}  & \Longrightarrow  \, \, \,   R_{ Y_a} \psi = \psi, \\ 
\operatorname{ div } \psi =0   &\Longrightarrow  \, \, \,  \operatorname{div} R_{ Y_a} \psi = 0 \; .
\end{cases}
\end{equation} 
Here $H^1_{\Gamma_{zs}}(Y_a)= \{ \psi \in H^1(Y_a): \, \psi =0 \text{ on } \Gamma_{zs}\}$.
For $\psi\in H^{1}(\Omega)$  we define $\psi_{j}^\ve(y)= \psi(\ve y)$ for  $y \in Y^j= (Y+ k^j)$, with $j=1,\ldots,J$, where $J \in \mathbb N$   such that $\Omega= \cup_{j=1}^J \ve Y^j$, and $k^j \in \mathbb Z^3$. 
Then the operator defined as 
$
\tilde R^\ve_{\Omega_a^\ve} \psi(x) = (R_{Y_a} \psi_j^\ve) (x/\ve) $   for   $x \in \ve Y^j_a$, with  $j=1,\ldots,J,
$
has the properties  that  $\tilde R_{\Omega^\ve_a} \in \mathcal L( H^1(\Omega),  H^1_{\Gamma_{zs}^\ve}(\Omega^\ve_a))$ and 
\begin{equation}\label{restriction_a}
\begin{cases}
\psi=0 \text{ on } \Gamma_{zs}^\ve  &\Longrightarrow \, \, \, \tilde R^\ve_{\tilde \Omega_a^\ve} \psi = \psi \; , \\
\operatorname{ div } \psi =0   &\Longrightarrow  \, \, \,  \operatorname{div } \tilde R^\ve_{\Omega_a^\ve} \psi = 0 \; , \\
\|\tilde R^\ve_{\Omega_a^\ve} \psi \|_{L^2(\Omega^\ve_a)} &\leq  \, \, \,   C \big(\|\psi\|_{L^2(\Omega)}+  \ve \|\nabla \psi\|_{L^2(\Omega)}\big) \; ,\\
\ve \|\nabla \tilde R^\ve_{\Omega_a^\ve} \psi \|_{L^2(\Omega^\ve_a)} &\leq  \, \,  \, C\big( \|\psi\|_{L^2(\Omega)}+  \ve \|\nabla \psi\|_{L^2(\Omega)}\big) \; , 
\end{cases}
\end{equation} 
where $H^{1}_{\Gamma_{zs}^\ve}(\Omega_a^\ve)= \{ \psi \in H^{1}(\Omega_a^\ve): \, \psi =0 \text{ on } \Gamma_{zs}^\ve\}$.

We define  $V^\ve_D = \tilde R_{\Omega^\ve_a} V_D$ and, using the assumptions on $V_D$, obtain   that  $\operatorname{div } V^\ve_D=0$ in $\Omega_a^\ve$ and 
$ \| V^\ve_D \|_{L^2(\Omega^\ve_a)} \leq C \|V_D\|_{H^1(\Omega)}.
$

For a $\sigma >0$, a Lipschitz domain $\Sigma$,  and for  $\psi, \varphi \in L^2((0,\sigma)\times \Sigma)$ we shall  denote
$$
\langle \psi, \varphi\rangle_{\Sigma} = \int_\Sigma \psi \, \varphi \, dx  \qquad  \text{ and } \qquad \langle \psi, \varphi\rangle_{\Sigma,\sigma} = \int_0^\sigma\int_\Sigma \psi \, \varphi \, dx dt\; .
$$
For a bounded Lipschitz domain $\Sigma$ we shall use the notion of the space 
$$
H(\text{div}, \Sigma) = \{\vect{v} \in L^2(\Sigma) \text{ such that  } \text{ div }\vect{v} \in L^2(\Sigma) \} \; ,
$$
provided with the norm 
$$\| \vect{ v} \|^2_{H(\text{div}, \Sigma) }= \|\vect{v}\|^2_{L^2(\Sigma)} + \|\text{div }\vect{v}\|^2_{L^2(\Sigma)}\; . 
$$
We introduce also
\begin{eqnarray*}
V^\ve_{as}&=& \{ v\in H(\text{div}, \Omega_{as}^\ve): 
v\cdot \vect{n} = 0 \text{ on } \partial\Omega_{as}^\ve\cap (\partial\Omega_{aw}^\ve \cup \partial\Omega)\} \; , 
\\ 
V^\ve_a&=& \{ v\in H(\text{div}, \Omega_a^\ve):   v\cdot \vect{n} \in L^2(\Gamma^\ve_{zs}), \,   v\cdot \vect{n}=0 \text{ on } \partial \Omega\} \; ,
\\ 
V^\ve_z&=& \{ v\in H^1(\Omega_z^\ve):  v\times\vect{n} =0 \, \text{ on } \Gamma^\ve_{zs} \} \; .
\end{eqnarray*}
For $\phi\in V_{as}^\ve$, since  $\phi \in H(\text{div},\Omega_{as}^\ve)$ we have $\phi\cdot \vect{n} \in H^{-1/2}(\partial\Omega_{as}^\ve)$ and together with 
 $\phi\cdot \vect{n}=0$ on $\partial\Omega_{as}^\ve\setminus \Gamma_{as}^\ve$  obtain 
$\phi\cdot\vect{n} \in H^{-1/2}(\Gamma_{as}^\ve)$, \cite{Galvis}.  Considering the geometrical structure of $\Omega_{as}^\ve$ and the fact that $\phi\cdot \vect{n} =0$ on $\Gamma_{aw}^\ve$, we can extend $\phi$ by zero from $\Omega_{as}^\ve$ to  $\Omega_a^\ve$  with $\text{div } \phi \in L^2(\Omega_a^\ve)$ and shall  use the same notation for the extension as for the original function.\\

\noindent
We denote the spaces
{\small
\begin{eqnarray*}
\mathcal V^\ve&=&\{ (v_1, v_2, v_3) \in  V^\ve_z\times V^\ve_{a}\times V^\ve_{as}: \langle (v_1-v_2)\cdot \vect{n}, \phi\rangle_{\Gamma_z^\ve}=0  \text{ for } \phi \in L^2(\Gamma_z^\ve)\; ,\\
&& \hspace{4. cm} \langle (v_1-v_2 - v_3)\cdot \vect{n}, \psi \rangle_{H^{-1/2}(\Gamma_{as}^\ve), H^{1/2}(\Gamma_{as}^\ve)}=0 \, \,  \text{ for } \psi\in H^{1/2}(\Gamma_{as}^\ve) \} \; ,\\
\mathcal P^\ve&=&\{(p_1, p_2,p_3)\in L^2(\Omega_z^\ve)\times L^2(\Omega_a^\ve)\times L^2(\Omega_{as}^\ve):
    \langle p_1, 1\rangle_{\Omega_z^\ve} + \langle p_2, 1\rangle_{\Omega_a^\ve}+  \langle p_3, 1\rangle_{\Omega_{as}^\ve}=0\}\; ,
\end{eqnarray*}}%
with the norms
\begin{eqnarray*}
 &\|v\|_{\mathcal V^\ve} &= \|v_1\|_{H(\text{div}, \Omega_z^\ve)} +  \ve\|\nabla v_1\|_{L^2(\Omega_z^\ve)}
 + \|v_2\|_{H(\text{div}, \Omega_a^\ve)}  +
\ve^{1/2} \|v_2\cdot \vect{n}\|_{L^2(\Gamma_{zs}^\ve)} +\|v_3\|_{H(\text{div},\Omega_{as}^\ve)}\; ,\\
&\|p\|_{\mathcal P^\ve}&=  \|p_1\|_{L^2(\Omega_z^\ve)} + \|p_2\|_{L^2(\Omega_a^\ve)} + \|p_3\|_{L^2(\Omega_{as}^\ve)} \; .
 \end{eqnarray*}
Notice that for $\psi \in \mathcal V^\ve$  due to the assumptions on the normal components at the boundaries we have  $\text{div }\varphi \in L^2(\Omega)$, where $\varphi= \psi_1$ in $\Omega_z^\ve$, $\varphi= \psi_2$ in $\Omega_{aw}^\ve$ and $\varphi=\psi_2 + \psi_3$ in $\Omega_{as}^\ve$. We shall  consider   also $L^2(0,T; \mathcal V^\ve)$ and $L^2(0,T; \mathcal P^\ve)$ with the norms
\begin{align*}
 \|v\|_{L^2(0,T;\mathcal V^\ve)} &= \|v_1\|_{L^2(0,T; H(\text{div}, \Omega_z^\ve))} +  \ve\|\nabla v_1\|_{L^2((0,T)\times\Omega_z^\ve)}
 + \|v_2\|_{L^2(0,T; H(\text{div}, \Omega_a^\ve))}  \\
 &\qquad+\ve^{1/2} \|v_2\cdot \vect{n}\|_{L^2((0,T)\times\Gamma_{zs}^\ve)} +\|v_3\|_{L^2(0,T;H(\text{div},\Omega_{as}^\ve))}\; ,\\
\|p\|_{L^2(0,T;\mathcal P^\ve)}&=  \|p_1\|_{L^2((0,T)\times\Omega_z^\ve)} + \|p_2\|_{L^2((0,T)\times\Omega_a^\ve)} + \|p_3\|_{L^2((0,T)\times\Omega_{as}^\ve)}\; .
 \end{align*}
The corresponding  divergence-free space is denoted by
\begin{equation*}
\mathcal V^\ve_d= \{(v_1, v_2, v_3) \in \mathcal V^\ve: \text{div } v_1 = 0  \text{ in } \Omega_z^\ve\; , \, \text{div } v_2 = 0  \text{ in } \Omega_a^\ve\; , \,  \text{div } v_3 = 0  \text{ in } \Omega_{as}^\ve\}\; .
\end{equation*}

%%%%%%%%%%%%%%%%%%%%%%%%%%%%%%%%%%%%%%%

\section{Well-posedness  and a  priori  estimates}\label{A_priori_Estim}
We start with a weak formulation of the microscopic model   \eqref{velosity}-\eqref{BC3velosity}.
\begin{definition}
 Weak solution of   \eqref{velosity}-\eqref{BC3velosity} are velocity field
$(\vect{v}^\ve_z, \vect{v}^\ve_{a}-V_{D}^\ve, \vect{v}^\ve_{sp}) \in L^2(0,T;\mathcal V^\ve)$,  
and   pressure $p^\ve=(p^\ve_z, p^\ve_{a}, p^\ve_{sp}) \in L^2(0,T; \mathcal P^\ve)$,  
 satisfying 
 equations \eqref{velosity} with boundary conditions \eqref{BC1velosity}  and \eqref{BC3velosity} in  variational formulation
{%\small
 \begin{align}\label{Stokes}
  \begin{aligned}
 \ve^2\langle 2\eta \operatorname{S} {\vect v}_z^\ve, \operatorname{S} \psi_1\rangle_{\Omega_z^\ve,T} -\langle p_z^\ve, \text{div }\psi_1
\rangle_{\Omega_z^\ve,T}  + \langle K^{-1}_{a, \ve}(x) {\vect v}^\ve_{a},  \psi_2\rangle_{\Omega_{a}^\ve,T}
  - \langle p_{a}^\ve,  \text{div } \psi_2 \rangle_{\Omega_{a}^\ve,T} \qquad\qquad\quad\\
+ \langle K^{-1}_{sp, \ve}(x){\vect v}^\ve_{sp}, \psi_3 \rangle_{\Omega^\ve_{as},T}
 -  \langle p_{sp}^\ve, \text{div } \psi_3\rangle_{\Omega^\ve_{as},T} 
+\ve
\langle  \kappa^\ve(x) {\vect v}_a^\ve\cdot \vect{n}
+  {\delta^\ve}(x)(c_s^\ve -c_{a}^\ve), \psi_2\cdot \vect{n} \rangle_{\Gamma^\ve_{zs},T} & =0 \; ,  \\
  \langle\text{div } \vect{v}_z^\ve, q_1 \rangle_{\Omega_z^\ve,T}+ 
 \langle \text{div } \vect{v}^\ve_a, q_2 \rangle_{\Omega_a^\ve,T}+ \langle\text{div } \vect{v}^\ve_{as}, q_3
  \rangle_{\Omega_{as}^\ve,T} &=0 \;  , 
\end{aligned} \hspace{-0.4 cm }
\end{align}}%
for $\psi=(\psi_1, \psi_2, \psi_3) \in  L^2(0,T;\mathcal V^\ve)$ and $q=(q_1, q_2, q_3)\in L^2(0,T;\mathcal P^\ve)$, and functions $c_l^\ve\in L^\infty((0,T)\times \Omega_l^\ve)$, $c^\ve_l \in L^2(0,T;H^1(\Omega^\ve_l))\cap H^1(0,T; L^2(\Omega^\ve_l))$, 
for $l=a,s$,  satisfying equations \eqref{Eq_Conc} with boundary conditions \eqref{Cell_membrBC}, \eqref{Cell_BC_other}, and \eqref{ExternalBC_N} in the weak form 
 \begin{align}\label{def_apop}
  \begin{aligned}
 \langle\partial_t c^\ve_a, \varphi_1 \rangle_{\Omega^\ve_a,T} +\langle D_a^\ve(t,x) \nabla  c^\ve_a  -
H_M({\vect v}^\ve_a)    c^\ve_a ,  \nabla\varphi_1 \rangle_{\Omega^\ve_a,T}  = 
\ve  \langle
\beta_s^\ve(t,x) \vartheta_{b,s}^\ve- \alpha_a^\ve(t,x) c^\ve_{a} \vartheta_{f,a}^\ve, \varphi_1  \rangle_{\Gamma^\ve_{zs}, T}\\
+\langle F^\ve_a(t,x,c^\ve_a), \varphi_1 \rangle_{\Omega^\ve_{a}, T} \; ,
\end{aligned}
\end{align}
and  
 \begin{align}\label{def_symp}
  \begin{aligned}
\langle \partial_t c^\ve_{s}, \varphi_2 \rangle_{\Omega^\ve_s, T}+ \langle
D_{s}^\ve(t,x) \nabla  c^\ve_{s}  -  H_M({\vect v}^\ve_{s}) c^\ve_{s},   \nabla \varphi_2 \rangle_{\Omega^\ve_s, T}  = 
 \ve \langle\beta_a^\ve(t,x)  \vartheta_{b,a}^\ve -
 \alpha_s^\ve(t,x) c^\ve_s \vartheta_{f,s}^\ve,\varphi_2 \rangle_{\Gamma^\ve_{zs}, T} \\
 +\langle F^\ve_s(t,x,c^\ve_s), \varphi_2 \rangle_{\Omega^\ve_{s}, T} \; ,
\end{aligned}
\end{align}
for all $\varphi_1 \in L^2(0,T;H^1(\Omega^\ve_a))$, $\varphi_2 \in L^2(0,T;H^1(\Omega^\ve_s))$, with  $c^\ve_a \to c^0_{a}$, $c^\ve_s \to c^0_{s}$ in $L^2$ as $t\to 0$, 
 and   transporter concentrations  $\vartheta_{j,l}^\ve 
 \in W^{1,\infty}(0,T; L^\infty(\Gamma^\ve_{zs}))$, with $j=f,b$ and $l=a,s$, satisfying 
 ordinary differential equations \eqref{transporter_eq} 
a.e. on $\Gamma_{zs}^\ve\times(0,T)$  together with    initial conditions  \eqref{init_cond}   a.e. on $\Gamma^\ve_{zs}$.
\end{definition}
In following we shall use the notation $\bar V_D^\ve=(0, V_D^\ve, 0)$ and $(\vect{v}^\ve-\bar V_D^\ve)= (\vect{v}^\ve_z, \vect{v}^\ve_{a}-V_{D}^\ve, \vect{v}^\ve_{sp})$.
 
%%%%%%%%%%%%%%%%%%%%%%
\subsection{Existence and estimates for $(\vect{v}_s^\ve, p_s^\ve)$ and   $(\vect{v}_a^\ve, p^\ve_a)$}
First we shall prove Korn's type inequality satisfied by functions from the space  
\begin{eqnarray*}
 \tilde{\mathcal V}^\ve&=&\{ (v_1, v_2, v_3) \in  H^1(\Omega_z^\ve)\times V^\ve_{a}\times V^\ve_{as}: \langle (v_1-v_2)\cdot \vect{n}, \phi\rangle_{\Gamma_z^\ve}=0  \text{ for } \phi \in L^2(\Gamma_z^\ve)\; ,\\
&& \hspace{4. cm} \langle (v_1-v_2 - v_3)\cdot \vect{n}, \psi \rangle_{H^{-1/2}(\Gamma_{as}^\ve), H^{1/2}(\Gamma_{as}^\ve)}=0 \, \,  \text{ for } \psi\in H^{1/2}(\Gamma_{as}^\ve) \}  \; .
\end{eqnarray*}
\begin{lemma}\label{Korn_Estim}
For $\vect\psi \in \tilde{\mathcal  V}^\ve$ we have  the  following Korn's type inequality
\begin{align}\label{Korn2}
  \begin{aligned}
\|\vect \psi_1\|_{L^2(\Omega^\ve_z)} + \ve\|\nabla \vect \psi_1\|_{L^2(\Omega^\ve_z)} \leq C\big(\ve \|\operatorname{S} \vect \psi_1 \|_{L^2(\Omega^\ve_z)}+\ve^{1/2}||\vect \psi_1\times\vect{n}||_{L^2(\Gamma^\ve_{zs})} +
 \|\vect \psi_2\|_{L^2(\Omega_{a}^\ve)}   \\+
 \|\vect \psi_3\|_{L^2(\Omega_{as}^\ve)}+ \ve \|\operatorname{div } \vect \psi_1^\ve\|_{L^2(\Omega_z^\ve)} +\ve \|\operatorname{div } \vect\psi_2^\ve\|_{L^2(\Omega_{a}^\ve)}+ \ve \|\operatorname{div } \vect\psi_3^\ve\|_{L^2(\Omega_{as}^\ve)}\big)\; .
\end{aligned}
\end{align}
\end{lemma}
\begin{proof}
The proof follows the same lines as in \cite{Arbogast}.  First  we show the estimate for    $\hat{\vect \psi} \in  \mathcal V(Y)$, where $\mathcal V(Y)=\{ \vect\psi_1 \in H^1(Y_z), \, \vect \psi_2 \in \mathcal V_a(Y), 
\vect\psi_3 \in \mathcal V_{as}(Y),  \langle(\vect \psi_1-\vect \psi_2)\cdot \vect{n}, \phi_1\rangle_{\Gamma_z}=0 \text{ for } \phi_1 \in L^2(\Gamma_z),  
\text{ and } \langle(\vect \psi_1- \vect \psi_2-\vect\psi_{3})\cdot \vect{n}, \phi_2\rangle_{H^{-1/2}, H^{1/2}} =0
  \text{ for } \phi_2 \in H^{1/2}(\Gamma_{as}) \}$, 
  with  $\mathcal V_a(Y)= \{ v \in H(\text{div}, Y_a), v \cdot \vect{n} \in L^2(\Gamma_a)\}$ and $\mathcal V_{as}(Y)= \{ v \in H(\text{div}, Y_{as}), v\cdot\vect{n} = 0 \text{ on } \Gamma_{aw}\}$.  Then  scaling argument will imply inequality \eqref{Korn2} for $\vect\psi \in \tilde{\mathcal  V}^\ve$. 
Suppose it is not true that there exists a constant $\tilde C$ such that 
\begin{align}\label{assum_contra}
  \begin{aligned}
\| \hat{\vect \psi}_1 \|_{L^2(Y_z)} + \|\nabla \hat{\vect \psi}_1\|_{L^2(Y_z)} \leq  \tilde C 
\Big( \| \operatorname{S} \hat{\vect \psi}_1\|_{L^2(Y_z)} + \| \hat{\vect\psi}_1\times \vect{n}\|_{L^2(\Gamma_{zs})} + \|\hat{\vect\psi}_2\|_{L^2(Y_{a})}  +  \|\hat{\vect\psi}_3\|_{L^2(Y_{as})} \\+
\| \text{div } \hat{\vect\psi}_1 \|_{L^2(Y_z)} + \| \text{div } \hat{\vect\psi}_2 \|_{L^2(Y_{a})}+ \| \text{div } \hat{\vect\psi}_3 \|_{L^2(Y_{as})} \Big)\; .
\end{aligned}
\end{align}
Then there exists a sequence $\{ \hat{\vect{\psi}}^m\} \subset  \mathcal V(Y)$ such that 
\begin{equation}\label{equal_one}
\| \hat{\vect \psi}_{z}^m \|_{L^2(Y_z)} + \|\nabla \hat{\vect \psi}_{z}^m\|_{L^2(Y_z)}=1 
\end{equation}
and 
\begin{align}\label{Inequal_n}
  \begin{aligned}
 \|\operatorname{S} \hat{\vect \psi}^m_{1}\|_{L^2(Y_z)} + \| \hat{\vect\psi}^m_{1}\times \vect{n}\|_{L^2(\Gamma_{zs})} + \|\hat{\vect\psi}^m_{2}\|_{L^2(Y_{a})}  +  \|\hat{\vect\psi}^m_{3}\|_{L^2(Y_{as})}\\+  \| \operatorname{div } \hat{\vect\psi}^m_1 \|_{L^2(Y_z)} + \| \operatorname{div } \hat{\vect\psi}^m_{2} \|_{L^2(Y_{a})}
 + \| \operatorname{div } \hat{\vect\psi}^m_{3} \|_{L^2(Y_{as})}
  \leq \frac 1m \; .
\end{aligned}
\end{align}
The last inequality implies that 
\begin{equation*}
\hat{\vect \psi}^m_{2} \to 0  \quad \text{ in } \, \, H(\text{div}, Y_{a}) \qquad \text{ and  } \qquad  \hat{\vect \psi}^m_{3} \to 0  \quad \text{ in } \, \, H(\text{div}, Y_{as})\; ,
\end{equation*}
whereas, due to \eqref{equal_one}, there exists  $\hat{\vect \psi}_1 \in H^1(Y_z) $ such that 
\begin{equation*}
\hat{\vect\psi}^m_{1} \to \hat{\vect\psi}_1 \quad  \text{ weakly in } \,  H^1(Y_z).
\end{equation*}
We shall denote by $\hat{\vect \psi}$ the extension of $\hat{\vect \psi}_1$ by zero into $Y$. Since $\hat{\vect \psi}^m=\hat{\vect \psi}^m_1\chi_{Y_z} +\hat{\vect \psi}^m_2\chi_{Y_a}+\hat{\vect \psi}^m_3\chi_{Y_{as}}$ is bounded in $H(\text{div}, Y)$, it converges weakly and we  conclude that 
\begin{equation}
\hat{\vect \psi}^m \to \hat{\vect \psi} \quad  \text{ weakly in } \, \,  H(\text{div}, Y) \; .
\end{equation}
Using  the estimate for the norm  in the space dual to $H^{1/2}_{00}(\Gamma_z)$,  i.e. 
\begin{equation}
\| \hat{\vect\psi}^m_{2} \cdot \vect{n} \|_{(H^{1/2}_{00}(\Gamma_{zs}))^\prime} \leq C \|\hat{\vect\psi}^m_{2}\|_{H(\text{div}, Y_{a})} \to 0 \; , 
\end{equation}
and    the boundary condition $\hat{\vect\psi}^m_{1} \cdot \vect{n}= \hat{\vect\psi}^m_{2} \cdot \vect{n}$ on $\Gamma_z$, we obtain
$
\hat{\vect\psi}_1\cdot \vect{n}= \hat{\vect\psi}_2\cdot\vect{n} = 0 \text{ on } \Gamma_z.
$
Additionally we have that 
$$
\|\hat{\vect\psi}^m_{1} \times \vect{n} \|_{L^2(\Gamma_{zs})} \to  0,  \quad \text{ as } \, m \to \infty, 
\quad \text{ and \,  therefore } \quad 
\hat{\vect\psi}_1 \times \vect{n} =  0 \text{ on } \Gamma_{zs}. 
$$
Now the classical Korn inequality in $Y_z$ with $\hat{\vect\psi}_1=0$ on $\Gamma_z \subset \partial Y_z$ can be applied and we have
\begin{equation}\label{Korn_clasic}
\| \hat{\vect\psi}_1 \|_{L^2(Y_z)} + \|\nabla \hat{\vect\psi}_1\|_{L^2(Y_z)} \leq \hat C  \|\operatorname{S}_y \hat{\vect \psi}_1\|_{L^2(Y_z)}.
\end{equation} 
Considering  \eqref{equal_one} and \eqref{Inequal_n},  we obtain that the left-hand side in \eqref{Korn_clasic} is equal to one, whereas  the right-hand side is zero. This  yields the contradiction to teassumption  that there no such constant   $\tilde C$ for which \eqref{assum_contra} hold true. 
\\
Due to the geometrical assumption  on $\Omega$ we can write $\Omega=\cup_{j=1}^{J} \,  \ve (Y+ k^j)$ with some   $J\in \mathbb N$ and $k^j \in \mathbb Z^3$. We consider now $\vect \psi \in \mathcal V^\ve$ and  for $y \in Y$ define $\hat{\vect \psi}^j(y)= \vect\psi(\ve y + \ve k^j)$, and obtain    $\hat{\vect \psi}^j \in \mathcal V(Y)$.
Applying  \eqref{assum_contra}  for  each $Y_j= (Y+ k^j)$ we obtain estimate for $\hat{\vect \psi}^j$ 
\begin{align}\label{assum_contra_2}
  \begin{aligned}
\| \hat{\vect \psi}^j_1 \|_{L^2(Y^j_z)} + \|\nabla \hat{\vect \psi}_1^j\|_{L^2(Y^j_z)} \leq  \tilde C 
\big( \| \operatorname{S} \hat{\vect \psi}_1^j\|_{L^2(Y_z^j)} + \| \hat{\vect\psi}^j_1\times \vect{n}\|_{L^2(\Gamma^j_{zs})} + \|\hat{\vect\psi}^j_2\|_{L^2(Y^j_{a})}  +  \|\hat{\vect\psi}_3^j\|_{L^2(Y^j_{as})} \\+
\| \text{div } \hat{\vect\psi}_1^j\|_{L^2(Y^j_z)} + \| \text{div } \hat{\vect\psi}_2^j \|_{L^2(Y_{a}^j)}+ \| \text{div } \hat{\vect\psi}_3^j \|_{L^2(Y_{as}^j)} \big)\; .
\end{aligned}
\end{align}
Summation over $j=1,\ldots, J$ and change of variables $x=\ve(y+k^j)$ for~$y \in Y$~in~\eqref{assum_contra_2}~yield~\eqref{Korn2}.
\end{proof}

For the proof of existence and uniqueness of  $\vect{v}^\ve \in L^2(0,T;\mathcal V^\ve)$  and $p^\ve \in L^2(0,T;  \mathcal P^\ve)$  we shall define 
two bilinear  forms  $a^\ve(\cdot, \cdot): L^2(0,T;\mathcal V^\ve)\times L^2(0,T;\mathcal V^\ve) \to \mathbb R$ and $b(\cdot, \cdot):L^2(0,T;\mathcal V^\ve) \times L^2(0,T;\mathcal P^\ve)  \to \mathbb R$: 
\begin{align*}
 a^\ve(\varphi,\psi)&=  \ve^2\langle 2\eta \operatorname{S} {\varphi}_1, \operatorname{S} \psi_1\rangle_{\Omega_z^\ve, T}  + \langle K^{-1}_{a, \ve} {\varphi}_2,  \psi_2\rangle_{\Omega_{a}^\ve, T}
  + \langle K^{-1}_{sp, \ve} {\varphi}_3, \psi_3 \rangle_{\Omega^\ve_{as}, T}
  +\ve \langle  \kappa^\ve\,  {\varphi}_2\cdot \vect{n}, \psi_2\cdot \vect{n} \rangle_{\Gamma^\ve_{zs}, T}\; , \\
 b(\psi, q) &=\langle  \text{div } \psi_1, q_1 \rangle_{\Omega^\ve_{z}, T}+\langle  \text{div } \psi_2, q_2  \rangle_{\Omega^\ve_{a}, T}+ \langle \text{div } \psi_3, q_3 \rangle_{\Omega^\ve_{as}, T }\; ,
\end{align*}
for $\varphi, \,  \psi \in L^2(0,T;\mathcal V^\ve)$ and  $ q\in L^2(0,T;  \mathcal P^\ve)$.\\ 

\noindent
For $c_s^\ve, c_a^\ve\in L^2((0,T)\times\Gamma^\ve_{zs})$ we define a linear form $f^\ve(\cdot):  L^2(0,T; \mathcal V^\ve)\to \mathbb R$  
\begin{equation*}
 f^\ve (\psi)=\ve\langle  {\delta^\ve}(c_s^\ve -c_{a}^\ve),   \psi_2 \cdot \vect{n}\rangle_{\Gamma^\ve_{zs}, T}  
   \quad \text{ for } \psi\in  L^2(0,T;\mathcal  V^\ve)\; .
\end{equation*}

\begin{theorem}\label{Existence_Stokes}
 For $c^\ve_a, c^\ve_s \in L^2((0,T)\times \Gamma^\ve_{zs})$ and  $K_a$, $K_{sp}$, $v_{D}$ satisfying Assumption \ref{assump}  and for any $\ve >0$ there exists a unique solution 
 $(\vect{v}^\ve-\bar V_D^\ve) \in L^2(0,T;\mathcal V^\ve)$, $p^\ve \in L^2(0,T; \mathcal P^\ve)$ of \eqref{velosity} with transmission conditions \eqref{BC1velosity} and boundary conditions \eqref{BC3velosity} satisfying
{\small 
\begin{equation}\label{ap_estim}
\begin{aligned}
\hspace{-0.13 cm}  &\ve ||\nabla {\vect v}^\ve_z||_{L^2((0,T)\times \Omega^\ve_z)} +   ||{\vect v}^\ve_{s}||_{L^2((0,T)\times\Omega^\ve_{s})}
  + ||{\vect v}^\ve_a||_{L^2((0,T)\times\Omega^\ve_a)} + \ve^{\frac12} ||{\vect v}^\ve_a\cdot\vect{n}||_{L^2((0,T)\times\Gamma^\ve_{zs})} +
 \\
\hspace{-0.2 cm}  & ||p^\ve_s||_{ L^2((0,T)\times\Omega^\ve_s)/\mathbb R} 
  +  ||p^\ve_a||_{L^2((0,T)\times\Omega^\ve_a)/\mathbb R}
 \leq C\ve^{\frac 12}(\|c_a^\ve\|_{L^2((0,T)\times\Gamma_{zs}^\ve)}+\|c_s^\ve\|_{L^2((0,T)\times\Gamma_{zs}^\ve)}) 
  + C\|V_D\|_{H^{1}(\Omega)} \; , \hspace{-0.1 cm }
\end{aligned}
\end{equation}}
and there exist extensions $P_s^\ve$  of $p^\ve_s$ from $\Omega_s^\ve$ into $\Omega$ and $P_a^\ve$ of $p^\ve_a$ from $\Omega_a^\ve$ into $\Omega$, satisfying
 \begin{equation}\label{ap_estim_extension}
 ||P^\ve_s||_{ L^2(\Omega_T)/\mathbb R}+  ||P^\ve_a||_{L^2(\Omega_T)/\mathbb R}
 \leq C(\ve^{\frac 12}\|c_a^\ve\|_{L^2((0,T)\times\Gamma_{zs}^\ve)}+ 
 \ve^{\frac 12}\|c_s^\ve\|_{L^2((0,T)\times\Gamma_{zs}^\ve)}
  + \|V_D\|_{H^{1}(\Omega)})\; , 
\end{equation}
where a universal  constant $C$ is independent of $\ve$ and $\Omega_T= \Omega\times (0,T)$.\\
Additionally for $c^\ve_{a,i}, c^\ve_{s,i} \in L^2((0,T)\times \Gamma^\ve_{zs})$, with $i=1,2$, holds
\begin{align}\label{uniqueness_estim}
  \begin{aligned}
 &  ||{\vect v}^\ve_{s,1}-{\vect v}^\ve_{s,2} ||_{L^2((0,T)\times\Omega^\ve_{s})}
 + ||{\vect v}^\ve_{a,1} - {\vect v}^\ve_{a,2}||_{L^2((0,T)\times\Omega^\ve_a)} \\
 &\hspace{4. cm}\leq C  
 ( \ve^{\frac 12} \|c_{a,1}^\ve- c_{a,2}^\ve\|_{L^2((0,T)\times\Gamma_{zs}^\ve)}+ \ve^{\frac 12}\|c_{s,1}^\ve-c_{s,2}^\ve\|_{L^2((0,T)\times\Gamma_{zs}^\ve)})\; .
  \end{aligned}
  \end{align}
\end{theorem}
\begin{proof}
We can reformulate    Stokes-Darcy problem  \eqref{Stokes}  as
\begin{eqnarray}\label{Stokes_2}
\begin{cases}
 a^\ve(\vect{v}^\ve-\bar V_D^\ve, \psi) + b(\psi, p^\ve) =-a^\ve(\bar V_D^\ve, \psi)- f^\ve(\psi) & \quad \text{ for } \psi \in L^2(0,T;\mathcal V^\ve)\; ,\\
   b(\vect{v}^\ve, q) = 0 & \quad \text{ for }  q\in L^2(0,T; \mathcal P^\ve)\; ,
   \end{cases}
\end{eqnarray}
and apply the abstract theory of mixed problems, \cite{Girault},~to show the existence of a~unique~solution~of~\eqref{Stokes_2}.\\
\noindent
Considering  $\psi \in L^2(0,T; \mathcal V^\ve_d)$ , using  ${\psi}_1 \times \vect{n} =0 $ on  $(0,T)\times\Gamma_{zs}^\ve$ and applying   inequality \eqref{Korn2}, we obtain 
\begin{align*}
 a^\ve(\psi,\psi)&\geq C \big(\ve^2\| \operatorname{S} \psi_1\|^2_{L^2((0,T)\times\Omega_z^\ve)} + \| \psi_2\|^2_{L^2((0,T)\times\Omega_{a}^\ve)} + \|\psi_3\|^2_{L^2((0,T)\times\Omega_{as}^\ve)} 
 + \ve\|\psi_2\cdot\vect{n}\|^2_{L^2((0,T)\times\Gamma_{zs}^\ve)}\big)\\
 &\geq C \|\psi\|^2_{L^2(0,T;\mathcal V^\ve)}\; ,
\end{align*}
and conclude  that $a^\ve(\cdot, \cdot)$ is $L^2(0,T;\mathcal V_d^\ve)$-elliptic.\\
The bilinear forms $a^\ve(\cdot, \cdot)$ and $b(\cdot, \cdot)$ are  continuous  with constants independent of $\ve$, i.e. for $\psi,\, \varphi \in L^2(0,T;\mathcal V^\ve)$ and $q\in L^2(0,T;\mathcal P^\ve)$, applying  H\"older's inequality,  we have 
\begin{align}\label{a_bound}
\begin{aligned}
 |a^\ve(\psi,\varphi)| & \leq  C \big(\ve^2\| \operatorname{S} \psi_1\|_{L^2((0,T)\times \Omega_z^\ve)}\| \operatorname{S} \varphi_1\|_{L^2((0,T)\times\Omega_z^\ve)} + \| \psi_{2}\|_{L^2((0,T)\times\Omega_{a}^\ve)} \| \varphi_{2}\|_{L^2((0,T)\times\Omega_{a}^\ve)} \\
  &\quad +\| \psi_{3}\|_{L^2((0,T)\times\Omega_{as}^\ve)} \| \varphi_{3}\|_{L^2((0,T)\times\Omega_{as}^\ve)}
+ \ve\| \psi_{2}\cdot\vect{n}\|_{L^2((0,T)\times\Gamma_{zs}^\ve)}\| \varphi_{2}\cdot\vect{n}\|_{L^2((0,T)\times\Gamma_{zs}^\ve)}
\big)\\
& \leq C \|\psi\|_{L^2(0,T;\mathcal V^\ve)} \|\varphi\|_{L^2(0,T;\mathcal V^\ve)}\; ,
\end{aligned}
\end{align}
and 
\[
\begin{split}
|b(\psi, q)| & \leq (\|\operatorname{div} \psi_1\|_{L^2((0,T)\times\Omega_z^\ve)} + \|\operatorname{div} \psi_2\|_{L^2((0,T)\times\Omega_a^\ve)} +\|\operatorname{div} \psi_3\|_{L^2((0,T)\times\Omega_{as}^\ve)} )\times\\
 & \qquad \times(\|q_1\|_{L^2((0,T)\times\Omega_z^\ve)} +\|q_2\|_{L^2((0,T)\times\Omega_a^\ve)} + \|q_3\|_{L^2((0,T)\times\Omega_{as}^\ve)}) \leq \|\psi\|_{L^2(0,T;\mathcal V^\ve)} \|q\|_{L^2(0,T;\mathcal P^\ve)} \; .
\end{split}
\]
Now we shall prove that $b(\cdot, \cdot)$ satisfies the inf-sup condition.
For  any $q\in L^2(0,T;\mathcal P^\ve)\setminus\{0\}$  we shall construct $\psi \in L^2(0,T;\mathcal V^\ve)\setminus\{0\}$ such that 
\[
b(\psi, q) \geq C \|\psi\|_{L^2(0,T;\mathcal V^\ve)}\|q\|_{L^2(0,T;\mathcal P^\ve)}\; . 
\]
For given $q=(q_1, q_2, q_3) \in L^2(0,T; \mathcal P^\ve)$ we define $\tilde q_1 \in L^2\big((0,T)\times\Omega\big)$ and  $\tilde q_2 \in L^2\big((0,T)\times\Omega_s^\ve\big)$  as
\begin{eqnarray*}
\tilde q_1= q_2 \, \,  \text{ in  } (0,T)\times \Omega_a^\ve\; ,  \, &&  \tilde q_1=q_1+\frac 1{|\Omega_z^\ve|}\langle q_3, 1\rangle_{\Omega_{as}^\ve} \, \text{ in } (0,T)\times\Omega_z^\ve \; , \\
 \tilde q_2= q_3 \, \, \text{  in  } (0,T)\times\Omega_{as}^\ve \; , &&  \tilde q_2= -\frac 1{|\Omega_z^\ve|}\langle q_3, 1\rangle_{\Omega_{as}^\ve}\, \quad  \,  \text{  in } (0,T)\times\Omega_{z}^\ve\; . 
 \end{eqnarray*}
Then, due to
 $\langle \tilde q_1, 1 \rangle_\Omega=0$ and $\langle \tilde q_2, 1\rangle_{\Omega_{s}^\ve}=0$,  there exist $\tilde \psi_1 \in L^2(0,T;H^1_0(\Omega))$ and 
 $\tilde \psi_2 \in L^2(0,T;H^1(\Omega^\ve_s))$, see \cite[Corollary 2.4, Lemma 2.2]{Girault} or \cite[Lemma 2.4]{Temam}, such that
\begin{gather*}
\begin{aligned}
\operatorname{div} \tilde \psi_1 &=\tilde q_1\quad \text{ in }\,  \Omega\times(0,T)\; , & \qquad\tilde \psi_1 &=0 \quad \text{ on } \, \partial \Omega \times(0,T)\; ,  \\
\operatorname{div} \tilde \psi_2 &=\tilde  q_2 \quad \text{ in }\,  \Omega_s^\ve\times(0,T)\; , &\qquad \tilde \psi_2\cdot \vect{n}&= 0 \quad \text{ on } \partial\Omega_s^\ve\times(0,T)\; ,  \\
\end{aligned}\\
\tilde \psi_2\times \vect{n} = - \tilde \psi_1\times \vect{n}\quad  \text{ on } \Gamma_{zs}^\ve\times(0,T)\; , 
\end{gather*}
and satisfy estimates
\begin{eqnarray*}
 \|\tilde \psi_1\|_{L^2(0,T;H^1(\Omega))} &\leq & C\|\tilde q_1\|_{L^2((0,T)\times\Omega)}
\leq C(\|q_1\|_{L^2((0,T)\times\Omega_z^\ve)} +\|q_2\|_{L^2((0,T)\times\Omega_a^\ve)} + \|q_3\|_{L^2((0,T)\times\Omega_{as}^\ve)})\; ,\\ 
 \|\tilde \psi_2\|_{L^2(0,T;H^1(\Omega_s^\ve))}& \leq & C(\ve)\|\tilde q_2\|_{L^2((0,T)\times \Omega_s^\ve)}\leq C(\ve)\|q_3\|_{L^2((0,T)\times\Omega_{as}^\ve)}\; .
 \end{eqnarray*}
 Notice that   $\tilde \psi_1\in L^2(0,T;H^1_0(\Omega))$ ensures $\tilde \psi_1\times \vect{n} \in L^2(0,T; H^{1/2}(\Gamma_{zs}^\ve))$ and, therefore, also the existence of $\tilde \psi_2$ with   $\tilde \psi_2\times \vect{n} = - \tilde \psi_1\times \vect{n}  \text{ on } \Gamma_{zs}^\ve\times(0,T)$, whereas  $t\in (0,T)$ plays  the role of a parameter.
 
We define $\psi=(\psi_1, \psi_2, \psi_3)$ by $\psi_1=(\tilde \psi_1 + \tilde \psi_2)|_{\Omega_z^\ve}$, \, $\psi_2=\tilde \psi_1|_{\Omega_a^\ve}$, \, $\psi_3= \tilde \psi_2|_{\Omega_{as}^\ve}$. The  definitions  of $\tilde \psi_1$ and $\tilde \psi_2$ imply that 
 $\psi_2\cdot \vect{n} = 0$ on $\partial \Omega\times(0,T)$ and $\psi_3\cdot \vect{n}=0$ on $(\partial \Omega_s^\ve\cap \partial \Omega)\times (0,T)$. 
The properties  $\tilde \psi_1\in L^2(0,T; H^1(\Omega))$ and $\tilde \psi_2\cdot \vect{n} = 0$ on $\partial \Omega_s^\ve\times(0,T)$ and $\tilde \psi_2\times \vect{n} = -\tilde \psi_1 \times \vect{n}$ on $ \Gamma_{zs}^\ve\times(0,T)$ ensure
 $\psi_1\cdot \vect{n}=\psi_2\cdot \vect{n} $ on $\Gamma_z^\ve\times(0,T)$  and $\psi_1\times \vect{n}=0 $ on $\Gamma_{zs}^\ve\times(0,T)$, 
 as well as $\psi_1\cdot \vect{n} = (\psi_2 + \psi_3)\cdot \vect{n}$ on  $\Gamma_{as}^\ve\times(0,T)$ and 
 $\psi_3\cdot \vect{n} =0 $ on $\Gamma_{aw}^\ve\times(0,T)$. The regularity of $\tilde \psi_1$ and $\tilde \psi_2$ implies also that 
 $\psi_1 \in L^2(0,T; H^1(\Omega_z^\ve))$, $\psi_2 \in L^2(0,T; H(\text{div}, \Omega_a^\ve))$ and $\psi_3 \in L^2(0,T;H(\text{div}, \Omega_{as}^\ve))$. Thus  
$\psi=(\psi_1, \psi_2, \psi_2)\in  W^\ve\cap L^2(0,T;\mathcal V^\ve)$ with $W^\ve=L^2(0,T;H^1(\Omega_z^\ve))\times L^2(0,T;H^1(\Omega_{a}^\ve))\times L^2(0,T;H^1(\Omega_{as}^\ve))$ and  $W^\ve\cap L^2(0,T;\mathcal V^\ve)$ is continuously embedded in $L^2(0,T;\mathcal V^\ve)$ and for $\psi \in W^\ve\cap L^2(0,T;\mathcal V^\ve)$ we have  
\[
\|\psi\|_{L^2(0,T;\mathcal V^\ve)}\leq C \|\psi\|_{W^\ve} \leq  C(\ve) (\|\tilde q_1\|_{L^2((0,T)\times\Omega)}+ \|\tilde q_2\|_{L^2((0,T)\times\Omega_s^\ve)} )\leq C(\ve) \|q\|_{L^2(0,T;\mathcal P^\ve)} \; .
\]
Here we used that  $q_1=(\tilde q_1 + \tilde q_2)|_{\Omega_z^\ve}$, $q_2=\tilde q_1|_{\Omega_a^\ve}$ and $q_3=\tilde q_2|_{\Omega_{as}^\ve}$. Thus, the definition of $b(\cdot, \cdot)$ yields  
\begin{eqnarray*}
& b(\psi, q)  &=  \langle\tilde q_1 + \tilde q_2, \operatorname{div}(\tilde \psi_1  + \tilde \psi_2)\rangle_{\Omega_z^\ve, T}
 + \langle \tilde q_1, \operatorname{div} \tilde \psi_1\rangle_{\Omega_a^\ve, T} + \langle \tilde q_2, \operatorname{div} \tilde \psi_2 \rangle_{\Omega^\ve_{as}, T} \\
&& = \| q_1\|^2_{L^2((0,T)\times\Omega_z^\ve)}+\| q_2\|^2_{L^2((0,T)\times\Omega_{a}^\ve)}
+
\| q_3\|^2_{L^2((0,T)\times\Omega_{as}^\ve)}
 \geq C_1(\ve)( \| \tilde \psi_1\|_{L^2(0,T;H^1(\Omega))}\\
 && + \| \tilde \psi_2\|_{L^2((0,T)\times\Omega_{s}^\ve)}
 + \|\operatorname{div} \tilde \psi_2\|_{L^2((0,T)\times\Omega_{s}^\ve)} 
 +\ve\| \nabla \tilde \psi_2\|_{L^2((0,T)\times\Omega_{s}^\ve)})
 ( \| q_1\|_{L^2((0,T)\times\Omega_z^\ve)}\\
 && +\| q_2\|_{L^2((0,T)\times\Omega_{a}^\ve)}+
\| q_3\|_{L^2((0,T)\times\Omega_{as}^\ve)} ) \geq  C_2(\ve) \|\psi\|_{L^2(0,T;\mathcal V^\ve)}\| q\|_{L^2(0,T;\mathcal P^\ve)} \; ,
\end{eqnarray*}
and $b(\cdot,\cdot)$ satisfies the inf-sup condition
\begin{equation}\label{inf-sup}
\inf\limits_{q\in L^2(0,T; \mathcal P^\ve), q\neq 0} \sup\limits_{\psi\in L^2(0,T;\mathcal V^\ve), \psi\neq 0} \frac{b(\psi,q)}{\|\psi\|_{L^2(0,T;\mathcal V^\ve)} \|q\|_{L^2(0,T;\mathcal P^\ve)}} \geq C>0 \; .
\end{equation}
The regularity of  $c^\ve_s$, $c_a^\ve$ and   H\"older's inequality imply the boundedness  in  $L^2(0,T;\mathcal V^\ve)$ of  $f^\ve(\cdot)$
\begin{eqnarray*}
|f^\ve(\psi)|
&\leq& C \ve(\|c_a^\ve\|_{L^2((0,T)\times\Gamma_{zs}^\ve)}+ \|c_s^\ve\|_{L^2((0,T)\times\Gamma_{zs}^\ve)})\|\psi_2\|_{L^2((0,T)\times\Gamma_{zs}^\ve)}\\
 &\leq& C\ve^{1/2} (\|c_a^\ve\|_{L^2((0,T)\times\Gamma_{zs}^\ve)}+ \|c_s^\ve\|_{L^2((0,T)\times\Gamma_{zs}^\ve)} ) \|\psi\|_{L^2(0,T;\mathcal V^\ve)}\; .
\end{eqnarray*}
The uniform boundedness of  $V_D^\ve$ in $L^2(\Omega_a^\ve)$ and  estimate \eqref{a_bound} ensure the boundedness of the linear form
$a^\ve(\bar V_D^\ve,\cdot)$. 
Combining all estimates and applying   Corollary I.4.1 in \cite{Girault} imply the existence of  a unique solution $(u^\ve, p^\ve) \in L^2(0,T;\mathcal V^\ve) \times L^2(0,T;\mathcal P^\ve)$ of  \eqref{Stokes} for given $c_{a}^\ve, c^\ve_{s} \in L^2((0,T)\times \Gamma_{zs}^\ve)$.

For the proof of  a priori estimates \eqref{ap_estim} we start with the estimates for $\vect{v}_z^\ve$, $\vect{v}_a^\ve$ and $\vect{v}_{sp}^\ve$. 
We  consider $\psi_1={\vect v}_z^\ve$ in $\Omega_z^\ve\times(0,T)$ and $\psi_2={\vect v}_{a}^\ve- V_{D}^\ve$ in 
$\Omega_{a}^\ve\times(0,T)$,  and $\psi_3={\vect v}_{sp}^\ve$ in  $\Omega_{as}^\ve\times(0,T)$
as test functions in \eqref{Stokes} and, using    the divergence-free property,  obtain  
\begin{align}\label{estim_v}
  \begin{aligned}
\ve^2  \|\operatorname{S} {\vect  v}^\ve_z\|^2_{L^2((0,T)\times\Omega_z^\ve)}+ 
   \|{\vect v}^\ve_{sp}\|^2_{L^2((0,T)\times\Omega_{as}^\ve)}  +
  \|{\vect v}^\ve_a\|^2_{L^2((0,T)\times\Omega_a^\ve)}  + \ve \|{\vect v}^\ve_a\cdot\vect{n}\|^2_{L^2((0,T)\times\Gamma^\ve_{zs})} \\
 \leq   C  \ve (\|c^\ve_s\|^2_{L^2((0,T)\times\Gamma^\ve_{zs})} +\|c^\ve_a\|^2_{L^2((0,T)\times\Gamma^\ve_{zs})})
    + C \|V_{D}^\ve\|^2_{L^2(\Omega_a^\ve)}\; .
\end{aligned}
\end{align}
This  together with  $ \|V_{D}^\ve\|_{L^2(\Omega^\ve_a)} \leq C \|V_D\|_{H^{1}(\Omega)}$,   inequality \eqref{Korn2} and  ${\vect v}_z^\ve \times \vect{n} = 0 $  on $\Gamma_{zs}^\ve\times (0,T)$ imply   the estimates  for $\vect{v}_i^\ve$, with $i=z, a,$ or $sp$, stated in \eqref{ap_estim}.\\ 
For $c_{a,i}^\ve$ and $c^\ve_{s,i}$ and corresponding $\vect{v}_i^\ve$ with $i=1,2$, the linearity of the problem, the coercivity of $a^\ve(\cdot, \cdot)$ and the  boundedness of $f^\ve(\cdot)$ give
\begin{equation}\label{veloc_uniq_estim}
 \|\vect{v}_1^\ve - \vect{v}_2^\ve\|^2_{L^2(0,T;\mathcal V^\ve_d)} \leq C \ve(\|c_{a,1}^\ve- c_{a,2}^\ve\|^2_{L^2((0,T)\times\Gamma_{zs}^\ve)} +  \|c_{s,1}^\ve- c_{s,2}^\ve\|^2_{L^2((0,T)\times\Gamma_{zs}^\ve)}) \; .
\end{equation}
This will ensure  the uniqueness of a weak solution of  the coupled problem  \eqref{velosity}-\eqref{BC3velosity}. 

\noindent
By  Lemma~4 in  \cite{Tartar1980}  there exists a restriction operator $R^\ve_{\Omega_a^\ve} : H^1_0(\Omega) \to H^1_0(\Omega_{a}^\ve)$
\begin{equation}\label{restriction_a}
\begin{cases}
\psi\in H^1_0(\Omega_a^\ve)  &\Longrightarrow \, \, \, R^\ve_{\Omega_a^\ve} \psi = \psi \; , \\
\text{ div } \psi =0   &\Longrightarrow  \, \, \,  \text{div } R^\ve_{\Omega_a^\ve} \psi = 0 \; , \\
\|R^\ve_{\Omega_a^\ve} \psi \|_{L^2(\Omega^\ve_a)} &\leq  \, \, \,   C ( \|\psi\|_{L^2(\Omega)}+  \ve \|\nabla \psi\|_{L^2(\Omega)}) \; ,\\
\ve \|\nabla R^\ve_{\Omega_a^\ve} \psi \|_{L^2(\Omega^\ve_a)} &\leq  \, \,  \, C ( \|\psi\|_{L^2(\Omega)}+  \ve \|\nabla \psi\|_{L^2(\Omega)}) \; .
\end{cases}
\end{equation} 
As in   \cite[Theorem~2.3]{Allaire1989}, to define    $R^\ve_{\Omega_s^\ve} : H^1_0(\Omega) \to V(\Omega_{s}^\ve)$, where $V(\Omega_{s}^\ve)= \{ v \in H^1_0(\Omega_s^\ve): v\times \vect{n} =0 \text{ on }  \Gamma_{as}^\ve \}$, such that 
\begin{equation}\label{restriction_s}
\begin{cases}
\psi\in V(\Omega_s^\ve)  &\Longrightarrow \, \, \, R^\ve_{\Omega_s^\ve} \psi = \psi \;, \\
\text{ div } \psi =0   &\Longrightarrow  \, \, \,  \text{div } R^\ve_{\Omega_s^\ve} \psi = 0 \; , \\
\|R^\ve_{\Omega_s^\ve} \psi \|_{L^2(\Omega^\ve_s)} &\leq  \, \, \,   C ( \|\psi\|_{L^2(\Omega)}+  \ve \|\nabla \psi\|_{L^2(\Omega)}) \; , \\
\ve \|\nabla R^\ve_{\Omega_s^\ve} \psi \|_{L^2(\Omega^\ve_s)} &\leq  \, \,  \, C ( \|\psi\|_{L^2(\Omega)}+  \ve \|\nabla \psi\|_{L^2(\Omega)}) \; ,
\end{cases}
\end{equation} 
we consider for $u \in H^1(Y)$,  as in   \cite[Lemma~3.4]{Allaire1989},   a linear operator $Q: H^1(Y) \to H^1(Y)$, such that $Qu=0$ in $Y_a$, and a modified problem 
\begin{equation}\label{restriction_Os}
\begin{cases}
(w,q)\in H^1(Y_s)\times L^2(Y_s)/\mathbb R \; , \\
-\eta \,  \Delta w + \nabla q = \Delta u  &\text{ in }   Y_z\cup Y_{as} \; , \\
\operatorname{div} w = \operatorname{div} u + \frac 1 {|Y_s|} \int_{Y_a} \operatorname{div} u \, dy & \text{ in } Y_s \; , \\
w\times \vect{n} =0 \;  , \quad   [w\cdot \vect{n}]=0\; ,\quad  
[-2\eta \, (\operatorname{S} w \, \vect{ n} )\cdot \vect{n} + q]=0 & \text{ on }  \Gamma_{as} \; , \\
w = Qu + \left[\frac{\zeta_i}{ \langle \zeta_i, 1 \rangle_{\partial Y_i}}  \int_{\partial Y_i} (u-Qu)  \cdot e_i dy\right] e_i & \text{ on } \partial Y_{s,i}\cap \partial Y \; , \\
w=0  & \text{ on } \partial Y_{s}\setminus \partial Y \; ,
\end{cases}
\end{equation}
where  $\partial Y_i$, with  $i=-3,-2,-1, 1,2,3$, are the six  faces of the cube $Y$, such that   $\partial Y_k$, $\partial Y_{-k}$ are orthogonal to the unit vector $e_k$, with $k=1,2,3$. Here $\zeta_i \in C^\infty(\overline Y)$, with  $i=-3,-2,-1, 1,2,3$, satisfy  $\zeta_i \geq 0$, $\zeta_i \equiv \hspace{-0.4 cm}/   \hspace{0.2 cm} 0$   on $\partial Y_i$, $\zeta_i \equiv 0$ in $\overline Y_a$ and on $\partial Y_j$ for $j \neq i$,  and $\zeta_k|_{\partial Y_k} \equiv \zeta_{-k}|_{\partial Y_{-k}}$ for $k=1,2,3$.    The existence of such $\zeta_i$ is ensured by the geometrical structure
of $Y_s$.  
 Notice that the construction of the boundary conditions on $\partial Y_s$ ensures that $\int_{Y_s} \operatorname{div} w \, dy =\int_{\partial Y_s} w\cdot \vect{n} \, dy$. The existence of a unique solution of \eqref{restriction_Os} can be shown by applying the abstract theory of mixed problems, similar as for \eqref{Stokes_2}.
 Then  $R_{Y_s}$, given by  $R_{Y_s} u = w$, belongs to $\mathcal L (H^1(Y); H^1(Y_s))$ and in each cell $\ve Y^k \subset \Omega$,  with $k\in \mathbb Z^3$, we define  $R^\ve_{\Omega_s^\ve} u=R_{Y_s}(u(\ve y))(x/\ve)$ for 
$y \in Y^k$ and $x\in \ve Y^k_s$, respectively.\\

Equations  \eqref{velosity} imply that $\nabla p_s \in L^2(0,T;H^{-1}(\Omega_s^\ve))$
and $\nabla  p_a \in L^2(0,T;H^{-1}(\Omega_a^\ve))$.
We define the  extensions of $\nabla p^\ve_s$  and $\nabla p^\ve_a$ into $\Omega$ using the duality argument and
consider  $\mathcal F^\ve_s\in L^2(0,T;H^{-1}(\Omega))$ and   $\mathcal F^\ve_a \in L^2(0,T;H^{-1}(\Omega))$ given by 
\begin{eqnarray*}
\langle \mathcal F^\ve_s, \psi\rangle_{H^{-1}, H^1} = \langle \nabla p_s^\ve, R^\ve_{\Omega_s^\ve}\psi
\rangle_{L^2(0,T;H^{-1}(\Omega_s^\ve)), L^2(0,T;H^1_0(\Omega_s^\ve))} && \text{ for } \psi \in L^2(0,T; H^1_0(\Omega))\; ,\\
\langle \mathcal F^\ve_a, \psi\rangle_{H^{-1}, H^1} =  \langle \nabla p_a^\ve, R^\ve_{\Omega_a^\ve}\psi\rangle_{L^2(0,T;H^{-1}(\Omega_a^\ve)), L^2(0,T; H^1_0(\Omega_a^\ve))} && \text{ for } \psi \in L^2(0,T; H^1_0(\Omega)) \; ,
\end{eqnarray*}
where $\langle \cdot, \cdot \rangle_{H^{-1}, H^1}:= \langle \cdot, \cdot\rangle_{L^2(0,T;H^{-1}(\Omega)), L^2(0,T;H^1_0(\Omega))}$.
Applying  the estimates for $R^\ve_{\Omega_s^\ve}(\psi)$ and $R^\ve_{\Omega_a^\ve}(\psi)$ in \eqref{restriction_a} and \eqref{restriction_s}, we obtain from  \eqref{velosity}  for $\psi \in L^2(0,T;H^1_0(\Omega))$ that 
\begin{align}\label{estim_F_s}
  \begin{aligned}
\langle \mathcal F^\ve_s, \psi\rangle_{H^{-1},H^1} & = 
\langle \ve^2 \eta \nabla \vect{v}_z^\ve, \nabla (R^\ve_{\Omega^\ve_s} \psi) \rangle_{\Omega_z^\ve,T} + \langle K^{-1}_{sp}\vect{v}^\ve_{sp}, R^\ve_{\Omega^\ve_{s}} \psi \rangle_{\Omega_{as}^\ve,T} \\
& \leq  C \left(\ve \|\nabla \vect{v}_z^\ve\|_{L^2((0,T)\times\Omega_z^\ve)} +  \|\vect{v}_{sp}^\ve\|_{L^2((0,T)\times\Omega_{as}^\ve)}\right)\left( \|\psi\|_{L^2(\Omega_T)}+ \ve  \|\nabla \psi\|_{L^2(\Omega_T)}\right)\; ,\,\, 
\end{aligned}
\end{align}
where $\Omega_T= (0,T)\times\Omega$,  and 
\begin{eqnarray}\label{estim_F_a}
\langle \mathcal F^\ve_a, \psi\rangle_{H^{-1}, H^1}= 
\langle  K^{-1}_{a} \vect{v}_a^\ve, R^\ve_{\Omega^\ve_a} \psi \rangle_{\Omega_a^\ve,T} \leq  C \|\vect{v}_a^\ve\|_{L^2((0,T)\times\Omega_a^\ve)} \left( \|\psi\|_{L^2((\Omega_T)}+ \ve  \|\nabla \psi\|_{L^2(\Omega_T)}\right) \; .
\end{eqnarray}
Thus for each $\ve >0$ we  have that  $\mathcal F^\ve_a$ and $\mathcal F^\ve_s$ are bounded functionals on $L^2(0,T;H^1_0(\Omega))$.\\

For  $\psi\in L^2(0,T;H^1_0(\Omega_l^\ve))$ with $\psi\times \vect{n}=0$ on $\Gamma_{as}^\ve\times (0,T)$,  where $l=a,s$,  the properties of $R^\ve_{\Omega_l^\ve}$ imply $R^\ve_{\Omega_l^\ve}(\psi)=\psi$  and 
$\mathcal F^\ve_l|_{\Omega_l^\ve}= \nabla p^\ve_l$. We have also
$\text{div }R^\ve_{\Omega_l^\ve}(\psi)=0$, for $l=a,s$,  provided $\text{div } \psi =0$, and the orthogonality property ensures that 
$\mathcal F^\ve_s$ and $\mathcal F^\ve_a$ are the gradients  with respect to $x$ of functions in $L^2((0,T)\times\Omega)$. It means that 
 $\mathcal F^\ve_s$ and $\mathcal F^\ve_a$ are continuations of $\nabla p^\ve_a$ and $\nabla p^\ve_s$ to $\Omega$, respectively, and 
\[ 
\mathcal F^\ve_s = \nabla P^\ve_s\; ,  \quad \mathcal F^\ve_a = \nabla P^\ve_a \qquad \text{  with } \quad P^\ve_s, P^\ve_a \in L^2((0,T)\times\Omega)/\mathbb R\; .
\]
We have also an explicit formula for the extension $P_l^\ve$, with  $l=s$ or $a$, see \cite{Allaire1989, hornung},  for $t\in (0,T)$
\[
P_l^\ve(t, x)=
\begin{cases}
p_l^\ve(t, x) & \text{ in } \, \Omega_l^\ve\; , \\
\frac 1 {|\ve Y_l^{k}|}\int_{\ve Y_l^{k} } p^\ve_l(t,x) dx  & \text{ in } \ve Y^k \setminus \ve\overline {Y_l^k} \text{ for  } k\in \mathbb Z^3 \text{ such that }  \ve Y^k\cap \Omega\neq\emptyset \; .
\end{cases} 
\]
Applying now estimates  \eqref{estim_F_s}, \eqref{estim_F_a} together with \eqref{estim_v} and \eqref{Korn2}, and using the estimate of $L^2$-norm by $H^{-1}$-norm, see \cite{Girault, Temam}, give 
\begin{align*}
 ||P^\ve_s||_{L^2((0,T)\times\Omega)/\mathbb R} &\leq C_1  ||P^\ve_s||_{L^2(0,T;H^{-1}(\Omega))} \leq C_2 (\ve\|\nabla\vect{v}^\ve_z\|_{L^2((0,T)\times\Omega_z^\ve)} + \| \vect{v}^\ve_{sp}\|_{L^2((0,T)\times\Omega_{as}^\ve)} )\\
&\leq C_3  \ve^{1/2}( \|c^\ve_s\|_{L^2((0,T)\times\Gamma^\ve_{zs})}+\|c^\ve_a\|_{L^2((0,T)\times\Gamma^\ve_{zs})})
    + C_4 \|V^\ve_{D}\|_{L^2(\Omega_a^\ve)}\;  , \\
 ||P^\ve_a||_{L^2((0,T)\times\Omega)/\mathbb R} & \leq C_1  ||P^\ve_a||_{L^2(0,T;H^{-1}(\Omega))} \leq C_2 
\| \vect{v}^\ve_a\|_{L^2((0,T)\times\Omega_a^\ve)}\\
&\leq C_3  \ve^{1/2}( \|c^\ve_s\|_{L^2((0,T)\times\Gamma^\ve_{zs})}+\|c^\ve_a\|_{L^2((0,T)\times\Gamma^\ve_{zs})})
    + C_4 \|V_{D}^\ve\|_{L^2(\Omega_a^\ve)}\; .
\end{align*}
The last estimates together with  $ \|V_{D}^\ve\|_{L^2(\Omega_a^\ve)} \leq C \|V_D\|_{H^{1}(\Omega)}$ and the definition of $P^\ve_s$, $P^\ve_a$ ensure  \eqref{ap_estim_extension} and $L^2$-estimates for $p_s^\ve$ and $p_a^\ve$  in \eqref{ap_estim}.
\end{proof}
%%%%%%%%%%%%%%%%%%%%%%

%%%%%%%%%%%%%%

\subsection{Existence and estimates for $c_s^\ve$, $c_a^\ve$,   $\vartheta^\ve_{f,a}$, $\vartheta_{b,a}^\ve$,  $\vartheta_{f,s}^\ve$ and $\vartheta_{b,s}^\ve$} 

Using classical results \cite{Acerbi,CiorPaulin99}
we   can extend the domain of definition of solute concentrations ${c}^\varepsilon_l$  from a connected domain $\Omega_{l}^{\varepsilon}$  to $\Omega$,  where $l=a,s$. 
\begin{lemma}\label{extension}
1. For $c_l \in W^{1,p}(Y_l)$, with $1<p<\infty$, there exists an extension $\tilde c_l$ to $Y$, such that
$$
\|\tilde c_l\|_{L^p(Y)} \leq \Xi\|c_l\|_{L^p(Y_l)}\,\,\,\,\text{and}\,\,\,\,\|\nabla \tilde c_l\|_{L^p(Y)} \leq \Xi\|\nabla c_l\|_{L^p(Y_l)} .
$$
2. There exists an extension ${\tilde{c}^{\varepsilon}_l}$ of ${c}_l^{\varepsilon}$ from
$W^{1,p}(\Omega_{l}^{\varepsilon})$ into $W^{1,p}(\Omega)$ such that 
\[
\|{\tilde{c}_l^{\varepsilon}}\|_{L^{p}(\Omega)}\leq\Xi\|{c}_l^{\varepsilon}\|_{L^{p}(\Omega_{l}^{\varepsilon})},\mbox{ and }
\|\nabla {\tilde{c}^{\varepsilon}_l}\|_{L^{p}(\Omega)}\leq\Xi\|\nabla{c}_l^{\varepsilon}\|_{L^{p}(\Omega_{l}^{\varepsilon})},
\]
where the constant $\Xi$ depends on $Y$ and  $Y_l$ only. 
\end{lemma}

 Due to the geometric assumptions on  $\Omega_a^\ve$, 
 holes in the domain do not touch each other,   have smooth boundary and do not touch the boundary $\partial \Omega$, i.e. 
 $\Gamma_{zs}^\ve \cap \partial \Omega= \emptyset$. Therefore,  
 classical extension results \cite{Acerbi, CiorPaulin99}
apply to $c^{\ve}_a$.
Due to the structural assumptions,  $\Omega^\ve_s$  is a connected domain in $\mathbb R^3$ with Lipschitz-continous boundary $\partial \Omega^\ve_s$. 
The geometrical assumptions on   $\Omega^\ve_s$  ensure also that it is sufficient to extend $c^\ve_s$ by reflection  in tangential directions  near the boundary $\partial\Omega_s^\ve \cap \partial\Omega$. Therefore the extension results  apply also for $c^\ve_s$, see  \cite{Acerbi, Ptashnyk_hairy}, and the extension operator   is defined globally in $\Omega$.

For $c_l^{\varepsilon}
 \in H^1(0,T; H^{1}(\Omega_l^{\varepsilon}))$ we define
 $\hat   c^{\varepsilon}_l(\cdot,t):={\tilde{c}^{\varepsilon}_l}(\cdot,t)$ a.e.~in time.
Since the extension operator is linear and bounded  and $\Omega_{l}^{\varepsilon}$ does not depend on $t$, 
we obtain   $\hat c^{\varepsilon}_l \in H^1(0,T; H^{1}(\Omega))$
and
\[
\|\partial_{t}{\hat c}_l^{\varepsilon}\|_{L^{2}(\Omega_T)}
\leq\Xi\|\partial_{t} {c}_l^{\varepsilon}\|_{L^{2}((0,T)\times\Omega_{l}^{\varepsilon})}, \quad  \|\partial_{t}\nabla {\hat c}_l^{\varepsilon}\|_{L^{2}(\Omega_T)}
\leq\Xi\|\partial_{t} \nabla {c}_l^{\varepsilon}\|_{L^{2}((0,T)\times\Omega_{l}^{\varepsilon})},
\]
where $l=a,s$. In the sequel, we shall identify ${c}^{\varepsilon}_l$ with its extension $\hat c_l^{\varepsilon}$.\\

%%%%%%%%%%%%%%%%%%%%%%%%%%%
%%%%%%%%%%%%%%%%%%%%%%%%%%%%%

\begin{theorem}\label{Existence_c}
 Under Assumption~\ref{assump}  there exists a nonnegative unique weak solution $(c_a^\ve, c_s^\ve, \vartheta_{f,a}^\ve, \vartheta_{b,a}^\ve,  \vartheta_{f,s}^\ve, \vartheta_{b,s}^\ve)$  of 
 \eqref{Eq_Conc}  and   \eqref{transporter_eq} with boundary conditions \eqref{Cell_membrBC}, \eqref{Cell_BC_other}, and \eqref{ExternalBC_N}  and  initial conditions \eqref{init_cond}  such  that  
$c_l^\ve \in H^1(0,T; H^1(\Omega_l^\ve))\cap L^{\infty}((0,T)\times \Omega_l^\ve)$, 
$\vartheta_{j,l}^\ve \in W^{1, \infty}(0,T; L^\infty(\Gamma^\ve_{zs}))$, with $l=a,s$, $j=f,b$,
and satisfies the   estimates 
\begin{align}\label{energy_estim}
  \begin{aligned}
 & \|c^\ve_l \|_{L^\infty(0,T; L^2(\Omega^\ve_l))}+    \|\nabla c^\ve_l \|_{L^2((0,T)\times\Omega^\ve_l)}
 +  \ve^{\frac 12}\|c^\ve_l \|_{L^2((0,T)\times\Gamma^\ve_{zs})}
 \leq C\; , \\
& \|c^\ve_l\|_{L^\infty((0, T)\times \Omega)} \leq C \; , \qquad 
 \|\vartheta^\ve_{j,l} \|_{L^\infty((0, T)\times \Gamma_{zs}^\ve)} \leq C\; ,
\end{aligned}
\end{align}
and 
\begin{align}\label{time_estim}
\begin{aligned}
  \|\partial_t c^\ve_l \|_{L^\infty(0,T; L^2(\Omega^\ve_l))}+ \|\partial_t\nabla c^\ve_l \|_{L^2((0,T)\times \Omega^\ve_l)} 
  +\ve^{\frac 12} \|\partial_t c^\ve_l \|_{L^2((0,T)\times \Gamma^\ve_{zs})}+\ve^{\frac 12} \|\partial_t \vartheta^\ve_{j,l}\|_{L^2((0,T)\times\Gamma^\ve_{zs})}
 \leq C
 \;,  \\
 \|\partial_t \vect{v}^\ve_a \|_{L^2((0,T)\times \Omega^\ve_a)} +  \|\partial_t \vect{v}^\ve_s \|_{L^2((0,T)\times \Omega^\ve_s)}+
 \ve \|\partial_t\nabla  \vect{v}^\ve_z \|_{L^2((0,T)\times \Omega^\ve_z)} 
 \leq C \;,
\end{aligned}
\end{align}
with $j=f,b$ and $l=a,s$,  and  the  constant $C$ depends on $M$ and is independent of $\ve$.

\end{theorem}

\begin{proof} 
The existence of a solution will be proven by showing  the existence of a fix point of  the operator  $\mathcal B$   defined on  
$\big(C([0,T]; H^s({\Omega}^{\varepsilon}_a))\cap L^\infty((0,T)\times \Omega_a^\ve)\big)\times \big(C([0,T]; H^s({\Omega}^{\varepsilon}_s))\cap L^\infty((0,T)\times \Omega_s^\ve)\big)$, with $1/2 < s<1$,
by  $(c^{\varepsilon,n}_a, c^{\ve,n}_s)=\mathcal B (c^{\varepsilon, n-1}_a, c^{\varepsilon, n-1}_s)$
given as solutions of \eqref{def_apop} and \eqref{def_symp}
with $(\vartheta^{\ve, n}_{f,l}, \vartheta^{\ve, n}_{b,l})$  solving  
\begin{eqnarray}\label{transporter_eq_Iter}
 \begin{cases}
  \partial_t \vartheta_{f,l}^{\ve, n}&= R_{l}^\ve(t,x,\vartheta_{f,l}^{\ve, n})- \alpha_{l}^\ve(t,x) c^{\ve, n-1}_{l} \vartheta_{f,l}^{\ve, n} + \beta_{l}^\ve(t,x) \vartheta_{b,l}^{\ve, n}-\gamma_{f,l}^\ve(t,x) \vartheta_{f,l}^{\ve, n} \quad
\text{ on } \Gamma^\ve_{zs}\times(0,T)\; ,  \\
 \partial_t \vartheta_{b,l}^{\ve, n} & = \phantom{ R_{l}^\ve(t,x,\vartheta_{f,l}^{\ve, n})- } \, \, \alpha_{l}^\ve(t,x) c^{\ve, n-1}_{l} \vartheta_{f,l}^{\ve, n} - \beta_{l}^\ve(t,x) \vartheta_{b,l}^{\ve, n} -\gamma_{b,l}^\ve(t,x) \vartheta_{b,l}^{\ve, n} \, \quad \text{ on } 
\Gamma^\ve_{zs}\times(0, T)\; , 
   \end{cases} \hspace{-1.5 cm }
 \end{eqnarray}
where  $l=a,s$. 
For a given nonnegative $c^{n-1, \varepsilon}_l\in  C([0,T]; H^{s}(\Omega^{\varepsilon}_l))\cap L^{\infty}((0,T)\times \Omega_l^\ve)$, with   $l=a,s$,  due to  the Lipschitz-continuity  of the right-hand side of   system   \eqref{transporter_eq_Iter}, 
there exists a unique solution 
$\vartheta_{j,l}^{\varepsilon, n} \in  C^1([0,T]; L^2(\Gamma^{\varepsilon}_{zs}))$  of  \eqref{transporter_eq_Iter}, with $j=f,b$ and $l=a,s$.
Using that initial values are nonnegative,  function $R_l^\ve(t,x,\xi)$ is nonnegative for nonnegative  $\xi$, and $c^{\varepsilon, n-1}_l(t,x)\geq 0$ a.e.~on $[0,T]\times \Gamma^{\varepsilon}_{zs}$,   we obtain $\vartheta_{j,l}^{n,\varepsilon}(t,x)\geq 0$ a.e~on $[0,T]\times \Gamma^{\varepsilon}_{zs}$ for $j=f,b$ and $l=a,s$. 
Adding  equations in   \eqref{transporter_eq_Iter} yields
\begin{eqnarray*}
\partial_t ( \vartheta_{f,l}^{\ve, n} + \vartheta_{b,l}^{\ve, n}) = R_l^\ve(t,x,\vartheta_{f,l}^{\ve, n}) - \gamma_{f,l}^\ve(t,x)  \vartheta_{f,l}^{\ve, n} - \gamma_{b,l}^\ve(t,x)  \vartheta_{b,l}^{\ve, n}\; , && \text{ for } t \in [0,T]  \text{ and  a.a }  x\in\Gamma_{zs}^\ve \; ,
\end{eqnarray*}
where $l=a,s$.  Considering the Lipschitz continuity of $R_l^\ve$,   the nonnegativity of  $\vartheta_{j,l}^{\varepsilon, n}$, the boundedness of initial conditions 
and applying   Gronwall's inequality,  we obtain the   boundedness  of   $ \vartheta_{j,l}^{\ve, n}$ a.e. on $[0,T]\times \Gamma^\ve_{zs}$,  with  $j=f,b$ and $l=a,s$.
The boundedness of $c^{\ve, n-1}_l$ and $\vartheta_{j,l}^{\ve, n}$ implies also the boundedness of $ \partial_t \vartheta_{j,l}^{\ve,n}$ a.e.
on $[0,T]\times \Gamma^\ve_{zs}$, whereas  $j=f,b$ and $l=a,s$. 

Using Galerkin's method and a priori estimates similar to those shown below,  we obtain the existence of a weak nonnegative solution  $c^{\varepsilon, n}_l\in H^1(0,T;H^1(\Omega^{\varepsilon}_l))\cap L^{\infty}((0,T)\times \Omega_l^\ve)$, for $l=s,a$, see \cite{Ladyzhenskaja}. 
The embedding   $H^1(0,T;H^1(\Omega^{\varepsilon}_l)) \subset C([0,T]; H^s(\Omega_l^\ve))$, for $1/2 <s <1$, 
 is compact and, by virtue of the Schauder theorem, there exists a fix point of $\mathcal B$, a solution  $(c^{\varepsilon}_l, \vartheta_{j,l}^{\varepsilon})$ of the microscopic problem,  where $j=f,b$ and $l=a,s$. 
In addition, we obtain boundedness of $c^{\varepsilon}_l(t,x)$ a.e.  in $[0,T]\times \Omega_l^\ve$ and on  $[0,T]\times \Gamma^\ve_{zs}$, and   $c^{\varepsilon}_l(t,x)\geq 0$ a.e. in  $[0,T]\times {\Omega_l^\ve}$  and  on $[0,T]\times {\Gamma_l^\ve}$, together with 
$\vartheta_{j,l}^{\varepsilon}(t,x)\geq 0$ a.e. on $[0,T]\times \Gamma_{zs}^\ve$, where $j=f,b$ and $l=a,s$.
This ensures also the boundedness of  $ \partial_t \vartheta_{j,l}^{\ve}$   a.e.
on $[0,T]\times \Gamma^\ve_{zs}$ and the uniform in $\varepsilon$ boundedness of  $\vartheta_{j,l}^{\ve}$   a.e.
on $[0,T]\times \Gamma^\ve_{zs}$,  and thus the last estimate in \eqref{energy_estim}.\\

Now we shall prove  the non-negativity of $c_l^{\ve, n}$, and thus  of $c_l^\ve$,  and the a priori estimates for $c_l^{\ve, n}$, and therefore for $c_l^\ve$,~uniformly~in~$\ve$, where $l=a,s$.
To show the non-negativity of $c_{l}^{\ve,n}$, with $l=a,s$ we take   $c_{l,-}^{\ve}=\min\{0, c_{l}^{\ve,n}\}$  as test functions in \eqref{def_apop} and \eqref{def_symp}, respectively.  Considering 
the assumptions on  the coefficients and functions $F_l^\ve$,  and the  nonnegativity of   $\vartheta^{\ve,n}_{j,l}$, with $l=a,s$ and $j=f,b$,
terms on the right-hand side in  \eqref{def_apop} and \eqref{def_symp} can be estimated by 
$
C_1 \|c_{l,-}^{\ve} \|_{L^2((0,T)\times\Omega_l^\ve)}.
$
Using   Gagliardo-Nirenberg-inequality, i.e. 
\begin{eqnarray}\label{GN}
\|v\|_{L^p(\Omega)} \leq C_G\big( \|v\|^{\lambda}_{L^r(\Omega)} \|\nabla v\|_{L^q(\Omega)}^{(1-\lambda)} + \|v\|_{L^1(\Omega)}\big), 
\end{eqnarray}
where $1\leq q <p$, $r\geq 1$ and  $\lambda\in [0,1]$ satisfies the relation $1/p = \lambda(1/r - 1/n) + (1-\lambda) 1/q$,  and considering extension of  $c^\ve_{l,-}$, defined in Lemma \ref{extension},  we can estimate the convective  term by 
\begin{eqnarray*}
&&|\langle H_M(\vect v_{l}^\ve) c^\ve_{l,-}, \nabla c^\ve_{l,-} \rangle_{\Omega^\ve_l, T}| \leq 
\frac 1{d_l}\|H_M(\vect v_{l}^\ve)\|^2_{L^\infty(0,T; L^4(\Omega^\ve_{l}))}\| c^\ve_{l,-}\|^2_{L^2(0,T; L^4(\Omega^\ve_l))} + \frac {d_l} 4\| \nabla c^\ve_{l,-}\|^2_{L^2((0,T)\times\Omega^\ve_l)}  \\
&&
\leq C  \Big( \|H_M(\vect v_{l}^\ve)\|^8_{L^\infty(0,T; L^4(\Omega^\ve_l))} + \|H_M(\vect v_{l}^\ve)\|^2_{L^\infty(0,T; L^4(\Omega^\ve_l))}\Big) \| c^\ve_{l,-}\|^2_{L^2((0,T)\times \Omega_l^\ve)} + 
\frac {d_l} 2 \| \nabla c^\ve_{l,-}\|^2_{L^2((0,T)\times\Omega^\ve_l)}.
\end{eqnarray*}
Considering  uniform boundedness of $\|H_M(\vect v_{l}^\ve)\|_{L^\infty(0,T; L^4(\Omega^\ve_l))}\leq C M$, the nonegativity of initial condition $c_l^0$ and applying  Gronwall's inequality  we obtain  
\[
\|c_{a,-}^{\ve}\|_{L^\infty(0,T;L^2(\Omega_a^\ve))} + \|c_{s,-}^{\ve}\|_{L^\infty(0,T;L^2(\Omega_s^\ve))}\leq 0\; ,
\]
which implies   the nonnegativity  of $c_l^{\ve,n}$ a.e. in $(0,T)\times{\Omega_l^\ve}$, where $l=a,s$. 
 For $c_l^{\ve,n} \in H^1(0,T; H^1(\Omega_l^\ve))$ we obtain also that  $c_l^{\ve, n} \geq 0$ a.e. on $[0,T]\times \Gamma_{zs}^\ve$. 
Considering  weak convergence of $\{c_l^{\ve,n}\}$ in $H^1(0,T; H^1(\Omega_l^\ve))$, as $n\to \infty$, we conclude that  $c_l^\ve(t,x)\geq 0$ a.e. in $[0,T]\times {\Omega_l^\ve}$ and a.e. on $[0,T]\times \Gamma_{zs}^\ve$. 

Taking $\varphi_1=c^\ve_a$ and $\varphi_2=c^\ve_s$ as test functions in 
 \eqref{def_apop} and \eqref{def_symp}, 
 using the nonnegativity of $\vartheta_{f,l}^\ve$ and of coefficients $\alpha_{l}^\ve$, with $l=a,s$,  the  first estimate in
\begin{eqnarray}\label{estim_boundary}
\begin{aligned}
& \ve \|c^\ve_l\|^2_{L^2(\Gamma^\ve_{zs})}  \leq C\big(\|c^\ve_l\|^2_{L^2(\Omega^\ve_l)} +
 \ve^2  \|\nabla c^\ve_l\|^2_{L^2(\Omega^\ve_l)} \big)\; ,\\
&  \ve \|c^\ve_l\|_{L^1(\Gamma^\ve_{zs})}  \leq C\big(\|c^\ve_l\|_{L^1(\Omega^\ve_l)} +
 \ve \|\nabla c^\ve_l\|_{L^1(\Omega^\ve_l)} \big)\; , 
 \end{aligned}
\end{eqnarray}
see \cite{AnnaMariya2008}, considering that   $\ve|\Gamma_{zs}^\ve| \leq C$ independently of $\ve$,  and applying H\"older's inequality, 
 we obtain 
\begin{eqnarray*}
 \|c^\ve_l(\tau)\|^2_{L^2(\Omega^\ve_l)} 
  + (d_l-\zeta\ve^2)   \|\nabla  c^\ve_l\|^2_{L^2((0,\tau)\Omega^\ve_l)}  & \leq &
 \frac 1{d_l} \|H_M({\vect v}^\ve_{l}) \|^2_{L^\infty(0,\tau; L^4(\Omega_l^\ve))} \|c^\ve_l\|^2_{L^2(0,\tau; L^4(\Omega^\ve_l))} \\
& + &C_1\|c^\ve_l\|^2_{L^2((0,\tau)\times\Omega^\ve_l)} 
  +   \|c^0_{l}\|^2_{L^2(\Omega^\ve_l)}+ C_2 \|\vartheta_{b,l-1}^\ve\|_{L^{\infty}((0,\tau)\times \Gamma_{zs}^\ve)} 
\end{eqnarray*}
 for $\tau \in (0, T]$, with $l=a,s$, whereas  $a-1:=s$ and $s-1:=a$. 
Using the extension of $c^\ve_l$ from $\Omega_l^\ve$ into $\Omega$, given by Lemma~\ref{extension},  and applying  inequality \eqref{GN} in the first term on the right-hand side  imply
\begin{eqnarray*}
 \|c^\ve_l(\tau)\|^2_{L^2(\Omega)} 
  + (d_l/4-\zeta\ve^2)   \|\nabla  c^\ve_l\|^2_{L^2((0,\tau)\times \Omega)}   \leq 
   \|c^0_{l}\|^2_{L^2(\Omega^\ve_l)}+ C_1 \|\vartheta_{b,l-1}^\ve\|_{L^{\infty}((0,T)\times \Gamma_{zs}^\ve)} \\
 + C_2 \big( \|H_M({\vect v}^\ve_{l}) \|^8_{L^\infty(0,\tau; L^4(\Omega_l^\ve))} + 
 \|H_M({\vect v}^\ve_{l}) \|^2_{L^\infty(0,\tau; L^4(\Omega_l^\ve))} 
 + 1\big)\|c^\ve_l\|^2_{L^2((0,\tau)\times\Omega_l^\ve)} \; . 
\end{eqnarray*}
 Then, for  all $\ve\leq \ve_0$, with some $\ve_0>0$,  and $\zeta$ such that 
$d_l/4-\zeta \ve_0^2\geq d_0 >0$, considering the regularity assumption on  initial data, the boundedness of $\vartheta_{b,l}^\ve$, where $l=a,s$, and applying 
 Gronwall's Lemma,  the last inequality,  combined with \eqref{estim_boundary}, implies first  estimate in \eqref{energy_estim}.\\

\noindent To show $L^\infty$-estimates  for the extension of $c_l^\ve$  into $\Omega$, given by Lemma~\ref{extension}, we shall apply  Theorem II.6.1 from \cite{Ladyzhenskaja} stating that inequality
$\|(c_l^\ve- S)_+\|_{L^\infty(0,T; L^2(\Omega))} + \|\nabla (c_l^\ve- S)_+\|_{L^2(\Omega_T)} \leq \zeta S \|\Omega_{l,S}(t) \|^{\tilde r}_{L^{\tilde q}(0,T)}$, for appropriate $\tilde r, \tilde q$, a positive constant $\zeta$ and $\Omega_{l,S}(t)=\{ x\in \Omega:  c_l^\ve(t,x) > S\}$ for a.a. $t \in (0,T)$, ensures the corresponding estimate for the $L^\infty$-norm of $c^\ve_l$.
 
We take $(c_a^\ve- S)_+$ and $(c_s^\ve- S)_+$ as test functions in  \eqref{def_apop} and \eqref{def_symp}. 
Considering the nonnegativity of $\alpha_a^\ve(t,x)$,  $c^\ve_{a}(t,x)$ and  $\vartheta_{f,a}^\ve(t,x)$ for $(t,x) \in (0,T)\times\Gamma_{zs}^\ve$, and  inequalities  \eqref{estim_boundary},  the boundary integrals  can be estimated by
\begin{eqnarray}\label{estim_bound_1}
&& \ve \langle
\beta_{l-1}^\ve(t,x) \vartheta_{b,l-1}^\ve- \alpha_l^\ve(t,x) c^\ve_{l} \vartheta_{f,l}^\ve, (c^\ve_l - S)_{+}  \rangle_{\Gamma^\ve_{zs}}
\leq \sup_{\Gamma_{zs}^\ve} ( \beta_{l-1}^\ve(t,x) \vartheta_{b,l-1}^\ve) \ve \|(c^\ve_l - S)_{+}\|_{L^1( \Gamma^\ve_{zs})}
\nonumber
\\
&&\qquad  \leq 1/2\left(  {\Upsilon_l}  \|(c^\ve_l - S)_{+}\|^2_{L^2(\Omega_l^\ve)}+  {d_l}  \|\nabla (c^\ve_l - S)_{+}\|^2_{L^2(\Omega_l^\ve)}+C^2_{Y_l} {\Upsilon_l}( 1+\Upsilon_l {\ve^2}/ {d_l} )  |\Omega_{l,S}^\ve(t)|  \right)\; , 
\end{eqnarray}
where $l=a,s$, with  $a-1=s$, $s-1=a$,  $\Omega^\ve_{l,S}(t)=\{ x\in \Omega^\ve_l: \,   c_l^\ve(t,x) > S\}$ for a.a. $t\in (0,T)$ and  $\Upsilon_l= \sup_{(0,T)\times\Gamma_{zs}^\ve} ( \beta_{l-1}^\ve(t,x) \vartheta_{{l-1},s}^\ve)$.
The Lipschitz continuity of $F_l$, with $l=a,s$,  ensures  
\begin{eqnarray}\label{estim_bound_2}
 \langle F^\ve_l(t,x,c^\ve_l), (c_l^\ve-S)_{+} \rangle_{\Omega^\ve_{l}} 
& \leq & 3 C_{F_l}/2 \|(c_l^\ve-S)_+\|^2_{L^2(\Omega_l^\ve)} +   ( S^2/C_{F_l} + C_{F_l})|\Omega_{l,S}^\ve(t)| \; . 
\end{eqnarray}
As next we shall  estimate $|\langle H_M({\vect v}^\ve_{l})    c^\ve_l,   \nabla  (c^\ve_l  - S)_{+}   \rangle_{\Omega^\ve_l,T}|  $ with $l=a,s$. Using H\"older's  inequality  yields 
\begin{eqnarray*}
&&|\langle H_M( {\vect v}^\ve_{l})    c^\ve_l,   \nabla  (c^\ve_l  - S)_{+}   \rangle_{\Omega^\ve_l, T}  | \leq \frac {d_l}4 \| \nabla (c^\ve_l-S)_{+} \|^2_{L^2((0,T)\times\Omega^\ve_{l})}\\ &&+
\frac 1{ d_l} \|H_M(\vect{v}^\ve_{l})\|^2_{L^{10}(0,T;L^4(\Omega^\ve_{l}))} \Big(\| (c^\ve_l-S)_{+} \|^2_{L^{5/2}(0,T;L^4(\Omega^\ve_{l}))} + S^2\Big(\int_0^T |\Omega^\ve_{l,S}|^{\frac 5 8}dt\Big)^{\frac{4}{5}} \Big)\; .
\end{eqnarray*}
Choosing $\varsigma=1/30$ we can estimate 
\begin{eqnarray*}
\| (c^\ve_l-S)_{+} \|^2_{L^{5/2}(0,T;L^4(\Omega^\ve_{l}))} \leq \| (c^\ve_l-S)_{+} \|^2_{L^{{5}/2(1+\varsigma)}(0,T;L^{4(1+\varsigma)}(\Omega^\ve_{l}))} \Big(\int_0^T |\Omega^\ve_{l,S}|^{\frac 58}dt\Big)^{\frac{4}{5} \frac \varsigma{\varsigma+1}} \; .
\end{eqnarray*}
Using imbedding result, see \cite[Chapter II, Eq.~(3.4)]{Ladyzhenskaja},  and extension of $c_l^\ve$ into $\Omega$,   we obtain 
\begin{eqnarray*}
\| (c^\ve_l-S)_{+} \|^2_{L^{ {5}/2(1+\varsigma)} (0,T;L^{4(1+\varsigma)}(\Omega^\ve_{l}))}  
&&\leq \| (c^\ve_l-S)_{+} \|^2_{L^{{5}/2 (1+\varsigma)}(0,T;L^{4(1+\varsigma)}(\Omega))} \\ 
&&\leq C_\Omega \left( \| (c^\ve_l-S)_{+} \|^2_{L^\infty(0,T;L^2(\Omega))} +   \|  \nabla (c^\ve_l-S)_{+} \|^2_{L^2(0,T;L^2(\Omega))}\right) \;.
\end{eqnarray*}
Considering  
$ \|H_M(\vect{v}^\ve_l)\|^2_{L^{10}(0,T;L^4(\Omega^\ve_{l}))} \leq T^{\frac 1 {5}} \|H_M(\vect{v}^\ve_l)\|^2_{L^\infty(0,T;L^4(\Omega^\ve_{l}))} \leq T^{\frac 1 5} C_v M^2 $
we conclude 
\begin{eqnarray}\label{estim_bound_3}
&& |\langle H_M({\vect v}^\ve_{l})    c^\ve_l,   \nabla  (c^\ve_l  - S)_{+}   \rangle_{\Omega^\ve_l, T}  | \leq \frac {d_l} 4 \| \nabla (c^\ve_l-S)_{+} \|^2_{L^2((0,T)\times\Omega^\ve_{l})} + 
T^{\frac 1{5}} C_\Omega C_v  M^2 / d_l\times  \\
&& \Big[
S^2\Big(\int_0^T |\Omega^\ve_{l,S}|^{\frac 5{8}}dt\Big)^{\frac{4}{5}}
 +  \Big(\int_0^T |\Omega^\ve_{l,S}|^{\frac 58}dt\Big)^{\frac{4}{5} \frac \varsigma{\varsigma+1}}
\Big(
\|  (c^\ve_l-S)_{+} \|^2_{L^\infty(0,T;L^2(\Omega))} + \| \nabla  (c^\ve_l-S)_{+} \|^2_{L^2(\Omega_T)}\Big)\Big] .  \nonumber
\end{eqnarray}
We shall define  
$$S=\max_{l=a,s}\max\{1, \|c_l^0\|_{L^\infty(\Omega)}, \Xi \big(\|c_l^0\|_{L^1(\Omega)} + \Upsilon_l |\Gamma||\Omega| + C_{F_l}|\Omega|\big) e^{C_{F_l} T}\},  $$ where $\Xi=\Xi(Y, Y_l)$ is the constant from  Lemma \ref{extension}. The choice  of $S$ ensures that $|\Omega_{l,S}^\ve(t)|\leq 1$ and $|\Omega_{l,S}(t)|\leq 1$.
Choosing 
{ $$\tilde T=\min_{l=a,s}\min\{(d_l^2/[8\Xi C_\Omega C_v M^2 |\Omega|^{\frac{\varsigma}{2\varsigma+2}}])^{\frac{5(\varsigma+1)}{(5\varsigma+1)}},  [d_l / (4\Xi C_\Omega C_v M^2|\Omega|^{\frac{\varsigma}{2\varsigma+2}})]^{\frac{5(\varsigma+1)}{5\varsigma+1}}, [4\Upsilon_l+ 12C_{F_l}]^{-1}, |\Omega|^{-\frac 58} \} \; ,$$
}
combining estimates \eqref{estim_bound_1}-\eqref{estim_bound_3} and   using the fact that $|\Omega_{l,S}^\ve(t)|\leq 1$
yield
\begin{align}
  \begin{aligned}
&  \|(c^\ve_l- S)_{+}\|^2_{L^\infty(0,\tilde T;L^2(\Omega))}  + \| \nabla (c^\ve_l  - S)_{+} \|^2_{L^2((0,\tilde T)\times\Omega)}\\
&\leq S^2 {8\Xi}/{\min\{1, d_l\}}\left(\tilde T^{\frac 1{5}} C_vC_\Omega M^2/d_l
+   1/C_{F_l} + 1 + C_{Y_l}(1+  \Upsilon_l\ve^2/d_l)/2 \right)\Big(\int_0^{\tilde T} |\Omega^\ve_{l,S}(t)|^{\frac 5{8}}dt\Big)^{\frac{4}{5}} \; . 
\nonumber
\end{aligned}
\end{align}
Choosing  $q=4(1+\varsigma)$ together with  $r=5(1+\varsigma)/2$ and  using $|\Omega_{l,S}^\ve(t)|\leq  |\Omega_{l,S}(t)|\leq 1$,
  we  obtain 
\begin{align}
  \begin{aligned}
&   \|(c^\ve_l- S)_{+}\|_{L^\infty(0,\tilde T;L^2(\Omega))}  +  \| \nabla (c^\ve_l  - S)_{+} \|_{L^2((0,\tilde T)\times\Omega)} \leq S
\Theta\Big(\int_0^{\tilde T} |\Omega_{l,S}(t)|^{\frac r{q}}dt\Big)^{\frac{2(1+\varsigma)}r}\; ,
\nonumber
\end{aligned}
\end{align}
where $\Theta=8\Xi/\min\{1, d_l\}\left(\tilde T^{\frac 1{5}} C_vC_\Omega M^2/d_l
+   1/C_{F_l} + 1 + C_{Y_l}(1+  \Upsilon_l\ve^2/d_l)/2 \right)$.
Applying   Theorem~II.6.1 in \cite{Ladyzhenskaja} yields 
\begin{eqnarray}
\|c_l^\ve \|_{L^\infty((0,\tilde T)\times \Omega)} \leq 2 S[1+ 2^{\frac {2\varsigma +1}{\varsigma^2}} (C_\Omega \Theta)^{1+\frac 1\varsigma} \tilde T^{1+ \frac 1\varsigma} |\Omega|^{\frac 14} ]  .
\end{eqnarray}
Considering   $T_c=\min\limits_{l=a,s}\min \{ \tilde T , 
 \frac{d_l^5}{(C_\Omega C_v M^2)^5}, |\Omega|^{-\frac {\varsigma} {4(1+\varsigma)}}
  [\frac{8C_\Omega  \Xi}{\min\{1, d_l\}}[2+\frac{1}{C_{F_l}}+ \frac{C_{Y_l}} 2(1+  \Upsilon_l\frac{\ve^2}{d_l})]
  2^{\frac {2+\varsigma} {\varsigma + \varsigma^2}}]^{-1}\} $ implies
$$
\|c^\ve_l \|_{L^\infty((0,\tau)\times \Omega)} \leq  4S \quad \text{ for all } \tau \leq T_c.
$$
 The iteration over time-intervals  will then  ensure the boundedness of $c^\ve_l$ in $(0,T)\times\Omega$, and thus 
 also the boundedness of $c^\ve_l$  in $(0,T)\times \Omega_l^\ve$, for $l=a,s$,
and   second  estimate  in  \eqref{energy_estim}.

Using the estimates for $c^\ve_l$   and $\vartheta_{j,l}^\ve$,   with $j=f,b$ and $l=a,s$, and considering $\partial_t \vartheta_{j,l}^\ve$  as test functions in \eqref{transporter_eq},  we conclude 
\begin{eqnarray}\label{estim_transp_time}
\ve^{1/2}\|\partial_t \vartheta_{f,l}^\ve\|_{L^2((0,T)\times \Gamma^\ve_{zs})} + \ve^{1/2}\|\partial_t \vartheta_{b,l}^\ve\|_{L^2((0,T)\times \Gamma^\ve_{zs})}\leq C_1(1+\ve^{1/2}\|c^\ve_l \|_{ L^2((0,T)\times\Gamma_{zs}^\ve))})\leq C_2\; . \hspace{-0.5 cm }
\end{eqnarray}

Differentiating   equations \eqref{velosity} and boundary conditions   \eqref{BC1velosity}, \eqref{BC3velosity} with respect to $t$,  considering 
$(\partial_t \vect{v}^\ve_z, \partial_t \vect{v}^\ve_a, \partial_t \vect{v}^\ve_{sp})$ as a test function, and applying Korn inequality \eqref{Korn2} we obtain for $\tau \in (0, T]$ 
 \begin{equation}
\label{veloc_time_der}
\begin{gathered}
\ve^2  \|\nabla\partial_t \vect{v}_z^\ve\|^2_{L^2((0,\tau)\times\Omega^\ve_z)}   +
  \|\partial_t \vect{v}_{s}^\ve\|^2_{L^2((0,\tau)\times\Omega^\ve_{s})} 
  +  \|\partial_t \vect{v}_{a}^\ve\|^2_{L^2((0,\tau)\times\Omega^\ve_{a})}
  +  \ve  \,    \|\partial_t \vect{v}_a^\ve\cdot \vect{n}\|^2_{L^2((0,\tau)\times\Gamma^\ve_{zs})} \\
\leq C \ve \big( \|\partial_t c_s^\ve\|^2_{L^2((0,\tau)\times\Gamma^\ve_{zs})}  +  \|\partial_t c_{a}^\ve\|^2_{L^2((0,\tau)\times\Gamma^\ve_{zs})}\big)\; . 
\end{gathered}
\end{equation}
Now we differentiate  with respect to $t$  equations \eqref{Eq_Conc}  and use  $\partial_t c^\ve_a$ and $\partial_t c^\ve_s$
 as test functions. 
  Estimates \eqref{energy_estim} together with  inequalities  \eqref{GN} and 
\eqref{estim_boundary},  estimates in Lemma~\ref{extension}, 
 and Assumption~\ref{assump}  give
\begin{eqnarray*}
&&  \|\partial_t c^\ve_l(\tau)\|^2_{L^2(\Omega^\ve_l)}   + 
\int_0^\tau \|\partial_t \nabla  c^\ve_l\|^2_{L^2(\Omega^\ve_l)} dt
 \leq C \|c^\ve_l\|^2_{L^\infty((0,T)\times\Omega_l^\ve)}
  \int_0^\tau \| \partial_t \vect{v}^\ve_l\|^2_{L^2(\Omega^\ve_l)} dt  +  \|\partial_t c^\ve_l(0)\|^2_{L^2(\Omega^\ve_l)} \\
&& + C(1+ \|H_M(\vect{v}_{l}^\ve)\|^2_{L^\infty(0,\tau; L^4(\Omega^\ve_{l}))})
 \int_0^\tau  \left[ \| \partial_t  c^\ve_l\|^2_{L^2(\Omega^\ve_l)}  + \ve \| \partial_t \vartheta_{b,l-1}^\ve\|^2_{L^2(\Gamma^\ve_{zs})}  + \ve \|\partial_t \vartheta_{f,l}^\ve\|^2_{L^2(\Gamma^\ve_{zs})} \right] dt  \; , 
  \end{eqnarray*}
where $l=a,s$ with  $a-1:=s$ and $s-1:=a$. The  regularity assumption on  $c^0_l$ ensures that $ \|\partial_t c^\ve_l(0)\|_{L^2(\Omega^\ve_l)}\leq C \|c^0_l\|_{H^2(\Omega_l^\ve)}$, with  $l=a,s$.  Combining the last inequality  together with    \eqref{estim_transp_time} and  \eqref{veloc_time_der}, using the  boundedness of  $H_M({\vect v}_{l}^\ve)$ and inequality \eqref{estim_boundary}, choosing $\ve$ sufficient small,  and applying  Gronwall's inequality imply  \eqref{time_estim}.\\

\noindent {\it Uniqueness.}
Suppose there are two solutions of the problem. We denote $c_l^{\varepsilon}= c_{l,1}^{\varepsilon}-c_{l,2}^{\varepsilon}$ and  $\vartheta_{j,l}^\ve =\vartheta_{j,l,1}^{\varepsilon} - \vartheta_{j,l,2}^{\varepsilon}$, with $j=f,b$ and $l=a,s$,
and choose $\varphi_1=c_a^{\varepsilon}$   in \eqref{def_apop} and $\varphi_2=c_s^{\varepsilon}$   in \eqref{def_symp} 
\begin{eqnarray*}
&&
\|c_l^{\varepsilon}(\tau)\|^2_{L^2(\Omega^\ve_l)}  + 2(d_l-\zeta_1)\int_0^\tau\|\nabla c_l^{\varepsilon}\|_{L^2(\Omega^\ve_l)}
dt 
\leq C_{\zeta_1}\|c^\ve_{l,1}\|^2_{L^\infty((0,T)\times\Omega_l^\ve)}
\int_0^{\tau} \|\vect{v}_{l,1}^{\varepsilon}- \vect{v}_{l,2}^{\varepsilon}\|^2_{L^2(\Omega^\ve_l)} 
dt \\
&& +C_{\zeta_1}(1+ \|H_M(\vect{v}_{l,2}^\ve)\|^2_{L^\infty(0,\tau; L^4(\Omega^\ve_{l}))})
\int_0^{\tau} \|c_{l}^{\varepsilon}\|^2_{L^2(\Omega^\ve_l)} dt
+C_{\zeta_2}\|c^{\ve}_{l,1}\|^2_{L^\infty((0,T)\times\Gamma_{zs}^\ve)} 
 \int_0^\tau \ve 
\|\vartheta_{f,l}^{\varepsilon} \|^2_{L^2(\Gamma^\varepsilon_{zs})} dt \\
&&+  C_{\zeta_2}\int_0^\tau \varepsilon 
\|\vartheta_{b,l-1}^{\varepsilon}\|^2_{L^2(\Gamma^{\varepsilon}_{zs})} dt+ \zeta_2
 \int_0^{\tau}\varepsilon \|c_l^{\varepsilon}\|^2_{L^2(\Gamma^{\varepsilon}_{zs})} dt \; ,
\end{eqnarray*}
for any $\tau \in (0,T]$, where $a-1:=s$ and $s-1:=a$. For $ \vartheta_{j,l}^{\varepsilon}$, with $j=f,b$ and $l=a,s$, we obtain
\begin{eqnarray*}
&{\partial_t}\|\vartheta_{f,l}^\ve\|^2_{L^2(\Gamma_{zs}^\ve)} &\leq 
C\left[(1+  \|c^{\ve}_{l,1}\|_{L^\infty} )
\|\vartheta_{f,l}^\ve\|^2_{L^2(\Gamma_{zs}^\ve)}  + 
\|\vartheta^{\ve}_{f,l,2}\|^2_{L^\infty} \|c_l^{\varepsilon}\|^2_{L^2(\Gamma^{\varepsilon}_{zs})} +
 \|\vartheta_{b,l}^\ve\|^2_{L^2(\Gamma_{zs}^\ve)} \right] \; , \\
&{\partial_t}\|\vartheta_{b,l}^\ve\|^2_{L^2(\Gamma_{zs}^\ve)}&\leq
C\left[ \|c^{\ve}_{l,1}\|^2_{L^\infty} 
\|\vartheta_{f,l}^\ve\|^2_{L^2(\Gamma_{zs}^\ve)}  + 
\|\vartheta^{\ve}_{f,l,2}\|^2_{L^\infty} \|c_l^{\varepsilon}\|^2_{L^2(\Gamma^{\varepsilon}_{zs})} + \|\vartheta_{b,l}^\ve\|^2_{L^2(\Gamma_{zs}^\ve)}\right] \; .
\end{eqnarray*}
Combining those estimates,  using  \eqref{uniqueness_estim} and \eqref{estim_boundary}, boundedness of
$c_l^\ve$ and  $H_M(\vect{v}_{l}^\ve)$,  and applying Gronwall's inequality yield
\begin{eqnarray*}
 \ve\|\vartheta_{j,l}^\ve\|^2_{L^\infty(0,T;L^2(\Gamma_{zs}^\ve))}
 + \|c_l^{\varepsilon}\|^2_{L^\infty(0,T; L^2(\Omega^\ve_l))}+ \|\vect{v}_{l,1}^{\varepsilon} - \vect{v}_{l,2}^{\varepsilon}\|^2_{L^2((0,T)\times\Omega^\ve_l)} \leq 0\; .
\end{eqnarray*}
Thus,  $c_{l,1}^{\varepsilon}=c_{l,2}^{\varepsilon}$,  
$\vect{v}_{l,1}^{\varepsilon}=\vect{v}_{l,2}^{\varepsilon}$  in  $(0,T)\times\Omega^{\varepsilon}_l$,
and
$\vartheta_{j,l,1}^{\varepsilon}=\vartheta_{j,l,2}^{\varepsilon}$ on $[0,T]\times\Gamma^{\varepsilon}_{zs}$, 
  with $l=a,s$,    $j=f,b$.
\hfill \end{proof}

%%%%%%%%%%%%%%%%%%%%%%%%%%%%%%%
%%%%%%%%%%%%%%%%%%%%%%%%%%%%%%%
\section{Derivation of macroscopic equations}\label{MacroModel}

We denote by $\tilde{\vect{v}}^\ve_i$, for $i=z, a, sp$ the extension by zero from $(0,T)\times\Omega_i^\ve$ into $\Omega_T=(0,T)\times \Omega$, where $\Omega_{sp}^\ve:=\Omega_{as}^\ve$,
and  define $\vect{v}^\ve=\tilde{\vect{v}}_z^\ve + \tilde{\vect{v}}_{a}^\ve + \tilde{\vect{v}}_{sp}^\ve$ in  $\Omega_T$.

\begin{lemma}\label{convergence_water}
%Assume  $\|\vect{v}^\ve_a\|_{L^\infty((0,T)\times \Omega_a^\ve)}+\|\vect{v}^\ve_s\|_{L^\infty((0,T)\times \Omega_s^\ve)}\leq C$  uniformly with respect to $\ve$ and 
Under  Assumption~\ref{assump} 
there exist $ \vect{v}_z \in L^2(\Omega_T; H^1_{\text{per}}(Y_z)/\mathbb R)$, $\vect{v}_{a} \in L^2(\Omega_T\times Y_{a})$,  $\vect{v}_{sp} \in L^2(\Omega_T\times Y_{as})$  and 
$p_a, p_s \in L^2(\Omega_T\times Y)/\mathbb R$,  
 such that, up to a subsequence, 
\begin{eqnarray}
\begin{cases}\label{conver_v_p}
 \vect{v}^\ve_z \to  \vect{v}_z, \quad  \vect{v}^\ve_{a} \to  \vect{v}_{a}, \quad 
 \vect{v}^\ve_{sp} \to  \vect{v}_{sp}   &   \text{ two-scale}\; , \\
 \vect{v}^\ve_s \to  \tilde{\vect{v}}_z+ \tilde{\vect{v}}_{sp} & \text{ two-scale in } \Omega_T\times Y_s\; ,   \\
\vect{v}^\ve \to\overline{\vect{v}}=\tilde{\vect{v}}_z + \tilde{\vect{v}}_{a} +\tilde{ \vect{v}}_{sp}  &  \text{ two-scale in } \Omega_T\times Y\; , \\
 \ve \nabla \vect{v}^\ve_z \to \nabla_y  \vect{v}_z &  \text{ two-scale in } \Omega_T\times Y_z \; ,\\
  P^\ve_a \to p_a\; , \, \,  P^\ve_s \to p_s  &  \text{ two-scale in } \Omega_T\times Y \; ,\\
   p^\ve_a \to p_a|_{\Omega_T\times Y_a} \, \, \text{ two-scale in }  \Omega_T\times Y_a\; ,\, \,  &
    p^\ve_s \to p_s|_{\Omega_T\times Y_s}   \, \, \text{ two-scale in }  \Omega_T\times Y_s \; ,
  \end{cases}
\end{eqnarray}
and  
\begin{eqnarray}
  \begin{cases}\label{conver_v_p_weak}
    P^\ve_a \to \frac 1{|Y|} \int_Y p_a \, dy\; , \qquad P^\ve_s \to \frac 1{|Y|} \int_Y p_s \, dy & \text{ weakly in } L^2(\Omega_T)\; ,\\
 \vect{v}^\ve \to  {\vect v}= \frac 1{|Y|} \int_{Y_z}{\vect v}_z\,  dy + \frac 1{|Y|}  \int_{Y_{a}}{\vect v}_{a}\,  dy + \frac 1{|Y|} \int_{Y_{as}}\vect{v}_{sp}\, dy &\text{ weakly  in } L^2(\Omega_T)\; , \\
\vect{v}^\ve_{a} \cdot \vect{n}\to  \vect{v}_{a}\cdot \vect{n} \, & \text{ two-scale on } \Omega_T\times \Gamma_{zs} \; . 
\end{cases}
\end{eqnarray}
\end{lemma}
\begin{proof}  The convergences in \eqref{conver_v_p} follows directly from  estimates \eqref{ap_estim},  \eqref{ap_estim_extension}  and   \eqref{energy_estim} together with  Lemma~4.1  in \cite{Arbogast}
and the definition and compactness theorems for two-scale convergence, see  \cite{allaire, Nguetseng} or
Definition~\ref{two-scale-conver} and Theorem~\ref{compact1} in Appendix.
Since $\tilde{\vect{v}}_i^\ve$, for $i=z,a, sp$, are zero in $\Omega\setminus \Omega_i^\ve$,  also the two-scale limits are equal to zero in  $Y\setminus Y_i$, respectively.
A priori estimates  \eqref{ap_estim}, \eqref{ap_estim_extension} and the relation between two-scale and weak limits,  \cite{allaire, Arbogast, Nguetseng}, ensure the first two convergences in \eqref{conver_v_p_weak}.
The uniform in $\ve$ boundedness of  $\ve^{1/2} \| {\vect v}_a^\ve\cdot \vect{n}\|_{L^2((0,T)\times\Gamma_{zs}^\ve)}$
implies that there exists  $w\in L^2(\Omega_T\times\Gamma_{zs})$ such that  $\vect{v}^\ve_a \cdot\vect{n}\to w$ two-scale, see  \cite{neuss-radu} or Definition~\ref{two-scale-boundary} and Theorem~\ref{th_comp_b} in Appendix.   Then $\operatorname{div} \vect{v}_a^\ve =0$ in $(0,T)\times\Omega_a^\ve$  and  two-scale convergence of $\vect{v}_a^\ve$ to $\vect{v}_a$  in $\Omega_T\times Y_a$ ensure  $w=\vect{v}_a \cdot\vect{n}$ a.e. on $\Omega_T\times\Gamma_{zs}$. 
\end{proof}

We consider the extension of $c_l^\ve$ from $\Omega^\ve_l$ into $\Omega$, for $l=a,s$, as in Lemma~\ref{extension}, and  identify $c_l^{\ve}$ with its extension.  We denote by $\mathcal T^\ve_{\Gamma_{zs}}: \Gamma^\ve_{zs} \to  \Omega\times \Gamma_{zs}$  the boundary unfolding operator, see \cite{Cioranescu1}  or Definition~\ref{def_unfolding} in Appendix. Here we shall use a shorted notation $\mathcal T^\ve_{\Gamma}(\psi):=\mathcal T^\ve_{\Gamma_{zs}}(\psi) $. 
 
\begin{lemma}\label{convergence_concentration}
Under Assumption~\ref{assump}
there exist  $c_l\in H^1(0,T; H^1(\Omega))\cap L^\infty(\Omega_T)$, $c_l^1 \in L^2(\Omega_T; H^1_{\text{per}}(Y_l)/\mathbb R)$
 and  $ \vartheta_{j,l}\in L^\infty(\Omega_T\times\Gamma_{zs})\cap H^1(0,T; L^2(\Omega\times \Gamma_{zs}))$, with $j=f,b$ and $l=a,s$,  such that, up to a subsequence,
\begin{eqnarray}\label{conc-conver}
& c^\ve_l \to c_l  \text{ weakly  in }  H^1(0,T; H^1(\Omega))\; ,  & \text{ weakly-}\ast \text{ in } L^\infty(\Omega_T)\; , \nonumber \\
& c^\ve_l  \to c_l    \text{ strongly  in }  L^2(0,T; H^s(\Omega))\; , \, \, \,  s<1 \; ,  \, \, & \, \,
 \ve^{1/2} \|c^\ve_l - c_l\|_{L^2((0,T)\times \Gamma^\ve_{zs})} \to 0 \; ,  \nonumber \\
& c^\ve_l \to c_l\; , \, \, \,  \partial_t c^\ve_l \to  \partial_t c_l\; , \,  \,  \,   \nabla  c^\ve_l \to \nabla c_l   + \nabla_y c_l^1 &  \text{ two-scale}\; ,  \\
&  \vartheta_{j,l}^\ve \to \vartheta_{j,l}\; , \, \, \, \partial_t \vartheta_{j,l}^\ve \to \partial_t \vartheta_{j,l} & \text{ two-scale} \; , \nonumber\\
&\mathcal T^\ve_{\Gamma} c_l^\ve \to c_l\; , \,\, \,  \mathcal T^\ve_{\Gamma} \vartheta_{j,l}^\ve \to \vartheta_{j,l} 
 & \text{ strongly in } L^2(\Omega_T\times \Gamma_{zs})\; . \nonumber 
\end{eqnarray}
\end{lemma}
\begin{proof} 
The first two convergences are a direct consequence of the a priori estimates for $c^\ve_l$  in \eqref{energy_estim}-\eqref{time_estim} and   properties of the  extension from $\Omega_l^\ve$  into $\Omega$, with $l=a,s$. 
The compact embedding   
$ H^1(0,T; H^1(\Omega))\subset L^2(0,T; H^{s}(\Omega))$ for $s<1$ ensures the strong convergence of $c^\ve_a$ and $c^\ve_s$. 
Then, the  strong convergence  in $L^2(0,T; H^{s}(\Omega))$, with $1/2<s<1$,  and the estimate
\begin{eqnarray*}
\ve^{1/2}\|c^\ve_l\|_{L^2((0,T)\times \Gamma^\ve_{zs})} \leq C \| c^\ve_l\|_{L^2(0,T; H^s(\Omega^\ve_l))}\; , \quad \qquad \text{ with }  l=a,s\; ,
\end{eqnarray*}
see \cite{AnnaMariya2008} for the proof, give  the fourth  convergence in \eqref{conc-conver}. A priori estimates \eqref{energy_estim}, \eqref{time_estim} and compactness theorems for two-scale convergence,  see \cite{allaire,  neuss-radu, Nguetseng} or Theorem~\ref{compact1} and \ref{th_comp_b} in Appendix,   imply  two-scale convergences stated in \eqref{conc-conver}.  Theorem~\ref{th_comp_b} in Appendix ensures also that $\vartheta_{j,l} \in L^\infty(\Omega_T\times\Gamma_{zs})$.

The  assumed structure of the initial data $\vartheta_{j,l}^{0,\ve}$, i.e. $\mathcal{T}^\ve_{\Gamma}(\vartheta^{0,\ve}_{j,l})(x,y)= \vartheta^0_{2j,l}(y)\mathcal{T}^\ve_{\Gamma}(\vartheta^{0}_{1j,l}(x))$  and the strong convergence in   $L^2(\Omega\times \Gamma_{zs})$ of  $\mathcal{T}^\ve_{\Gamma}(\psi)$ for  $\psi\in L^2(\Omega)$, see \cite{Cioranescu1},
 yield  $\mathcal{T}^\ve_{\Gamma}(\vartheta^{0,\ve}_{j,l}) \to \vartheta^0_{j,l}$   strongly in $L^2(\Omega\times \Gamma_{zs})$, with $j=f,b$ and $l=a,s$. The properties of the unfolding operator,  see \cite{Cioranescu1}, 
 give
\begin{eqnarray*}
\|\mathcal T^\ve_{\Gamma}( c_l^\ve) -\mathcal T^\ve_{\Gamma} (c_l) \|_{L^2(\Omega_T\times\Gamma_{zs})} \leq  C \ve^{1/2}\| c^\ve_l - c_l \|_{L^2((0,T)\times\Gamma_{zs}^\ve)} \quad \text{ for } \, l=a,s \; ,
\end{eqnarray*}
and $\mathcal T^\ve_{\Gamma}(c_l) \to c_l$ strongly in  $L^2(\Omega_T\times\Gamma_{zs})$.
Then, together with the fourth  convergence in  \eqref{conc-conver} we obtain that 
$\mathcal T^\ve_{\Gamma} (c_l^\ve) \to c_l$ strongly in  $L^2(\Omega_T\times\Gamma_{zs})$, where $l=a,s$.
Applying now the unfolding operator to the equations in \eqref{transporter_eq},  using the convergence of $c^\ve_a$, $c^\ve_s$, and the equivalence between the  two-scale convergence of a sequence and the weak convergence of the corresponding unfolded sequence, see  \cite{AnnaMariya2008} or Lemma~\ref{two-scale-unfold} in Appendix,  we can show, in the same manner as in \cite{AnnaMariya2008},   that 
$\{\mathcal T^\ve_{\Gamma}(\vartheta_{j,l}^\ve)\}$, with $j=f,b$ and $l=a,s$,   are  Cauchy sequences 
in $L^2(\Omega_T\times\Gamma_{zs})$ and conclude the last  strong convergences stated in \eqref{conc-conver}. 
\end{proof}

\begin{theorem}
For the  sequence of solutions  $({\vect{v}}_z^\ve , \vect{v}_{a}^\ve, {\vect{v}}_{sp}^\ve) $
and $( p_{z}^\ve, p_a^\ve, p_{sp}^\ve)$  of the microscopic problem \eqref{velosity}, \eqref{BC1velosity}, \eqref{BC3velosity},  we have 
$\vect{v}^\ve \to \vect{v}$ weakly in $ L^2(0,T; H(\text{div},\Omega))$, with $\vect{v}^\ve=\tilde{\vect{v}}_z^\ve + \tilde{\vect{v}}_{sp}^\ve + \tilde {\vect{v}}_a^\ve$,  the  extensions 
$P^\ve_s \to p$, $P^\ve_a\to p$ weakly in $L^2(\Omega_T)/\mathbb R$ as $\ve \to 0$, and  $(\vect{v},p) \in L^2(0,T;H(\text{div},\Omega))\times L^2(\Omega_T)/\mathbb R$ is the unique solution of the   Darcy problem
\begin{equation}\label{Darcy_macro}
\begin{cases}
\vect{ v} + \vect{K} \nabla p =  \vect{M}(c_s-c_a) & \text{ in } (0,T)\times \Omega\; ,\\
\operatorname{div} \vect{v} =0 & \text{ in } (0,T)\times \Omega\; ,\\
\vect{v} \cdot \vect{n} = v_D & \text{ on } (0,T)\times\partial \Omega\; ,
\end{cases}
\end{equation}
where  the  tensor $\vect{K}$  and vector ${\vect M }$ are defined in \eqref{permeab}, and $c_s, c_a$ are solutions of the macroscopic equations \eqref{macro_c_a}-\eqref{macro_transporter_theta}.
\end{theorem}

\begin{proof}
 The weak convergences of $\vect{v}^\ve$ in $L^2(\Omega_T)$ and $P^\ve_l$ in   $L^2(\Omega_T)/\mathbb R$, with $l=a,s$, follows from Lemma~\ref{convergence_water}.
Using  $\text{div } \vect{v}^\ve \in L^2(\Omega_T)$ and $\text{div } \vect{v}^\ve =0 $ in  $\Omega_T$, and applying  the weak convergence of $\vect{v}^\ve$ we obtain for $\psi\in C^\infty_0(\Omega_T)$
\begin{equation}\label{div_bc_1}
0=\lim\limits_{\ve\to 0} \langle \text{div } \vect{v}^\ve, \psi \rangle_{\Omega,T} = \lim\limits_{\ve\to 0}\langle \vect{v}^\ve, \nabla \psi \rangle_{\Omega,T} = \langle  {\vect{v}}, \nabla \psi \rangle_{\Omega,T}
= \langle\text{div }{\vect{v}}, \psi \rangle_{\Omega,T} \; .
\end{equation}
For $\psi \in C^{\infty}(\Omega_T)$ we have 
\begin{eqnarray}\label{div_bc_2}
0=\lim\limits_{\ve\to 0}\langle \text{div } \vect{v}^\ve, \psi\rangle_{\Omega,T} =\lim\limits_{\ve\to 0}\big( \langle {\vect v}^\ve \cdot \vect{n}, \psi\rangle_{\partial \Omega,T}- \langle \vect{v}^\ve, \nabla \psi \rangle_{\Omega,T}\big)= \langle v_D, \psi \rangle_{\partial \Omega,T}- \langle \vect{v}, \nabla \psi \rangle_{\Omega,T} \; ,
\end{eqnarray}
and, together with $\text{div }{\vect{v}}=0$ in $\Omega_T$,  given by \eqref{div_bc_1}, conclude $\vect{v}\cdot \vect{n}=v_D$ on $(0,T)\times \partial \Omega$.

Considering  $\psi_1, \psi_2 \in C^\infty_0(\Omega_T; C^\infty_{per}(Y))$ 
with $\psi_1(t,x,y) =0 $ in $Y\setminus Y_s$, such that $\psi_1\times \vect{n} =0$ on $\Gamma_{as}$,  and 
$\psi_2(t,x,y)=0$ in $Y\setminus Y_a$, for $(t,x) \in \Omega_T$,
 taking  $\psi^\ve(t,x)= \ve (\psi_1(t, x, x/\ve), \psi_2(t,x,x/\ve), \psi_1(t,x, x/\ve))$  as a test function in  \eqref{Stokes} and applying two-scale convergences 
 of $\vect{v}_j^\ve$, with $j=z,a,sp$ and of $p_l^\ve$, with $l=a,s$,  stated in Lemma~\ref{convergence_water}, yield
\begin{eqnarray*}
\langle  p_s , \text{ div}_y \psi_1 \rangle_{\Omega\times Y_s, T}
 + \langle  p_a, \text{ div}_y  \psi_2\rangle_{\Omega\times Y_a, T} =0 \; .
\end{eqnarray*}
This implies   that 
 $p_l \in L^2(\Omega_T; H^1(Y_l))$ with   $\nabla_y p_l=0$ a.e. in $\Omega_T\times Y_l$, where $l=a,s$. Thus  $p_l$ is independent of $y$ in $\Omega_T\times Y_l$ and  $p_l=p_l(t,x)$ a.e.  in $\Omega_T\times Y_l$, where $l=a,s$. 
Taking  now $\psi^\ve(t,x)= \ve (\psi_1(t,x, x/\ve), \psi_1(t,x,x/\ve), 0)$ with  $\psi_1\in C^\infty_0(\Omega_T; C^\infty_{per}(Y))$ and $\psi_1(t,x,y)=0$  in $\overline Y_{as}$  and   $\psi_1\times \vect{n} =0$ on $\Gamma_{z}$ as  a test function in  \eqref{Stokes}  and letting  $\ve\to 0$ give
\begin{eqnarray*}
\langle p_{a}(t,x)-p_s(t,x), \psi_1(t,x,y)\cdot \vect{n} \rangle_{\Omega\times \Gamma_z, T}=0 \; .
\end{eqnarray*}
Thus  $p_{a}(t,x)=p_s(t,x)$ a.e. in $\Omega_T\times\Gamma_z$ and $p_{a}(t,x)=p_s(t,x) = p(t, x) \, \text{ a.e. in }  \Omega_T$.

Now we consider   $\psi^\ve(t,x)=(\psi_1(t,x,x/\ve)+\psi_2(t,x,x/\ve), \psi_1(t,x,x/\ve), \psi_2(t,x,x/\ve))$, where  functions
$\psi_1, \psi_2 \in C^\infty_0(\Omega_T, C^\infty_{per}(Y))$ with $\psi_2\cdot \vect{n} =0 $ on $\Gamma_z\cup \Gamma_{aw}$  and 
$\psi_2\times\vect{n}=-\psi_1\times\vect{n}$ on $\Gamma_{zs}$ and 
$\text{div}_y \psi_1 =0$ in $\Omega_T\times Y$, $\text{div}_y\psi_2=0$ in $\Omega_T\times Y_s$ as a test function in  \eqref{Stokes}. The two-scale convergences in \eqref{conver_v_p} and \eqref{conver_v_p_weak},
 and the convergence of $c^\ve_a$ and $c^\ve_s$ on $(0,T)\times\Gamma_{zs}^\ve$ stated in \eqref{conc-conver}  imply
\begin{equation}\label{macro_eq_water_2}
\begin{aligned}
\langle 2 \eta \operatorname{S}_y \! {\vect v}_z,   \operatorname{S}_y (\psi_1+\psi_2) \rangle_{Y_z\times\Omega_T}-\langle  p, \text{div}_x (\psi_1+\psi_2) \rangle_{Y_z\times\Omega_T} +\langle K^{-1}_a(y) {\vect v}_a ,  \psi_1 \rangle_{Y_a\times\Omega_T}\\
  -  \langle p , \text{div}_x  \psi_1\rangle_{Y_a\times\Omega_T}  +\langle K^{-1}_{sp}(y) {\vect v}_{sp} ,  \psi_2 \rangle_{Y_{as}\times\Omega_T} -  \langle p , \text{div}_x  \psi_2 \rangle_{Y_{as}\times\Omega_T}\\
+ \langle { \delta(y)} (c_s-c_a) +  \kappa(y) {\vect v}_a \cdot \vect{n}, \psi_1\cdot \vect{n} \rangle_{\Gamma_{zs}\times \Omega_T}  & = 0 \; , 
\end{aligned}
\end{equation}
where $\operatorname{S}_y \!{\vect v}= 1/2(\nabla_y {\vect v}+ \nabla_y {\vect v}^T)$. Choosing  $\psi_1=0$,  restricting $\psi_2$ to $Y_z$, i.e. $\psi_2\in C^\infty_0(\Omega_T, C^\infty_0(Y_z))$, and applying the integration by parts imply
\begin{eqnarray*}
 \langle -2\eta \text{ div}_y( \operatorname{S}_y\!{\vect v}_z)+ \nabla_x p,  \psi_2 \rangle_{Y_z\times\Omega_T} =0 \; .
\end{eqnarray*}
Since $\text{div}_y \psi_2(t,x,y) =0$ in $\Omega_T\times Y_z$, there exists $p_{1,z} \in L^2(\Omega_T, L^2_{\text{per}}(Y_z)/\mathbb R)$, see \cite{Girault, hornung},  such that
\begin{eqnarray}\label{eq_Stokes_macro1}
 -2\eta \text{ div}_y (\operatorname{S}_y\!{\vect v}_z)+ \nabla_x p+ \nabla_y p_{1,z}=0 \quad \text{ in }\, \, \Omega_T\times Y_z \; .
\end{eqnarray}
 Similarly we obtain  the existence of $p_{1,a} \in L^2(\Omega_T, L^2_{\text{per}}(Y_{a})/\mathbb R)$, $p_{1,sp} \in L^2(\Omega_T, L^2_{\text{per}}(Y_{as})/\mathbb R)$ and
 \begin{align}\label{eq_Darcy_macro1}
 \begin{aligned}
 K^{-1}_a{\vect v}_a+ \nabla_x p+ \nabla_y p_{1,a}=0 &\quad \text{ in } \Omega_T\times Y_a\; , \\
K^{-1}_{sp} {\vect v}_{sp}+ \nabla_x p+ \nabla_y p_{1,sp}=0 &\quad \text{ in } \Omega_T\times Y_{as}\; .
\end{aligned}
\end{align}
Considering $\text{div }\vect{v}^\ve_j \in L^2((0,T)\times\Omega_j^\ve)$ and $\text{div } \vect{v}_j^\ve=0$ in $(0,T)\times\Omega_j^\ve$,    
with  $j = z, a, sp$ and $\Omega_{sp}^\ve := \Omega_{as}^\ve$, 
we obtain for $\psi \in C^\infty_0((0,T)\times\Omega; C^\infty_0(Y_j))$, where $Y_{sp}:= Y_{as}$,
\begin{equation*}
0=\lim\limits_{\ve\to 0} \langle \text{div } \vect{v}_{j}^\ve,  \psi(t,x , x/\ve) \rangle_{\Omega_j^\ve, T} = \lim\limits_{\ve\to 0}
\langle\vect{v}_{j}^\ve, \nabla_x \psi(t, x , x/\ve) + \ve^{-1}\nabla_y \psi(t,x , x /\ve)   \rangle_{\Omega_j^\ve, T}\; . 
\end{equation*}
The  two-scale convergence of $\vect{v}_j^\ve$ 
ensures $\lim\limits_{\ve\to 0}\langle \vect{v}_{j}^\ve(t,x), \nabla_y \psi(t, x ,x/\ve) \rangle_{\Omega_j^\ve,T}  = 0$ and then implies that 
$
\text{div}_y \vect{v}_z=0 \text{ in }\Omega_T\times Y_z$,  $\text{div}_y \vect{v}_{sp}=0 \text{ in }\Omega_T\times Y_{as}$, and 
$\text{div}_y \vect{v}_{a}=0 \text{ in }\Omega_T\times Y_a.$ \\
Similarly, using  $\text{div }\vect{v}^\ve \in L^2(\Omega_T)$ and $\text{div }\vect{v}^\ve= 0$ in $\Omega_T$,  the two-scale convergence of $\vect{v}^\ve$ ensures
 $\text{ div}_y \overline{\vect{v}} = 0$ in $\Omega_T\times Y$.

Considering the two-scale convergence of   $\vect{v}^\ve_z$,  $\vect{v}^\ve_a$,  $\vect{v}^\ve_{sp}$, and using the calculations similar to  \eqref{div_bc_1}-\eqref{div_bc_2} we can    conclude  $\vect{v}_z \cdot \vect{n} = (\vect{v}_a+ \vect{v}_{sp}) \cdot \vect{n} $ on $\Omega_T\times\Gamma_{as}$ and $\vect{v}_z\cdot\vect{n} = \vect{v}_a\cdot\vect{n}$ on $\Omega_T\times \Gamma_{z}$.
 Inequality \eqref{estim_boundary}, applied to $\vect{v}_z^\ve$, with $\Gamma_{zs}^\ve$ and $\Omega_z^\ve$, and the estimates for  $\| \vect{v}_z^\ve\|_{L^2((0,T)\times \Omega^\ve_{z})} $   and  $\ve \| \nabla \vect{v}_z^\ve\|_{L^2((0,T)\times \Omega^\ve_{z})}$ in \eqref{ap_estim} ensure the uniform in $\ve$ boundedness of $\ve^{1/2} \| \vect{v}_z^\ve\|_{L^2((0,T)\times \Gamma^\ve_{zs})}$.
Then the compactness theorem for the two-scale convergence on  oscillating boundaries, see \cite{neuss-radu} or Theorem \ref{th_comp_b} in Appendix,   together with the two-scale convergence of $\ve \nabla \vect{v}^\ve_z$  and  $\text{div }  \vect{v}^\ve_z=0$, as well as   $\text{div}_y  \vect{v}_z=0$, implies
the two-scale convergence of $\vect{v}^\ve_z$ to  $\vect{v}_z$ on $\Omega_T\times\Gamma_{zs}$ and ensures
$\vect{v}_z\times\vect{n} =0 $  on $\Omega_T\times\Gamma_{zs}$.

Applying in 
\eqref{macro_eq_water_2} integration by parts  and   accounting equations  \eqref{eq_Stokes_macro1} and \eqref{eq_Darcy_macro1} yield
 \begin{align*}
 \langle 2\eta (\operatorname{S}_y\!\vect{v}_z\vect{n}) \cdot  \vect{n} - p_{1,z}, \psi_1\cdot \vect{n} \rangle_{\Omega_T\times\Gamma_z}+
 \langle 2\eta ( \operatorname{S}_y\!\vect{v}_z\vect{n} )\cdot \vect{n}- p_{1,z}, (\psi_1+\psi_2)\cdot \vect{n} \rangle_{\Omega_T\times\Gamma_{as}} \\
 + \langle p_{1,a} +   \delta(y) (c_s-c_a)+ \kappa(y) \vect{v}_a \cdot \vect{n}, \psi_1\cdot \vect{n} \rangle_{\Omega_T\times\Gamma_{zs}} +
\langle p_{1,sp}, \psi_2\cdot \vect{n} \rangle_{\Omega_T\times\Gamma_{as}} &=0 
 \end{align*} 
for  $\psi_1, \psi_2 \in C^\infty_0(\Omega_T, C^\infty_{per}(Y))$ with $\psi_2\cdot \vect{n} =0 $ on $\Gamma_z\cup\Gamma_{aw}$ and 
$\psi_2\times\vect{n}=-\psi_1\times\vect{n}$ on $\Gamma_{zs}$,  $\text{div}_y \psi_1 =0$ in $\Omega_T\times Y$, $\text{div}_y\psi_2=0$ in $\Omega_T\times Y_s$. Thus, we conclude that   $\vect{v}_z \in L^2(\Omega_T; H^1_{\text{per}}(Y_z)/\mathbb R)$, $\vect{v}_{sp}\in L^2(\Omega_T\times Y_{as})$,  $\vect{v}_a \in L^2(\Omega_T\times Y_a)$, and $p\in L^2(\Omega_T)/\mathbb R$,  
$p_{1,z} \in L^2(\Omega_T; L^2_{\text{per}}(Y_z)/ \mathbb R)$, $p_{1,sp} \in L^2(\Omega_T; L^2_{\text{per}}(Y_{as})/ \mathbb R)$,  and 
$p_{1,a} \in L^2(\Omega_T; L^2_{\text{per}}(Y_a)/ \mathbb R)$,  satisfy equations 
\begin{equation}
\begin{cases}\label{coupled_macro}
  -\eta\Delta_y {\vect v}_z + \nabla_x p  + \nabla_y p_{1,z} \;  \ =  0 & \text{ in } \Omega_T\times Y_z\; ,\\
 \hspace{0.3 cm}  K^{-1}_a {\vect v}_{a} + \nabla_x p + \nabla_y p_{1,a}\,\,  =0 & \text{ in } \Omega_T\times Y_{a}\; , \\
 \hspace{0.2 cm} K^{-1}_{sp} {\vect v}_{sp} + \nabla_x p + \nabla_y p_{1, sp} =0 & \text{ in } \Omega_T\times Y_{as}\; , \\
  \text{div}_y {\vect v}_z = 0 \, \,  \text{ in } \Omega_T\times Y_z\; ,  \quad
\text{div}_y {\vect v}_a = 0  \, \, \text{ in } \Omega_T\times Y_a\; , \quad \text{div}_y {\vect v}_{sp}=0 &  \text{ in } \Omega_T\times Y_{as}\;,
 \end{cases}
 \end{equation}
 with  boundary and transmission conditions 
 \begin{equation}
\begin{cases}\label{boundary_macro}
  - 2 \eta (\operatorname{S}_y\! {\vect v}_z\vect{n})\cdot \vect{n} + p_{1,z}-p_{1,a} =  
 \delta(y) (c_s-c_a) +  \kappa(y) {\vect  v}_a\cdot  \vect{n}  &  \text{ on }\Omega_T\times \Gamma_{zs}\; ,\\
-2 \eta (\operatorname{S}_y {\vect v}_z\vect{n})\cdot \vect{n} + p_{1,z}=p_{1, sp}  &  \text{ on } \Omega_T\times \Gamma_{as}\; ,\\
 \vect{v}_z\cdot \vect{n}= \vect{v}_{a}\cdot \vect{n} + \vect{v}_{sp}\cdot \vect{n}\; , \ \vect{v}_z\times \vect{n} =0 & \text{ on } \Omega_T\times \Gamma_{as}\; , \\
  \vect{v}_z\cdot \vect{n}= \vect{v}_{a}\cdot \vect{n}\; , \ \qquad\qquad
    \vect{v}_z\times \vect{n} =0  &   \text{ on } \Omega_T\times\Gamma_{z}\; ,\\
      \vect{v}_{sp}\cdot \vect{n}_{aw}= 0 &  \text{ on } \Omega_T\times\Gamma_{aw}\; .
 \end{cases}
 \end{equation}
Considering the structure, linearity, and uniqueness of a solution of  equations  \eqref{coupled_macro} and \eqref{boundary_macro}, the proof of which follows the same lines as for microscopic model \eqref{Stokes},  
we can express $\vect{v}_l$  and $p_{1,l}$, with $l=z, a$ or $sp$,  in the form   
{\small
\begin{eqnarray}\label{ansatz}
\begin{cases}
{\vect v}_z =- \sum\limits_{i=1}^3\partial_{x_i} p\, \,  {\vect w}^i_z  + (c_s-c_a){\vect  r}_z \; ,\quad   
{\vect v}_a =- \sum\limits_{i=1}^3   \partial_{x_i} p\, \,  {\vect w}^i_a  + (c_s-c_a){\vect  r}_a\; , &
{\vect v}_{sp} =- \sum\limits_{i=1}^3\partial_{x_i} p\,  \,    {\vect w}^i_{sp}\; ,  \\
p_{1,z} =- \sum\limits_{i=1}^3 \partial_{x_i} p\, \,  \pi^i_z  + (c_s-c_a)\zeta_z\; ,  \quad 
p_{1,a} =- \sum\limits_{i=1}^3\partial_{x_i} p\, \,  \pi^i_a  + (c_s-c_a)\zeta_a\; , &   p_{1, sp} = -\sum\limits_{i=1}^3\partial_{x_i} p\, \,   \pi^i_{sp}\; , 
\end{cases} 
\end{eqnarray}}
where $\vect{w}_l^i$, $\pi_l^i$, with $l=z,a$ or $sp$, are solutions of the unit cell problems  
\begin{eqnarray*}
\begin{cases}
 -\eta\Delta_y {\vect w}^i_z + \nabla_y \pi^i_z = e_i  & \text{ in } Y_z\; , \\
 \ \ K^{-1}_a {\vect w}^i_a + \nabla_y \pi^i_a = e_i & \text{  in } \, Y_a\; ,\\
 K^{-1}_{sp} {\vect w}^i_{sp} + \nabla_y \pi^i_{sp} = e_i  &  \text{  in } \, Y_{as}\; , \\
  \text{ div}_y \vect{w}^i_z=0 \quad  \text{ in } Y_z\; ,    & \text{ div}_y \vect{w}^i_a=0 \quad \text{ in }  Y_a\; , \quad \text{ div}_y \vect{w}^i_{sp}=0 \quad \text{ in }  Y_{as} \; ,
\end{cases} 
\end{eqnarray*}
with  transmission conditions 
\begin{eqnarray*}
\begin{cases}
-2 \eta (\operatorname{S}_y {\vect w}^i_z\vect{n}) \cdot \vect{n} + \pi_{z}^i - \pi_a^i =\kappa(y) \, {\vect w}^i_a\cdot \vect{n} 
& \text{ on } \Gamma_{zs}\; , \\
  -2 \eta (\operatorname{S}_y {\vect w}^i_z\vect{n}) \cdot \vect{n} + \pi_{z}^i = \pi_{sp}^i  & \text{ on } \Gamma_{as}\; , \\
   \vect{w}_z^i \cdot \vect{n} =( \vect{w}_a^i+ \vect{w}_{sp}^i)\cdot \vect{n}\; , \quad  {\vect w}^i_z\times \vect{n} = 0 &  \text{ on } \Gamma_{as}\; , \\
 \vect{w}_z^i\cdot \vect{ n}= {\vect w}^i_a\cdot \vect{n}\; , \ \ \qquad\qquad {\vect w}^i_z\times \vect{n} = 0 & \text{ on } \Gamma_{z}\; , \\
  \vect{w}_{sp}^i \cdot \vect{n}_{aw} =0   & \text{ on } \Gamma_{aw} \; , 
 \end{cases}
\end{eqnarray*}
and $\vect{r}_l$, $\zeta_l$, with $l=z$ or $a$, solve  the following problem 
\begin{eqnarray*}
\begin{cases}
 -\eta\Delta_y {\vect r}_z + \nabla_y \zeta_z = 0 & \text{ in }  Y_z\; , \\
  K^{-1}_a {\vect r}_a + \nabla_y \zeta_a = 0 & \text{ in }  Y_a\; ,\\
-2  \eta(\operatorname{S}_y {\vect r}_z \vect{n}) \cdot \vect{n} + \zeta_z - \zeta_a = \delta(y) + \kappa(y)\,  {\vect r}_a\cdot \vect{n} & \text{ on } \Gamma_{zs}\; ,\\
 \vect{r}_a\cdot \vect{ n}= {\vect r}_z\cdot \vect{n}\; , \ \;\qquad\qquad\qquad\qquad  \vect{r}_z\times \vect{ n}=  0 &\text{ on } \Gamma_{zs}\; , \\
\text{ div}_y {\vect r}_z  =0 \quad \text{ in }  Y_z\; , \qquad\qquad\qquad  \text{ div}_y {\vect r}_a  =0  & \text{ in }  Y_a \; . 
\end{cases}
\end{eqnarray*}

Averaging the expressions in \eqref{ansatz}  over $Y_z$, $Y_a$ and $Y_{as}$, respectively, using the definition of $\vect{v}$ and  defining the vector $\vect{M}$ and the permeability tensor ${\vect K}=(K_{ij})_{1\leq i,j\leq 3}$ by
\begin{eqnarray}\label{permeab}
\vect{M}= \frac 1 {|Y|} \Big( \int_{Y_z} \vect{r}_z dy +   \int_{Y_a} \vect{r}_a dy \Big)\; ,\quad
K_{ij} = \frac 1 {|Y|}\Big( \int_{Y_z} \vect{w}^i_{z,j} dy +   \int_{Y_a} \vect{w}^i_{a,j} dy + \int_{Y_{as}} \vect{w}^i_{sp,j} dy \Big) \; ,
\end{eqnarray}
we obtain the equation for $\vect{v}$  stated in \eqref{Darcy_macro}. 
\end{proof}

%%%%%%%%%%%%%%%%%%%%%%%%%%%%%%%
%%%%%%%%%%%%%%%%%%%%%%%%%%%%%%%%%%%%%%%%%%%%%%%%%
%%%%%%%%%%%%%%%%%%%%%%%%%%%%%%%%%%%%

For  concentrations of a osmotically active solute and transporters  on the cell membrane we obtain the following macroscopic problem 
\begin{theorem}
The sequence of   solutions  of the microscopic model \eqref{Eq_Conc}-\eqref{init_cond}, \eqref{ExternalBC_N} converges to a unique solution 
$c_l \in H^1(0,T;H^1(\Omega))\cap L^\infty(\Omega_T)$, 
 $\vartheta_{j,l} \in H^1(0,T; L^2(\Omega\times\Gamma_{zs}))\cap L^\infty(\Omega_T\times\Gamma_{zs})$, with $l=a,s$, $j=f,b$,  of the initial boundary values problem in $\Omega_T=(0,T)\times \Omega$
\begin{eqnarray}\label{macro_c_a}
 \begin{cases}
  \partial_t c_a -\operatorname{div}( \mathcal A_{a}\nabla c_a-\hat H_M({\vect{v}}_{a}) \, c_a) = \frac 1{|Y_a|}
 \langle F_a(t,y, c_a), 1\rangle_{Y_a} +
  \frac 1{|Y_a|} \langle\beta_s(t,y) \vartheta_{b,s} - c_a \alpha_a(t,y) \vartheta_{f,a}, 1\rangle_{\Gamma_{zs}}, \\
 (\mathcal A_{a}\nabla c_a - \hat H_M({\vect{v}}_{a})  \, c_a)\cdot \vect{n}=0 \hspace{6.1 cm} \text{ on } \, (0,T)\times\partial\Omega, \\
 c_a(0,x) = c_a^0(x) \hspace{8.4 cm} \text{ in } \Omega , 
\end{cases} \hspace{-1.5 cm } 
\end{eqnarray}
and
\begin{eqnarray}\label{macro_c_s}
\begin{cases}
 \partial_t c_s -\operatorname{div}( \mathcal A_{s}\nabla c_s- \hat H_M({\vect{v}}_{s}) \, c_s) = 
 \frac 1{|Y_s|} \langle F_s(t,y, c_s), 1\rangle_{Y_s}  +   \frac 1{|Y_s|}
\langle
\beta_a(t,y) \vartheta_{b,a} - c_s \alpha_s(t,y) \vartheta_{f,s}, 1\rangle_{\Gamma_{zs}}, \\
 (\mathcal A_{s}\nabla c_s - \hat H_M({\vect{v}}_{s})  \, c_s)\cdot \vect{n}=0 \hspace{6. cm} \text{ on } \, (0,T)\times\partial\Omega\; , \\
 c_s(0,x) = c_s^0(x) \hspace{8.1 cm} \text{ in } \Omega \; .  
\end{cases}  \hspace{-1.5 cm } 
\end{eqnarray}
where homogenized diffusion coefficients $\mathcal A_a$, $\mathcal A_s$ and velocity fields
$\hat H_M({\vect v}_a)$,  $\hat H_M({\vect v}_s)$ are defined by \eqref{macro_diff_velocity} with $\vect{v}_a$ and $\vect{v}_s$ given as solutions of 
\eqref{coupled_macro}, \eqref{boundary_macro}, 
and  transporter concentrations satisfy ordinary differential equations, for $l=a,s$,
\begin{eqnarray}
\begin{cases}
  \partial_t \vartheta_{f,l}  =  R_l(t,y,\vartheta_{f,l})-  \alpha_l(t,y) c_l \vartheta_{f,l} + \beta_l(t,y) \vartheta_{b,l} -\gamma_{f,l}(t,y) \vartheta_{f,l} & \text{ on } \Omega_T\times \Gamma_{zs}\; , \\
  \partial_t \vartheta_{b,l}  = \phantom{ R_l(t,y,\vartheta_{f,l})-  }\, \, \,  \alpha_l (t,y) c_l \vartheta_{f,l} - \beta_l(t,y) \vartheta_{b,l}-\gamma_{b,l}(t,y) \vartheta_{b,l}  & \text{ on } \Omega_T\times \Gamma_{zs}\; ,\\
  \vartheta_{f,l}(0,x,y)=\vartheta_{f,l}^0(x,y),  \,  \vartheta_{b,l}(0,x,y)=\vartheta_{b,l}^0(x,y) & \text{ on } \Omega\times \Gamma_{zs}\; .
  \end{cases} \label{macro_transporter_theta}
\end{eqnarray}

\end{theorem}

\begin{proof} 
To derive  macroscopic equations we shall apply two-scale and   strong convergences stated in Lemmata \ref{convergence_water} and \ref{convergence_concentration}. 
The Lipschitz continuity  of $F_l$ and the strong $L^2$-convergence  of $c_l^\ve$ imply
  $F^\ve_l(t,x, c_l^\ve) \to F_l(t,y, c_l)$ two-scale in $\Omega_T\times Y_l$, with $l=a,s$.
Taking  $\varphi(t,x)= \psi_1(t,x)+ \ve\psi_2(t,x, x/\ve)$ with $\psi_1\in L^2(0,T; H^1(\Omega))$ and $\psi_2 \in C^\infty_0(\Omega_T; C^\infty_{per}(Y))$  as a  test function in \eqref{def_apop} and \eqref{def_symp}, applying two-scale convergences and  strong convergence in $L^2(\Omega_T)$ of $c_l^\ve$, together with  two-scale convergences for $\vect{v}_l^\ve$, 
 and    $\vartheta_{j,l}^\ve$, where $j=f,b$ and  $l=a,s$, and  the  linearity of $H_M$,   considering appropriate subsequences and  passing to the limit as $\ve\to 0$,  imply
\begin{equation}\label{macro_a_0}
\begin{aligned}
  |Y_l|\langle\partial_t c_l, \psi_1 \rangle_{\Omega_T}
 + \langle D_l(t,y) (\nabla c_l + \nabla_y c^{1}_l) - H_M({\vect v}_{l}(t, x,y)) c_l, \nabla \psi_1 + \nabla_y \psi_2 \rangle_{\Omega_T\times Y_l}\\
 = \langle \beta_{l-1}(t,y) \vartheta_{b,l-1}
 - \alpha_l(t,y) c_{l} \, \vartheta_{f,l} , \psi_1 \rangle_{\Omega_T\times \Gamma_{zs}}  + \langle F_l(t,y, c_l),  \psi_1 \rangle_{\Omega_T\times Y_l} \; .
\end{aligned}
 \end{equation}
where $l=a,s$, with $a-1:=s$ and $s-1:=a$. 
Choosing now  $\psi_1=0$ in   \eqref{macro_a_0} 
we obtain 
\begin{eqnarray*}
 &&\langle  D_l(t,y) ( \nabla c_l + \nabla_y c_l^1)-  H_M({\vect v}_{l}) \,  c_l,  \nabla_y \psi_2\rangle_{\Omega\times Y_l,T}   =0\; , \qquad \text{ with } \, l=a,s\; ,
\end{eqnarray*}
for all $\psi_2 \in C^\infty_0(\Omega_T; C^\infty_{\text{per}}(Y))$  and can conclude that $c_l^1$ depends linearly on $\nabla c_l$ and $c_l$.
Thus we can consider an Ansatz
\begin{eqnarray*}
c_l^1(t,x,y)=\sum_{i=1}^3\frac{\partial c_l(t,x)}{\partial{x_i}} \omega^i_l(y) + c_l(t,x) z_l(t,x,y)\; ,
\end{eqnarray*}
where $l=a,s$,  and $\omega^i_l$  and $z_l$ are solutions of the cell problems 
\begin{eqnarray}\label{unit_diff}
\begin{cases}
 -\nabla_y\cdot (D_l(t,y) \nabla_y \omega^i_l) = \sum_{k=1}^3 \partial_{y_k} D_{l, ki} (t,y) &
 \text{ in } (0,T)\times Y_l\; , \\
\quad \, \, -D_l(t,y) \nabla_y \omega^i_l \cdot \vect{n} = \sum_{k=1}^3 D_{l,ki}(t,y)  \vect{n}_k  & \text{ in } (0,T)\times  (\partial Y_l\setminus\partial Y)\; , \\
\quad\qquad \omega_l^i  \   \text{ and} \  D_l(t,y) \nabla_y \omega^i_l \cdot \vect{n}|_{\partial Y_l \cap \partial Y} &  \text{ are }  \, Y-\text{periodic}\; ,
\end{cases}
\end{eqnarray}
and 
\begin{eqnarray}\label{unit_veloc}
\begin{cases}
 \nabla_y\cdot \big(D_l(t,y) \nabla_y z_l- H_M(\vect{v}_{l})\big) = \, \, 0 & \text{ in } \Omega_T\times Y_l,  \\ 
\,\ \big(D_l(t,y) \nabla_y z_l - H_M(\vect{v}_{l})\big) \cdot \vect{n} = \, \, 0  & \text{ on }  \Omega_T\times (\partial Y_l\setminus \partial Y)\; ,\\
\qquad z_l \text{ and }  D_l(t,y) \nabla_y z_l \cdot \vect{n}|_{\partial Y_l \cap \partial Y} & \text{ are } \,  Y-\text{periodic}\; .
\end{cases}
\end{eqnarray}
Next, setting  $\psi_2=0$  in \eqref{macro_a_0}  yields  macroscopic equations \eqref{macro_c_a}-\eqref{macro_c_s} for $c_s$ and $c_a$, where  homogenized diffusion matrices and macroscopic velocity fields are given by 
\begin{equation}\label{macro_diff_velocity}
\begin{aligned}
{\mathcal A}_{l, ij}(t)&= \frac 1 {|Y_l|}\int_{Y_l} \bigl( D_{l, ij}(t,y)+ \sum\limits_{k=1}^3 D_{l, ik}(t,y) \partial_{y_k} \omega_l^j(y)\bigr) dy \, \qquad \text{ for } l=a,s\; , \\
\hat H_M({\vect{v}}_{l} )_i(t,x)&=  \frac 1 {|Y_l|} \int_{Y_l} H_M({\vect v}_{l}(t,x, y))_i\,  dy -
 \frac 1 {|Y_l|} \int_{Y_l}\sum\limits_{k=1}^3 D_{l,ik}(t,y)\partial_{y_k} z_l(t,x,y)\,  dy \; , %& \text{ for } l=a,s\; ,
\end{aligned}
 \end{equation}
 whereas $w_l^j$ and $z_l$ are solutions of the unit cell problems \eqref{unit_diff} and \eqref{unit_veloc}.
 \\

Applying the boundary unfolding operator $\mathcal T^\ve_\Gamma$ to   \eqref{transporter_eq} and testing with $\psi \in L^2(\Omega_T\times\Gamma_{zs})$  give
\begin{align}\label{unfolded_receptors}
\begin{aligned}
\langle \partial_t \mathcal T^\ve_{\Gamma} (\vartheta_{f,l}^\ve), \psi\rangle_{\Omega\times\Gamma_{zs}, T}=
\langle   R_l(t,y,\mathcal T^\ve_{\Gamma} (\vartheta_{f,l}^\ve)), \psi\rangle_{\Omega\times\Gamma_{zs}, T}- 
\langle  \alpha_l(t,y) \mathcal T^\ve_{\Gamma}(\vartheta_{f,l}^\ve) \mathcal T^\ve_{\Gamma} (c^\ve_{l}), \psi\rangle_{\Omega\times\Gamma_{zs}, T}\\
 + \langle \beta_l(t,y) \mathcal T^\ve_{\Gamma} (\vartheta_{b,l}^\ve)- \gamma_{f,l}(t,y)\mathcal T^\ve_{\Gamma}( \vartheta_{f,l}^\ve), \psi\rangle_{\Omega\times\Gamma_{zs}, T}\; ,\\
 \langle \partial_t \mathcal T^\ve_{\Gamma}(\vartheta_{b,l}^\ve), \psi\rangle_{\Omega\times\Gamma_{zs}, T}=
\langle  \alpha_l(t,y) \mathcal T^\ve_{\Gamma}(\vartheta_{f,l}^\ve) \mathcal T^\ve_{\Gamma}(c^\ve_{l})- (\beta_l(t,y) + \gamma_{b,l}(t,y))\mathcal T^\ve_{\Gamma} (\vartheta_{b,l}^\ve), \psi\rangle_{\Omega\times\Gamma_{zs}, T}\; .
\end{aligned}
\end{align}

Considering the strong convergence of $\mathcal T^\ve_{\Gamma}(\vartheta_{f,l}^\ve)$ and $ \mathcal T^\ve_{\Gamma}(c^\ve_{l})$ stated in \eqref{conc-conver}, the equivalence between two-scale convergence and the weak convergence of the unfolded sequence, and  Lipschitz continuity of $R_l$, with $l=a,s$, we can pass in \eqref{unfolded_receptors} to the limit as $\ve\to 0$ and obtain equations  \eqref{macro_transporter_theta}. 
The assumption on the initial data, the similar arguments as in the proof of Lemma~\ref{convergence_concentration}, 
the  two-scale convergence  of $c_l^\ve$,  $\vartheta^\ve_{j,l}$ and of their time derivatives ensure that the initial conditions for $c_l$ and $\vartheta_{j,l}$,  with $l=a,s$ and $j=f,b$,  are satisfied  in  $L^2(\Omega)$ and in $L^2(\Omega\times \Gamma_{zs})$, respectively. The proof of the uniqueness follows along the same lines as for the microscopic model and  implies   the convergence of the entire sequence of  solutions of the microscopic problems.
\end{proof}

\section*{Appendix} We recall here the definition of the two-scale convergence and  unfolding operator. 
\begin{definition}\label{two-scale-conver}\cite{allaire, Nguetseng}
 A sequence $\{u^{\varepsilon}\}\subset L^2((0,T) \times \Omega)$ is said to two-scale converge to a  limit $u_0 \in L^2((0,T)\times \Omega \times Y)$ iff for any $\phi \in L^2((0,T)\times \Omega, C_{per}(Y))$ we have
\[
\lim_{\varepsilon \rightarrow 0}\int_{0}^T\int_{\Omega}u^{\varepsilon}(t,x)\phi(t,x,x/{\varepsilon})dx dt=\int_{0}^T\int_{\Omega}\int_Y u_0(t,x,y)\phi(t,x,y)dy dx dt\; .
\]
\end{definition} 
\begin{theorem}\cite{allaire, Nguetseng}
From each bounded sequence $\{u^{\varepsilon}\}$ in $L^2((0,T)\times \Omega)$ we can extract a subsequence, which two-scale converges to $u_0 \in L^2((0,T)\times \Omega \times Y)$
\end{theorem}
\begin{theorem}\label{compact1}\cite{allaire, Nguetseng}
1. Let $\{u^{\varepsilon}\}$  be a bounded sequence in  $L^2(0,T; H^{1}(\Omega))$, which converges weakly to $u \in L^2(0,T; H^{1}(\Omega))$. Then, there exists $u_1 \in L^2((0,T)\times\Omega; H^{1}_{per}(Y)/\mathbb R)$ such that, up to a subsequence, $u^{\varepsilon}$ two-scale converges to $u$ and $\nabla u^{\varepsilon}$ two-scale converges to $\nabla u+\nabla_y u_1$.\\
\indent 2. Let $\{u^{\varepsilon}\}$ and $\{{\varepsilon} \nabla u^{\varepsilon}\}$  be  bounded sequences in  $L^2((0,T)\times \Omega)$. Then, there exists $u_0 \in L^2((0,T)\times \Omega, H^{1}_{per}(Y)/\mathbb R)$ such that, up to a subsequence, $u^{\varepsilon}$ and ${\varepsilon} \nabla u^{\varepsilon}$ two-scale converge to $u_0$ and $\nabla_y u_0$, respectively.
\end{theorem}
\begin{definition} \label{two-scale-boundary}\cite{neuss-radu}
A sequence  $\{w^{\varepsilon}\}\subset L^2((0,T) \times {\Gamma}^{\varepsilon}_{zs})$ is said to two-scale converge to a  limit $w \in L^2((0,T)\times\Omega \times \Gamma_{zs})$ iff for every $\psi \in L^2((0,T) \times \Omega; C_{per}(\Gamma_{zs}))$ we have
\[
\lim_{\varepsilon \rightarrow 0} \, \varepsilon \int_0^T\int_{{\Gamma}^{\varepsilon}_{zs}}w^{\varepsilon}(t,x)\psi(t,x, x/\varepsilon)d\gamma_x dt=\int_{0}^T\int_{\Omega}\int_{\Gamma_{zs}} w(t,x,y)\psi(t,x,y) d\gamma_y  dxdt\; .
\]
\end{definition}
\begin{theorem}\label{th_comp_b}
1. \cite{neuss-radu} For each  sequence $\{w^{\varepsilon}\}\subset L^2((0,T) \times {\Gamma}^{\varepsilon}_{zs})$
with $\ve^{\frac 12}\|w^{\varepsilon}\|_{L^2((0,T) \times {\Gamma}^{\varepsilon}_{zs})}$ bounded  uniformly in $\ve$,
there exists a   subsequence and  $w \in L^2((0,T)\times \Omega \times \Gamma_{zs})$ such that the subsequence  two-scale converges to $w$.\\
 2. \cite{AnnaMariya2008} If $\{w^{\varepsilon}\}$ is bounded in $L^{\infty}((0,T) \times \Gamma^{\varepsilon}_{zs})$, then the limit $w\in L^{\infty}((0,T)\times \Omega \times \Gamma_{zs})$.
\end{theorem}
\begin{definition}\label{def_unfolding} \cite{Cioranescu1}
 1. For any function $\phi$ Lebesgue-measurable on the perforated domain $\Omega_l^\varepsilon$,
 the unfolding operator
$\mathcal T^\varepsilon_{Y_l}: \Omega^\varepsilon_l \to \Omega\times Y_l$, $l=a,s$, is defined by
$$
\mathcal T^\varepsilon_{Y_l}(\phi)(x,y)=
\begin{cases}
\phi(\varepsilon \big[\frac{x}{\varepsilon}\big]_Y+\varepsilon y) &
\text{ a.e. for } \; y\in Y_l, \,
  x\in {\tilde \Omega^\varepsilon_{int}}\; , \\
0  & \text{ a.e. for } \, y\in Y_l, \, x \in \Omega \setminus \tilde \Omega^\varepsilon_{int}\; ,
\end{cases}
$$
where $k:=[\frac{x}{\varepsilon}]$ denotes the unique integer combination, such that $x-[\frac{x}{\varepsilon}]$ belongs to $Y_l$,
and $\tilde \Omega_{int}^\ve = \text{Int}(\cup_{k\in \mathbb Z^3}\{\ve \overline{Y^k}, \ve Y^k \subset \Omega\})$. We note that for $w\in H^1(\Omega)$ it holds that
$\mathcal{T}_{Y_l}^\varepsilon(w|_{\Omega_l^\varepsilon}) = \mathcal{T}^\varepsilon_Y(w)|_{\Omega\times Y_l}$.\\
2.  For any function $\phi$ Lebesgue-measurable on oscillating boundary $\Gamma^\varepsilon_l$, $l=zs, a, s$,
the boundary unfolding operator
$\mathcal T^\varepsilon_{\Gamma_l}: \Gamma^\varepsilon_l \to \Omega\times \Gamma_l$,  is defined by
$$\mathcal T^\varepsilon_{\Gamma_l}(\phi)(x,y)=
\begin{cases}
\phi(\varepsilon \big [\frac{x}{\varepsilon}\big]_Y+\varepsilon y) & \text{ a.e. for }
 \; y\in \Gamma_l, \, x\in \tilde \Omega^\varepsilon_{int}\; ,  \\
0  & \text{ a.e. for } \, y\in \Gamma_l, \, x \in \Omega \setminus \tilde \Omega^\varepsilon_{int}\; .
\end{cases}
 $$
\end{definition}
\begin{lemma}\label{two-scale-unfold}\cite{AnnaMariya2008}
If   $\{\psi^\ve\}\subset L^2((0,T)\times \Gamma_{zs}^\ve)$ converges two-scale to $\psi \in L^2((0,T)\times\Omega\times \Gamma_{zs})$
and $\{\mathcal T_{\Gamma_{zs}}^\ve (\psi^\ve)\}\subset L^2((0,T)\times\Omega\times \Gamma_{zs})$ converges weakly to $\psi^\ast$ in $L^2((0,T)\times\Omega\times \Gamma_{zs})$, 
then $\psi=\psi^\ast$ a.e. in $(0,T)\times\Omega\times \Gamma_{zs}$.
\end{lemma}

{\it Acknowledgement}
 The first author was funded by the German Research Foundation [grant number CH 958/1-1].

\end{document}